\documentclass[a4paper,reqno,11pt]{amsart}
\usepackage{amsmath, amsfonts, amssymb, amsthm, amscd}
\usepackage{graphicx}
\usepackage{psfrag}
\usepackage{perpage}
\usepackage{url}
\usepackage{color}
\usepackage{mathrsfs}
\usepackage{mathabx}
\usepackage{scalerel} 
\usepackage{tikz,tikz-cd}\usepackage{caption,subcaption}
\usetikzlibrary{arrows,decorations.pathmorphing,backgrounds,positioning,fit,petri}
\usepackage{subdepth}

\usepackage{dsfont} 

\usepackage[utf8]{inputenc}
\usepackage[T1]{fontenc}
\usepackage{microtype}
\usepackage{mathtools}
\usepackage[a4paper,scale={0.77,0.74},marginratio={1:1},footskip=7mm,headsep=10mm]{geometry}

\usepackage[linktocpage]{hyperref}
\definecolor{dukeblue}{rgb}{0.0, 0.0, 0.61}

\hypersetup{
	colorlinks=true, linkcolor=dukeblue,
	citecolor=dukeblue
}

\makeatletter
\def\@secnumfont{\bfseries}

\def\section{\@startsection{section}{1}%
  \z@{.7\linespacing\@plus\linespacing}{.5\linespacing}%
  {\normalfont\large\bfseries\centering}}

\def\subsection{\@startsection{subsection}{2}%
  \z@{.5\linespacing\@plus.7\linespacing}{-.5em}%
  {\normalfont\bfseries}}

\def\subsubsection{\@startsection{subsubsection}{3}%
  \z@{.5\linespacing\@plus.7\linespacing}{-.5em}%
  {\normalfont}}

\def\specialsection{\@startsection{section}{1}%
  \z@{\linespacing\@plus\linespacing}{.5\linespacing}%
  {\normalfont\centering\large\bfseries}}
\makeatother

\makeatletter

\renewenvironment{proof}[1][\proofname]{\par
\pushQED{\qed}%
\normalfont \topsep4\p@\@plus4\p@\relax
\trivlist
\item[\hskip\labelsep
\bfseries
#1\@addpunct{.}]\ignorespaces
}{%
\popQED\endtrivlist\@endpefalse
}
\makeatother

\makeatletter
\newcommand \Dotfill {\leavevmode \leaders \hb@xt@ 6pt{\hss .\hss }\hfill \kern \z@}
\makeatother

\makeatletter
\def\@tocline#1#2#3#4#5#6#7{\relax
  \ifnum #1>\c@tocdepth 
  \else
    \par \addpenalty\@secpenalty\addvspace{#2}%
    \begingroup \hyphenpenalty\@M
    \@ifempty{#4}{%
      \@tempdima\csname r@tocindent\number#1\endcsname\relax
    }{%
      \@tempdima#4\relax
    }%
    \parindent\z@ \leftskip#3\relax \advance\leftskip\@tempdima\relax
    \rightskip\@pnumwidth plus4em \parfillskip-\@pnumwidth
    #5\leavevmode\hskip-\@tempdima
      \ifcase #1
       \or\or \hskip 1.65em \or \hskip 3.3em \else \hskip 4.95em \fi%
      #6\nobreak\relax
    \Dotfill
    \hbox to\@pnumwidth{\@tocpagenum{#7}}\par
    \nobreak
    \endgroup
  \fi}
\makeatother

\makeatletter
\def\l@section{\@tocline{1}{0pt}{1pc}{}{}}
\renewcommand{\tocsection}[3]{%
\indentlabel{\@ifnotempty{#2}{\ignorespaces#1 #2.\hskip 0.7em}}#3}
\def\l@subsection{\@tocline{2}{0pt}{1pc}{5pc}{}}

\def\l@subsubsection{\@tocline{3}{0pt}{1pc}{7pc}{}}

\makeatother

\numberwithin{equation}{section}

\newtheoremstyle{mytheorem}{.7\linespacing\@plus.3\linespacing}{.7\linespacing\@plus.3\linespacing}%
     {\itshape}
     {}
     {\bfseries}
     {. }
     {0.3ex}
     {\thmname{{\bfseries #1}}\thmnumber{ {\bfseries #2}}\thmnote{ (#3)}}  

\theoremstyle{mytheorem}

\newtheorem{theorem}{Theorem}[section]
\newtheorem{lemma}[theorem]{Lemma}
\newtheorem{proposition}[theorem]{Proposition}

\newcommand{\bbE}{{\ensuremath{\mathbb E}} }

\newcommand\overlap{
\begin{tikzpicture}[scale=0.2]
\foreach \i in {0,1,...,4}{
	\draw [fill] (2*\i,0)--(2*\i,0) circle [radius=0.2];
	}
\draw [thick] (0,0)  to [out=45,in=135]  (2,0) to [out=45,in=135]  (4,0) to [out=45,in=135]  (6,0) to [out=45,in=135]  (8,0) ;
\draw [thick, red] (0,0)  to [out=-45,in=-135]  (2,0) to [out=-45,in=-135]  (4,0)  to [out=-45,in=-135]  (6,0) to [out=-45,in=-135]  (8,0);
\end{tikzpicture}
}

\newcommand\overlapbr{
\begin{tikzpicture}[scale=0.2]
\foreach \i in {0,1,...,4}{
	\draw [fill] (2*\i,0)--(2*\i,0) circle [radius=0.2];
	}
\draw [thick, blue] (0,0)  to [out=45,in=135]  (2,0) to [out=45,in=135]  (4,0) to [out=45,in=135]  (6,0) to [out=45,in=135]  (8,0) ;
\draw [thick, red] (0,0)  to [out=-45,in=-135]  (2,0) to [out=-45,in=-135]  (4,0)  to [out=-45,in=-135]  (6,0) to [out=-45,in=-135]  (8,0);
\end{tikzpicture}
}

\newcommand\wiggle{
\begin{tikzpicture}[scale=0.2]
 \draw [fill] (4.2, 5.5)  circle [radius=0.08];  \draw [fill] (5.8, 5.6)  circle [radius=0.08];
\draw[thick, dashed] (4.2, 5.5) to [out=20,in=180] (4.4,5.8) to [out=-30,in=120] (4.8,5.2) to [out=0, in=180] (5.8, 5.6);
\end{tikzpicture}
}

\newcommand{\cA}{{\ensuremath{\mathcal A}} }

\newcommand{\cI}{{\ensuremath{\mathcal I}} }

\DeclareMathSymbol{\leqslant}{\mathalpha}{AMSa}{"36} 
\DeclareMathSymbol{\geqslant}{\mathalpha}{AMSa}{"3E} 
\DeclareMathSymbol{\eset}{\mathalpha}{AMSb}{"3F}     

\newcommand{\sumtwo}[2]{\sum_{\substack{#1 \\ #2}}} 
\newcommand{\sumthree}[3]{\sum_{\substack{#1 \\ #2 \\ #3}}} 
\newcommand{\prodtwo}[2]{\prod_{\substack{#1 \\ #2}}}     

\newcommand{\R}{\mathbb{R}}

\newcommand{\Z}{\mathbb{Z}}
\newcommand{\N}{\mathbb{N}}

\newcommand{\PEfont}{\mathrm}

\newcommand{\p}{\ensuremath{\PEfont P}}

\newcommand{\E}{\ensuremath{\PEfont  E}}
\renewcommand{\P}{\p}

\newcommand{\ind}{\mathds{1}}

\renewcommand{\epsilon}{\varepsilon}
\renewcommand{\theta}{\vartheta}
\renewcommand{\rho}{\varrho}

\newenvironment{myenumerate}{%
\renewcommand{\theenumi}{\arabic{enumi}}%
\renewcommand{\labelenumi}{{\rm(\theenumi)}}%
\begin{list}{\labelenumi}
	{%
	\setlength{\itemsep}{0.4em}%
	\setlength{\topsep}{0.5em}%
	\setlength\leftmargin{2.45em}%
	\setlength\labelwidth{2.05em}%
	\setlength{\labelsep}{0.4em}%
	\usecounter{enumi}%
	}%
	}%
{\end{list}
}

{\end{list}
}

{\end{list}
}

\renewenvironment{enumerate}{
\begin{myenumerate}}%
{\end{myenumerate}}

\newenvironment{myitemize}{%
\begin{list}{$\bullet$}%
 	{%
	\setlength{\itemsep}{0.4em}%
	\setlength{\topsep}{0.5em}%
	\setlength\leftmargin{2.65em}%
	\setlength\labelwidth{2.65em}%
	\setlength{\labelsep}{0.4em}%
	}%
	}%
{\end{list}}

{\end{myitemize}}

\MakePerPage[2]{footnote} 

\date{\today}

\newcommand\sfC{\mathsf C}

\newcommand\sfL{\mathsf L}

\newcommand\sfP{\mathsf P}
\newcommand\sfQ{\mathsf Q}

\newcommand\sfT{\mathsf T}
\newcommand\sfU{\mathsf U}

\newcommand\sfa{\mathsf a}

\newcommand\sfp{\mathsf p}

\newcommand\bx{\boldsymbol{x}}
\newcommand\by{\boldsymbol{y}}
\newcommand\bz{\boldsymbol{z}}

\definecolor{cadmiumgreen}{rgb}{0.0, 0.42, 0.24}
\definecolor{red(munsell)}{rgb}{0.95, 0.0, 0.24}

\newtheorem{innercustomthm}{Theorem}
\newenvironment{customthm}[1]
  {\renewcommand\theinnercustomthm{#1}\innercustomthm}
  {\endinnercustomthm}

\newcommand{\norm}[1]{\left\lVert#1\right\rVert} 
\newcommand{\ms}{\scriptscriptstyle}
\usepackage{mathtools}
\DeclarePairedDelimiter{\floor}{\lfloor}{\rfloor}

\author{Dimitris Lygkonis, Nikos Zygouras}
\address{Department of Mathematics,
    University of Warwick,
    Coventry CV4 7AL, UK}
\email{dimitris.ligonis@gmail.com, n.zygouras@warwick.ac.uk}

\title{A multivariate extension of the Erd\H{o}s-Taylor theorem }

\begin{document}
\begin{abstract}
  The Erd\H{o}s-Taylor theorem [Acta Math. Acad. Sci. Hungar, 1960] 
  states that if $\sfL_N$ is the local time at zero, up to time $2N$, of a two-dimensional simple,
symmetric   random walk, then
  $\tfrac{\pi}{\log N} \,\sfL_N$ converges in distribution to an exponential random variable with parameter one. This can be equivalently stated in
  terms of the total collision time of two independent simple random walks on the plane. More precisely, if
  $\sfL_N^{(1,2)}=\sum_{n=1}^N \ind_{\{S_n^{(1)}= S_n^{(2)}\}}$, then  $\tfrac{\pi}{\log N}\, \sfL^{(1,2)}_N$ converges in distribution to an exponential random variable
  of parameter one.
  We prove that for every $h \geq 3$, the family $ \big\{ \frac{\pi}{\log N}  \,\sfL_N^{(i,j)} \big\}_{1\leq i<j\leq h}$, of
  logarithmically rescaled, two-body collision local times between $h$ independent simple, symmetric
  random walks on the plane converges jointly to a vector of independent exponential random variables with parameter one, thus providing a multivariate
  version of the Erd\H{o}s-Taylor theorem. We also discuss connections to directed polymers in random environments. 
\end{abstract}

\keywords{planar random walk collisions, Erd\H{o}s-Taylor theorem, Schr\"odinger operators with point interactions, directed polymer in random environment}
\subjclass[2010]{82B44, 60G50, 60H15, 82D60, 47D08}
\maketitle

\tableofcontents
\section{Introduction.}
Let $S^{(1)},\dots,S^{(h)}$ be independent, simple, symmetric random walks on $\Z^2$ starting at the origin. We will use $\P_x$ and $\E_x$ to denote the probability and expectation with respect to the law of the simple random walk when starting from $x \in \Z^2$ and we will omit the subscripts when the walk starts from $0$. For $1\leq i <j \leq h$ we define the {\em collision local time} between $S^{(i)}$ and $S^{(j)}$ up to time $N$ by
\begin{equation*} 
  \sfL_N^{(i,j)}:=\sum_{n=1}^N \ind_{\{S_n^{(i)}=S_n^{(j)}\}} \, .
\end{equation*}
Notice that given $1\leq i<j\leq h$, $\sfL_N^{(i,j)}$ has the same law as the number of returns to zero,  before time $2N$,
for a single simple, symmetric random walk $S$ on $\Z^2$, that is $ \sfL_N^{(i,j)}\stackrel{\text{law}}{=}\sfL_N:=\sum_{n=1}^N \ind_{\{S_{2n}=0\}}$. This equality is a consequence of the
independence of $S^{(i)}$, $S^{(j)}$ and the symmetry of the simple random walk. A first moment calculation shows that
\begin{equation} \label{first_moment_normalisation}
  R_N := \E\big[\sfL_N\big] = \sum_{n=1}^N \P(S_{2n}=0) \stackrel{N \to \infty}{\approx} \frac{\log N}{\pi} \, ,
\end{equation}
see Section \ref{Preliminaries} for more details. It was established by Erd\H{o}s and Taylor, $60$ years ago \cite{ET60}, that under normalisation \eqref{first_moment_normalisation},  $\sfL_N$ satisfies the following limit theorem.
\begin{customthm}{A}[\cite{ET60}] \label{ET60_theorem}
  Let $\sfL_N:=\sum_{n=1}^N \ind_{\{S_{2n}=0\}}$ be the local time at zero, up to time $2N$, of a two-dimensional, simple, symmetric random walk $(S_n)_{n\geq 1}$
  starting at $0$. Then, as $N \to \infty$,
  \begin{equation*}
    \frac{\pi}{\log N}\, \sfL_N \xrightarrow{(d)} Y \, ,
  \end{equation*}
  where $Y$ has an exponential distribution with parameter $1$.
\end{customthm}

Theorem \ref{ET60_theorem} was recently generalised in \cite{LZ21}. In particular,
\begin{customthm}{B}[\cite{LZ21}] \label{LZ21_theorem}
  Let $h \in \N$ with $h \geq 2$ and $S^{(1)},...,S^{(h)}$ be $h$ independent two-dimensional, simple random walks starting all at zero.
  Then, for $\beta$ such that $|\beta| \in (0,1)$, it holds that the total collision time $\sum_{1\leq i<j\leq h}   \sfL_N^{(i,j)}$ satisfies
  \begin{equation*} 
    \E^{\otimes h}\bigg[ e^{ \frac{\pi\, \beta}{\log N}  \sum_{1\leq i<j\leq h}   \sfL_N^{(i,j)}} \bigg]
    \xrightarrow[N \to \infty ]{}
    \bigg(\frac{1}{1-\beta}\bigg)^{\frac{h(h-1)}{2}} \, ,
  \end{equation*}
  and, consequently,
  \begin{equation*}
    \frac{\pi}{\log N} \sum_{1\leq i<j\leq h}   \sfL_N^{(i,j)} \xrightarrow[N \to \infty]{(d)} \Gamma\big(\tfrac{h(h-1)}{2},1\big) \, ,
  \end{equation*}
  where $ \Gamma\big(\tfrac{h(h-1)}{2},1\big)$ denotes a Gamma variable, which has a density $\Gamma(h(h-1)/2)^{-1} x^{\tfrac{h(h-1)}{2}-1} e^{-x}$;
  $\Gamma(\cdot)$, in the expression of the density,  denotes the Gamma function.
\end{customthm}
Given the fact that a gamma distribution $\Gamma(k,1)$, with parameter $k\geq 1$, arises as the distribution of the sum of $k$ independent
random variables each one distributed according to an exponential random variable with parameter one (denoted as ${\rm Exp}(1)$), Theorem \ref{LZ21_theorem} raises the
question as to whether the joint distribution of the individual rescaled collision times
$\Big\{\frac{\pi}{\log N} \,  \sfL_N^{(i,j)}\Big\}_{1\leq i<j \leq h}$ converges to that of a
family of independent $\text{Exp}(1)$ random variables.
This is what we prove in this work. In particular,
\begin{theorem} \label{main_result}
  Let $h \in \N$ with $h\geq 2$ and $\boldsymbol{\beta}:=\{\beta_{i,j}\}_{1\leq i<j \leq h} \in \R^{\frac{h(h-1)}{2}}$ with 
  $|\beta_{i,j}| <1$ for all $1\leq i<j\leq h$. Then we have that
  \begin{equation}\label{thmeq:laplace2}
    \E^{\otimes h}\bigg[ e^{ \frac{\pi}{\log N}  \sum_{1\leq i<j\leq h}  \beta_{i,j} \sfL_N^{(i,j)}} \bigg]
    \xrightarrow[N \to \infty ]{}
    \prod_{1\leq i<j \leq h} \frac{1}{1-\beta_{i,j}} \,
  \end{equation}
  and, consequently,
  \begin{equation}\label{thmeq:joint}
    \Big\{ \tfrac{\pi}{\log N}  \,\sfL_N^{(i,j)} \Big\}_{1\leq i<j\leq h} \xrightarrow[N \to \infty]{(d)} \big\{ Y^{(i,j)} \big\}_{1\leq i<j\leq h} \, ,
  \end{equation}
  where $\big\{ Y^{(i,j)} \big\}_{1\leq i<j\leq h}$ are independent and identically distributed random variables following an $\text{Exp}(1)$ distribution.
\end{theorem}
An intuitive way to understand the convergence of the individual collision times, or equivalently of the local time of a planar walk, to
an exponential variable is the following. By \eqref{first_moment_normalisation}, the number of visits to zero of a planar walk,
which starts at zero, is $O(\log N)$ and, thus, much smaller than the time horizon $2N$. Typically, also, these visits happen within a short time,
much smaller than $2N$, so that every time the random walk is back at zero, the probability that it will return there again before time $2N$ is
not essentially altered. This results in the local time $\sfL_N$ being asymptotic to a geometric random variable with parameter
of order $(\log N)^{-1}$ (as also manifested by \eqref{first_moment_normalisation}),
which when rescaled suitably converges to an exponential random variable.

The fact that the joint distribution of $ \Big\{ \tfrac{\pi}{\log N}  \,\sfL_N^{(i,j)} \Big\}_{1\leq i<j\leq h} $ converges to that
of {\it independent } exponentials is much less apparent as the collision times have obvious correlations. A way to understand this
is, again, through the fact that collisions happen at time scales much shorter than the time horizon $N$ and, thus, it is
not really possible to actually distinguish which pairs collide when collisions take place. 
More crucially, the logarithmic scaling, as indicated via \eqref{first_moment_normalisation},
introduces a separation of scales between collisions of different pairs of walks, which is what, essentially, 
leads to the asymptotic factorisation of the Laplace transform \eqref{thmeq:laplace2}.  
This intuition is reflected in the two main steps of our proof,
which are carried out in Sections \ref{step_3} and \ref{step_5}.

Even though the Erd\H{o}s-Taylor theorem appeared a long time ago, the multivariate extension that we establish here appears to be
new. In \cite{GS09} it was shown that the law of $\tfrac{\pi}{\log N}\, \sfL^{(1,2)}_N$, conditioned on $S^{(1)}$,  converges a.s.
to that of an ${\rm Exp}(1)$ random variable. This implies that $\big\{ \tfrac{\pi}{\log N} \,\sfL^{(1,i)}_N\big\}_{1<i\leq h}$ converge
 to independent exponentials. However, it does not address the full independence of the
family of {\it all} pairwise collisions $\big\{ \tfrac{\pi}{\log N} \,\sfL^{(i,j)}_N\big\}_{1\leq i<j\leq h}$.

 In the continuum, phenomena of independence in functionals of planar Brownian motions have appeared in works around {\it log-scaling laws}
see \cite{PY86} (where the term {\it log-scaling laws} was introduced) as well as \cite{Y91} and \cite{Kn93}. These
works are mostly concerned with the problem of identifying the limiting distribution of {\it windings} of a planar Brownian motion
around a number of points $z_1,...,z_k$, different than the starting point of the Brownian motion, or the winding around the origin
of the differences $B^{(i)}-B^{(j)}$ between $k$ independent Brownian motions $B^{(1)},...,B^{(k)}$, starting all from different points,
which are also different than zero. Without getting into details, we mention that the results of \cite{PY86, Y91, Kn93} establish
that the windings (as well as some other functionals that fall within the class of {\it log-scaling laws})
converge, when logarithmically scaled, to independent Cauchy variables. \cite{Kn93} outlines a proof that the
local times of the differences $B^{(i)}-B^{(j)}, 1\leq i<j\leq k$, on the unit circle $\{z\in\R^2 \colon |z|=1\}$ converge, jointly,
to independent exponentials ${\rm Exp}(1)$, when logarithmically scaled, in a fashion similar to the scaling of Theorem \ref{main_result}.
The methods employed in the above works rely heavily on continuous techniques (It\^o calculus, time changes etc.), which
do not have discrete counterparts. In fact,  the passage from continuous to discrete is not straightforward either at a
technical level (see e.g. the discussion on page 41 of \cite{Kn93} and \cite{Kn94}) or at a phenomenological level
(see e.g. discussion on page 736 of \cite{PY86}).

The approach we follow towards Theorem \ref{main_result} starts with
expanding the joint Laplace transform in the form of {\it chaos series}, which  take the form of Feynman-type diagrams.
To control (and simplify) these diagrams, we start by inputing a renewal representation as well as a functional analytic framework.
The renewal theoretic framework was originally introduced in \cite{CSZ19a} in the context of scaling limits of random polymers
(we will come back to the connection with polymers later on) and it captures the stream of collisions within a single pair of walks. The
functional analytic framework can be traced back to works on spectral theory of {\it delta-Bose gases} \cite{DFT94, DR04}
and was also recently used in works on random polymers \cite{GQT21, CSZ21, LZ21}.
The core of this framework is to establish operator norm bounds for the total Green's functions of a set of planar random walks conditioned 
on a subset of them starting at the same location and on another subset of them ending up at the same location. 
Roughly speaking, the significance of these operator estimates is to control the redistribution of collisions when walks switch pairs. 
The operator framework (together with the renewal one) allows to reduce the number of Feynman-type diagrams that need to be considered.
For the reduced Feynman diagrams, one, then,
 needs to look into the logarithmic structure, which induces a {\it separation of scales} and leads to the fact that,
asymptotically, the structure of the Feynman diagrams becomes that of the product of Feynman diagrams corresponding to
Laplace transforms of single pairs of random walks. 
\vskip 2mm
{\bf Relations to random polymers.}
Exponential moments of collision times arise naturally when one looks at moments of partition functions of 
the model of directed polymer in a random environment (DPRE), we refer to \cite{C17} for an introduction to this model.
For a family of i.i.d. variables $\big( \omega_{n,x} \colon n\in\N, x\in \Z^2 \big)$ with log-moment generating function $\lambda(\beta)$,
 the partition function of the directed polymer measure is defined as
\begin{align*} 
	Z_{N,\beta}(x):=\E_{x}\Big[\exp\Big(\sum_{n=1}^{N} \big(\beta \omega_{n,S_n} -\lambda(\beta) \big) \Big)   \Big ] ,
\end{align*}
where $\E_{x}$ is the expected value with respect to a simple, symmetric walk starting at $x\in\Z^2$. 
In the case that
$\omega_{n,x}$ is a standard normal variable, 
an explicit computation, for $\beta_N:=\beta \sqrt{\tfrac{\pi}{ \log N}}$, gives that
\begin{align}\label{h-mom}
\bbE\Big[ \big( Z_{N,\beta_N}(x) \big)^h \Big] =  \E^{\otimes h}\bigg[ e^{ \frac{\pi \beta^2}{\log N}  \sum_{1\leq i<j\leq h}  \sfL_N^{(i,j)}} \bigg].
\end{align}
A corollary of our Theorem \ref{main_result} is that the limit of the $h$-th moment of the DPRE partition function
converges to $(1-\beta^2)^{-h(h-1)/2}$; a result that was previously obtained in \cite{LZ21} combining upper bounds on moments,
established in \cite{LZ21}, with results on the distributional convergence of the partition function established in \cite{CSZ17}. 
Using Theorem \ref{main_result}, we can further extend this to convergence of mixed moments. More precisely, 
for $\beta_{i,N}:=\beta_i  \sqrt{\tfrac{\pi}{\log N}}$ and $i=1,...,h$, we have that
\begin{align}\label{joint}
\bbE\big[  Z_{N,\beta_{1,N}}(x) \cdots Z_{N,\beta_{h,N}}(x)  \big] 
= \E^{\otimes h}\bigg[ e^{ \frac{\pi}{\log N}  \sum_{1\leq i<j\leq h}  \beta_{i}\beta_{j} \sfL_N^{(i,j)}} \bigg]
\xrightarrow[N\to\infty]{} \prod_{1\leq i<j \leq N} \frac{1}{1-\beta_i\beta_j}.
\end{align}
where the equality is again via an explicit computation as in \eqref{h-mom} when the disorder is standard normal and the convergence 
follows from Theorem \ref{main_result} after specialising the parameters $\beta_{i,j}$ to the particular case of $\beta_i\beta_j$.
\footnote{\eqref{joint} could also be achieved if one refined the results of \cite{CSZ17} to joint convergence of 
$Z_{N,\beta_{1,N}}(x), \dots ,Z_{N,\beta_{h,N}}(x)$ and combined with the moment estimates of \cite{LZ21}. However, the independence of
the collision times that we establish here cannot be recovered from joint moments of partition functions as these will only give rise to Laplace transforms with a restricted form of Laplace parameters as $\{\beta_i\beta_j\}_{1\leq i<j\leq h}$ instead of $\{\beta_{i,j}\}_{1\leq i<j\leq h}$, 
needed for the joint convergence of the vector $\{L_N^{(i,j)}\}_{1\leq i <j \leq h}$.
The structural aspects of the collision structure that we unveil in the present work are crucial towards establishing the independence property.}

Moment estimates on polymer partition functions 
are important in establishing fine properties, such as structure of maxima, of the field of partition functions 
$\big\{ \sqrt{\log N } \big( \log Z_{N,\beta}(x) -\bbE [\log Z_{N,\beta}(x) ] \big)\colon x\in \Z^2 \big\}$, which is known to converge to 
a log-correlated gaussian field \cite{CSZ18}. We refer to \cite{CZ21} for more details. We expect the independence 
structure of the collision local times, that we establish here, to be useful towards these investigations.
An interesting problem, in relation to this (but also of broader interest),
 is how large can the number $h$ of random walks be (depending on $N$),
before we start seeing correlations in the limit of the rescaled collisions.
The work of Cosco-Zeitouni \cite{CZ21}
has shown that there exists $\beta_0 \in (0,1)$ such that for all $\beta \in (0,\beta_0)$ and $h=h_N \in \N$ such that
$$
  \limsup_{N \to \infty} \frac{3\beta}{1-\beta} \frac{1}{\log N} \binom{h}{2}<1 \, ,
$$
one has that
\begin{equation*}
  \E^{\otimes h}\bigg[ e^{ \frac{\pi\, \beta}{\log N}  \sum_{1\leq i<j\leq h}   \sfL_N^{(i,j)}} \bigg] \leq c(\beta)\, \Big(\frac{1}{1-\beta}\Big)^{\binom{h}{2} (1+\epsilon_N)} \, ,
\end{equation*}
with $c(\beta) \in (0,\infty)$ and $0\leq \epsilon_N =\epsilon(\beta,N)\downarrow 0$ as $N \to \infty$. 
This suggests that the threshold might be $h=h_N=O(\sqrt{\log N})$. More recent results
\cite{CZ23}, imply that the independence fails  when the number of walks is $ \gg \log N$.
The question of whether there is a critical constant $c_{crit.}$ such that for a number of walks larger that $c_{crit.}\log N$ the independence of the collisions
fails is an open and interesting problem.
\vskip 2mm
{\bf Outline.}
The structure of the article is as follows:
In Section \ref{Preliminaries} we set the framework of the chaos expansion, its graphical representations in
terms of Feynman-type diagrams, as well as the renewal and functional analytic frameworks. In Section \ref{approximation_steps}
we carry out the approximation steps, which lead to our theorem. At the beginning of Section \ref{approximation_steps}
we also provide an outline of the scheme.

We close by mentioning our convention on the constants: whenever a constant depends on specific parameters, we will indicate this at the beginning of the statements but then drop the dependence, while if no dependence on parameters is indicated, then
they will be understood as absolute constants.
\vskip 0.2cm

\section{Chaos expansions and auxiliary results} \label{Preliminaries}
In this section we will introduce the framework, within which we work, and which consists of chaos expansions for the joint Laplace transform
\begin{equation} \label{M_N,h}
  M^{\boldsymbol{\beta}}_{N,h}:=\E^{\otimes h}\bigg[ e^{   \sum_{1\leq i<j\leq h}  \frac{\pi\beta_{i,j}}{\log N} \, \sfL_N^{(i,j)}} \bigg] \, ,
\end{equation}
for a fixed collection of numbers $\boldsymbol{\beta}:=\{\beta_{i,j}\}_{1\leq i<j \leq h} \in \R^{\frac{h(h-1)}{2}}$
with $|\beta_{i,j} | \in (0,1)$ for all $1 \leq i<j \leq h$. We denote by
\begin{equation}\label{barb}
  \bar{\beta}:=\max_{1\leq i<j \leq h} | \beta_{i,j} |<1 \,  ,
\end{equation}
and define
\begin{align}\label{def:sigma}
  \sigma_N^{i,j}:=\sigma_N^{i,j}(\beta_{i,j}):= e^{ \beta^{i,j}_N}-1 \qquad \text{with} \qquad \beta^{i,j}_N:=\frac{\pi \, \beta_{i,j}}{\log N}.
\end{align}
{\bf Convention:} From now on we will be assuming that all parameters $\beta_{i,j}$ are nonnegative. We will return to the general
case at the very end when discussing the proof of Theorem \ref{main_result}.
\vskip 2mm
We will use the notation $q_n(x):=\P(S_n=x)$ for the transition probability of the simple, symmetric random walk.
The expected collision local time between two independent simple, symmetric random walks will be
\begin{equation} \label{R_N}
  R_N:=\E^{\otimes 2}\Big[\sum_{n=1}^N \ind_{S_n^{(1)}=S_n^{(2)}}\Big]=\sum_{n=1}^N q_{2n}(0)
\end{equation}
and by Proposition 3.2 in \cite{CSZ19a} we have that in the two-dimensional setting
\begin{equation} \label{R_N_asymp}
  R_N=\frac{\log N}{\pi}+ \frac{\alpha}{\pi}+ o(1) \, ,
\end{equation}
as $N\to \infty$, with $\alpha=\gamma +\log 16 -\pi \simeq 0.208$ and $\gamma \simeq 0.577$ is the Euler constant.

\subsection{Chaos expansion for two-body collisions and renewal framework.} 
We start with the Laplace transform of the simple case of two-body collisions $ \E \Big[e^{\beta^{i,j}_N\, \sfL^{(i,j)}_N}\Big]$ and deduce its
chaos expansion as follows:
\begin{align}\label{eq:2body}
  \E \Big[e^{\beta^{i,j}_N\, \sfL^{(i,j)}_N}\Big]
   & = \E \bigg[e^{ \beta^{i,j}_N \,
  \sum_{n=1}^N \sum_{x\in \Z^2} \ind_{\{ S_n^{(i)}=x \}} \ind_{\{S_n^{(j)}=x \}}}  \bigg]
  =\E \bigg[ \prod_{\substack{ 1\leq n \leq N                                    \\ x\in \Z^2 }}  \Big( 1+ \Big(e^{ \beta^{i,j}_N}-1 \Big)
  \, \ind_{\{ S_n^{(i)}=x \}} \ind_{\{S_n^{(j)}=x \}} \Big)\bigg] \notag         \\
   & =1+\sum_{k\geq 1} (\sigma_N^{i,j})^k
  \sum_{\substack{1\leq n_1< \dots<n_k\leq N                                     \\ x_1,...,x_k\in\Z^2}} \E \bigg[\prod_{a=1}^k
  \, \ind_{\{ S_{n_a}^{(i)}=x_a \}} \ind_{\{S_{n_a}^{(j)}=x_a \}}  \bigg] \notag \\
   & =1+ \sum_{k\geq 1}  (\sigma_N^{i,j})^k
  \sumtwo{1\leq n_1< \dots<n_k\leq N, \,}{x_1,\dots,x_k \in \Z^2} \prod_{a=1}^k q^2_{n_a-n_{a-1}}(x_a-x_{a-1})
\end{align}
where in the last equality we used the Markov property, in the third we expanded the product
and in the second we used the simple fact that
\begin{align*}
  e^{\beta_N^{i,j} \, \ind_{\{ S_n^{(i)}=S_n^{(j)}=x \}} }
   & = 1 + \Big( e^{\beta_N^{i,j} \, \ind_{\{ S_n^{(i)}= S_n^{(j)}=x\} } }-1 \Big)
  = 1 + \Big( e^{\beta_N^{i,j}}  -1 \Big) \ind_{\{ S_n^{(i)}=S_n^{(j)}=x \}  }     \\
   & = 1 + \sigma_N^{i,j} \,  \ind_{\{ S_n^{(i)}=S_n^{(j)}=x \}  },
\end{align*}
with $\sigma^{i,j}_{N}$ defined in \eqref{def:sigma}.
We will express \eqref{eq:2body} in terms of the following quantity $U^{\beta}_N(n,x)$, which plays an important role in our formulation.
For $\sigma_N:=\sigma_N(\beta):=e^{\frac{\pi\beta}{\log N}}-1$ and $(n,x)\in\N\times \Z^2$, we define
\begin{equation} \label{Un}
  \begin{split}
    U^{\beta}_N(n,x) & :=\sigma_{N}\, q^2_{n}(x)\\
    & + \sum_{k \geq 1} \sigma_{N}^{k+1} \sumtwo{0<n_1<\dots<n_k<n}{z_1,z_2,\dots,z_k \in \Z^2}  q^2_{n_1}(z_1)\Big\{ \prod_{j=2}^k q^2_{n_j-n_{j-1}}(z_j-z_{j-1}) \Big\}\, q^2_{n-n_k}(x-z_k).
  \end{split}
\end{equation}
and  $U_N^{\beta}(n,x):=\ind_{\{x=0\}}$, if $n=0$. Moreover, for $n\in\N$ we define
\begin{equation*}
  U^{\beta}_N(n):=\sum_{x \in \Z^2} U^{\beta}_N(n,x) \, .
\end{equation*}
$U^{\beta}_N(n,x)$ represents the Laplace transform of the two-body collisions, scaled by $\beta$,
between a pair of random walks that are constrained to end at the spacetime point $(n,x) \in \{1,\dots,N\}\times \Z^2$, starting from $(0,0)$.
In particular, for any $1\leq i <j\leq h$, we can write \eqref{eq:2body} as
\begin{align*}
  \E \Big[e^{\beta^{i,j}_N\, \sfL^{(i,j)}_N}\Big] = \sum_{n=0}^N \sum_{x \in \Z^2} U^{\beta_{i,j}}_N(n,x)  =  \sum_{n=0}^N U^{\beta_{i,j}}_N(n).
\end{align*}
We will call $U^{\beta}_N(n,x)$ a {\bf replica} and for $\sigma_N(\beta)=e^{\frac{\pi \beta}{\log N}}-1$
we will graphically represent $\sigma_N(\beta)\, U^{\beta}_N(n,x)$ as
\begin{align*} 
  \sigma_N(\beta) \, U^\beta_N(b-a,y-x)
   & \equiv \quad
  \begin{tikzpicture}[baseline={([yshift=1.3ex]current bounding box.center)},vertex/.style={anchor=base,
          circle,fill=black!25,minimum size=18pt,inner sep=2pt}, scale=0.45]
    \draw  [fill] (0, 0)  circle [radius=0.2];  \draw  [fill] ( 4,0)  circle [radius=0.2];
    \draw [-,thick, decorate, decoration={snake,amplitude=.4mm,segment length=2mm}] (0,0) -- (4,0);
    \node at (0,-1) {\scalebox{0.8}{$(a,x)$}}; \node at (4,-1) {\scalebox{0.8}{$(b,y)$}};
  \end{tikzpicture} \\
   & := \sum_{k\geq 1 }\sumtwo{n_1<\cdots < n_k}{x_1,...,x_k}
  \begin{tikzpicture}[baseline={([yshift=1.3ex]current bounding box.center)},vertex/.style={anchor=base,
          circle,fill=black!25,minimum size=18pt,inner sep=2pt}, scale=0.45]
    \draw  [fill] ( 0,0)  circle [radius=0.2];
    \draw  [fill] ( 2,0)  circle [radius=0.2]; \draw  [fill] (4,0)  circle [radius=0.2];
    \draw  [fill] (8, 0)  circle [radius=0.2]; \draw  [fill] (10, 0)  circle [radius=0.2];
    \draw [thick] (0,0)  to [out=45,in=135]  (2,0) to [out=45,in=135]  (4,0) to [out=45,in=180]  (4.5,0.3);
    \draw [thick] (0,0)  to [out=-45,in=-135]  (2,0) to [out=-45,in=-135]  (4,0)  to [out=-45,in=180]  (4.5,-0.3);
    \draw [thick] (7.5,0.3)  to [out=0,in=135]  (8,0) to [out=45,in=135]  (10,0);
    \draw [thick] (7.5,-0.3) to [out=0,in=-135] (8,0)  to [out=-45,in=-135]  (10,0);
    \node at (0,-1) {\scalebox{0.7}{$(a,x)$}}; \node at (2,-1) {\scalebox{0.7}{$(n_1,x_1)$}};
    \node at (4,-1) {\scalebox{0.7}{$(n_2,x_2)$}}; \node at (8,-1) {\scalebox{0.7}{$(n_k,x_k)$}};
    \node at (10,-1) {\scalebox{0.7}{$(b,y)$}};
    \node at (6,0) {$\cdots$};
  \end{tikzpicture} \notag
\end{align*}
In the second line we have assigned weights $q_{n'-n}(x'-x)$ to the solid lines going from
$(n,x)$ to $(n',x')$ and we have assigned the weight
$\sigma_N(\beta)=e^{\frac{\pi \beta}{\log N}}-1$ to every solid dot.

$U^{\beta}_N(n)$ and $U^{\beta}_N(n,x)$ admit a very useful probabilistic interpretation in terms of certain renewal processes.
More specifically,  consider the family of i.i.d. random variables
$(T^{\ms (N)}_i, X^{\ms (N)}_i)_{i \geq 1}$ with law
\begin{align*} 
  \P\Big(\,\big(T^{\ms (N)}_1, X^{\ms (N)}_1 \big)=(n,x) \Big)=\frac{q_n^2(x)}{R_N}\,\ind_{\{n \leq N\}} \, .
\end{align*}
and $R_N$ defined in \eqref{R_N}.
Define the random variables
$\tau^{\ms (N)}_k:=T_1^{\ms (N)}+\dots+T_k^{\ms (N)}$, $S^{\ms (N)}_k:=X^{\ms (N)}_1+\dots+X^{\ms (N)}_k$, if $k \geq 1$, and
$(\tau_0,S_0):=(0,0)$, if $k=0$.
It is not difficult to see that $U^{\beta}_N(n,x) $ and $U^{\beta}_N(n) $ can, now, be written as
\begin{equation} \label{ren_rep}
  U^{\beta}_N(n,x)  =\sum_{k\geq 0} (\sigma_N R_N)^{k} \, \P\big(\tau^{\ms (N)}_k=n, S^{\ms (N)}_k=x\big) \quad \text{and} \quad
  U^{\beta}_N(n)    =\sum_{k\geq 0}  (\sigma_N R_N)^{k} \, \P\big(\tau^{\ms (N)}_k=n\big)
\end{equation}
This formalism was developed in \cite{CSZ19a} and is very useful in obtaining sharp asymptotic estimates.
In particular, it was shown in \cite{CSZ19a} that the rescaled process
$\Big(\frac{\tau^{(N)}_{\lfloor s\log N\rfloor}}{N}, \frac{S^{(N)}_{\lfloor s\log N\rfloor}}{\sqrt{N}} \Big)$ converges in distribution for $N\to\infty$
with the law of the marginal limiting process for $\tfrac{\tau^{(N)}_{\lfloor s\log N\rfloor}}{N}$ being the {\it Dickman subordinator},
which was defined in \cite{CSZ19a} as a truncated, zero-stable L\'evy process.

An estimate that follows easily from this framework, which is useful for our purposes here, is the following: for $\beta<1$, it holds
\begin{align}\label{Uest}
  \limsup_{N\to\infty}\sum_{n=0}^N U^{\beta}_N(n)
   & =\limsup_{N\to\infty}\sum_{k\geq 0}  (\sigma_N R_N)^{k} \, \P\big(\tau^{\ms (N)}_k \leq N\big) \notag \\
   & \leq \limsup_{N\to\infty} \sum_{k\geq 0}  (\sigma_N R_N)^{k} \notag                                   \\
   & = \limsup_{N\to\infty} \frac{1}{1-\sigma_N R_N} = \frac{1}{1-\beta},
\end{align}
where we used the fact that
\begin{align}\label{sigmaR}
  \sigma_N R_N
  = \big(e^{\frac{\pi \beta}{\log N}}-1\big)\cdot \Big( \frac{\log N}{\pi}+\frac{\alpha}{\pi}+o(1)\Big)
  \xrightarrow[N\to\infty]{} \beta <1.
\end{align}
\subsection{Chaos expansion for many-body collisions.} \label{chaos_expansion_for_many_body_collisions}
We now move to the expansion of the Laplace transform $M^{\boldsymbol{\beta}}_{N,h}$ of the many-body collisions .
The goal is to obtain an expansion in the form of products of certain Markovian operators. The desired expression will be
presented in \eqref{def:Pexpansion}. This expansion will be instrumental in obtaining some important estimates in
Section \ref{Operators}.

The first steps are similar as in the expansion for the two-body collisions, above. In particular, we have
\begin{align}\label{def:expansion}
 M^{\boldsymbol{\beta}}_{N,h} &:= \E^{\otimes h}\bigg[e^{\sum_{1\leq i<j \leq h}\beta^{i,j}_N \,  \sfL^{(i,j)}_{N}}\bigg]
    = \E \bigg[ \prod_{1\leq i<j \leq h} \, \prod_{\substack{ 1\leq n \leq N                                   \\ x\in \Z^2 }}  \Big( 1+\sigma_N^{i,j}
  \, \ind_{\{ S_n^{(i)}=x \}} \ind_{\{S_n^{(j)}=x \}} \Big)\bigg]                                               \\
   & = 1+ \sum_{k\geq 1}  \,\,  \sum_{\substack{ (i_a, j_a, n_a, x_a) \in \cA_h, \,\, \text{for} \,\, a=1,...,k \\ \text{distinct}}}
  \,\, \E \Big[ \prod_{a=1}^k  \sigma_N^{i_a,j_a}\, \ind_{\{ S_{n_a}^{(i_a)} = x_a \}} \,\ind_{ \{S_{n_a}^{(j_a)}=x_a \}}  \Big] \notag
\end{align}
where the last sum is over $k$ distinct elements of the set
\begin{align*}
  \cA_h:=\big\{ (i,j,n,x) \in \N^3\times \Z^2 \colon 1\leq i<j \leq h\big\}.
\end{align*}

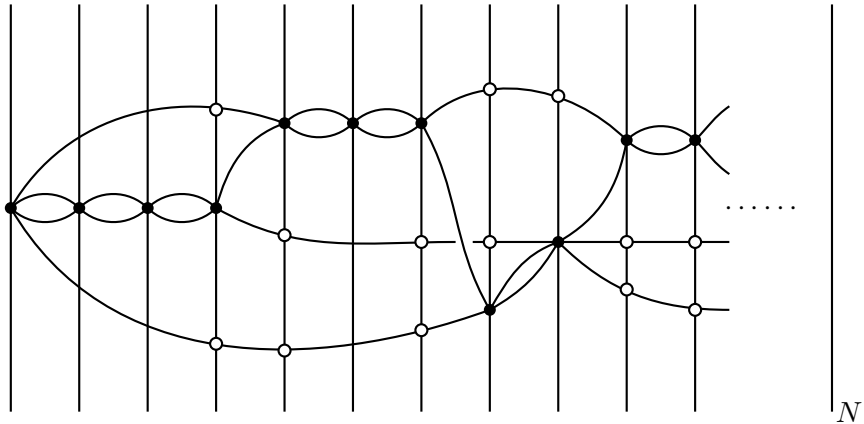
\begin{figure}
  \begin{tikzpicture}[scale=0.45]
    \foreach \i in {0,1,...,10}{
        \draw[-, thick] (2*\i,-6) -- (2*\i, 6);
      }
    \fill[black]  (0,0) circle [radius=0.175]; \fill[black]  (2,0) circle [radius=0.175]; \fill[black]  (4,0) circle [radius=0.175];
    \fill[black]  (6,0) circle [radius=0.175]; \fill[black]  (8,2.5) circle [radius=0.175];
    \fill[black]  (10,2.5) circle [radius=0.175]; \fill[black]  (12,2.5) circle [radius=0.175];
    \fill[black]  (14, -3) circle [radius=0.175]; \fill[black]  (16, -1) circle [radius=0.175];
    \fill[black]  (18, 2) circle [radius=0.175]; \fill[black]  (20, 2) circle [radius=0.175];
    \draw[thick] (0,0) to [out=-45,in=-135] (2,0) to [out=-45,in=-135] (4,0) to [out=-45,in=-135] (6,0);
    \draw[thick] (0,0) to [out=45,in=135] (2,0) to [out=45,in=135] (4,0) to [out=45,in=135] (6,0);
    \draw[thick] (6,0) to [out=75,in=200] (8,2.5) ; \draw[thick] (0,0) to [out=60,in=160] (8,2.5) ;
    \draw[thick] (8,2.5) to [out=45,in=135] (10, 2.5) to [out=45,in=135] (12,2.5);
    \draw[thick] (8,2.5) to [out=-45,in=-135] (10, 2.5) to [out=-45,in=-135] (12,2.5);
    \draw[thick] (0,0) to [out=-60,in=200] (14,-3) ; \draw[thick] (12,2.5) to [out=-60,in=120] (14,-3) ;
    \draw[thick] (6,0) to [out=-30,in=180] (13, -1) ; \draw[thick] (13.5,-1) to [out=0,in=180] (16, -1) ;
    \draw[thick] (14,-3) to [out=30,in=-120] (16, -1) ;
    \draw[thick] (16,-1) to [out=30,in=-100] (18, 2) ; \draw[thick] (12, 2.5) to [out=45,in=135] (18, 2) ;
    \draw[thick] (18, 2) to [out=45,in=135] (20, 2) ;  \draw[thick] (18, 2) to [out=-45,in=-135] (20, 2) ;
    \draw[thick] (16, -1) to [out=-45,in=180] (21, -3) ;  \draw[thick] (14, -3) to [out=60,in=200] (16, -1) ;
    \draw[thick]  (16, -1) to [out=0,in=180] (21, -1) ;
    \draw[thick] (20, 2) to [out=-45,in=145] (21, 1) ;   \draw[thick] (20, 2) to [out=45,in=-145] (21, 3) ;
    \fill[white]  (6, -4) circle [radius=0.175]; \draw[thick]  (6,-4) circle [radius=0.175];
    \fill[white]  (6, 2.9) circle [radius=0.175]; \draw[thick]  (6,2.9) circle [radius=0.175];
    \fill[white]  (8, -4.2) circle [radius=0.175]; \draw[thick]  (8,-4.2) circle [radius=0.175];
    \fill[white]  (8, -0.8) circle [radius=0.175]; \draw[thick]  (8,-0.8) circle [radius=0.175];
    \fill[white]  (12, -1) circle [radius=0.175]; \draw[thick]  (12,-1) circle [radius=0.175];
    \fill[white]  (12, -3.6) circle [radius=0.175]; \draw[thick]  (12,-3.6) circle [radius=0.175];
    \fill[white]  (14, -1) circle [radius=0.175]; \draw[thick]  (14,-1) circle [radius=0.175];
    \fill[white]  (14, 3.5) circle [radius=0.175]; \draw[thick]  (14,3.5) circle [radius=0.175];
    \fill[white]  (16, 3.3) circle [radius=0.175]; \draw[thick]  (16,3.3) circle [radius=0.175];
    \fill[white]  (18, -1) circle [radius=0.175]; \draw[thick]  (18,-1) circle [radius=0.175];
    \fill[white]  (18, -2.4) circle [radius=0.175]; \draw[thick]  (18,-2.4) circle [radius=0.175];
    \fill[white]  (20, -3) circle [radius=0.175]; \draw[thick]  (20,-3) circle [radius=0.175];
    \fill[white]  (20, -1) circle [radius=0.175]; \draw[thick]  (20,-1) circle [radius=0.175];
    \node at (22,0) {{$\cdots\cdots$}};  \node at (24.5,-6) {{$N$}};
    \draw[-, thick] (24,-6) -- (24, 6);
  \end{tikzpicture}
  \caption{ This is a graphical representation of expansion \eqref{def:expansion}
    corresponding to the collisions of four random walks, each starting
    from the origin. Each solid line will be marked with the label of the walk that it corresponds to throughout the diagram.
    Each solid dot, which marks a collision among a subset $A$ of the random walks, is given a weight
    $\prod_{ i,j\in A} \sigma_N^{i,j}$. Any solid line between points $(m,x), (n,y)$ is assigned the
    weight of the simple random  walk transition kernel $q_{m-n}(y-x)$.
    The hollow dots are assigned weight $1$ and they mark the places where
    we simply apply the Chapman-Kolmogorov formula.
  }\label{fig:expansion}
\end{figure}

The graphical representation of expansion \eqref{def:expansion} is depicted in Figure \ref{fig:expansion}. There, we have marked
with black dots the space-time points $(n,x)$ where some of the walks collide and we have assigned to each the weight
$\prod_{1\leq i<j \leq h}\sigma_N^{i,j} \ind_{\{S_n^{(i)}=S_n^{(j)}=x \}}$.

We now want to write the above expansion as a convolution of Markovian operators, following the Markov property of the simple
random walks.
We can partition the time interval $\{0,1,...,N\}$ according to the times when collisions take place; these are depicted in Figure \ref{fig:expansion} by vertical lines.
In between two successive times $m,n$, the walks will move from their locations $(x^{(i)})_{i=1,...,h}$ at time $m$
to their new locations  $(y^{(i)})_{i=1,...,h}$ at time $n$ (some of which might coincide) according to their transition probabilities, giving a
total weight to this transition of $\prod_{i=1}^h q_{n-m}(y^{(i)}-x^{(i)})$.
We, now, want to encode in this product
the coincidences that may take place
within the sets $(x^{(i)})_{i=1,...,h}$ and $(y^{(i)})_{i=1,...,h}$. To this end, we consider partitions $I$ of the set of indices
$\{1,...,h\}$, which we denote by $I\vdash \{1,...,h\}$. We also denote by $|I|$ the number of parts of $I$.
Given a partition $I \vdash \{1,\dots,h\}$, we define an equivalence relation $\stackrel{I}{\sim}$ in $\{1,\dots,h\}$ such that $k \stackrel{I}{\sim} \ell$ if and only if $k$ and $\ell$ belong to the same part of partition $I$. Given a vector $\by=(y_1,\dots,y_h) \in (\Z^2)^h$ and $I \vdash \{1,\dots,h\}$, we shall use the notation $\by \sim I$ to mean that $y_k=y_\ell$ for all pairs $k \stackrel{I}{\sim}\ell$. We use the symbol $\circ$ to denote the one-part partition\footnote{the notation $\circ$, with which we denote the one-part partition, here, should not be confused with the $\circ$ that appears in the figures, where
  it just marks places where we apply the Chapman-Kolmogorov formula. }, that is, $\circ:=\{1,\dots,h\}$,
and $*$ to denote the partition consisting only of singletons, that is $*:=\bigsqcup_{i=1}^h \{i\}$. Moreover, given $I \vdash \{1,\dots,h\}$ such that $|I|=h-1$ and $I=\{i,j\}\sqcup \bigsqcup_{k \neq i,j} \{k\}$, by slightly abusing notation, we may identify and denote $I$ by its non-trivial part $\{i,j\}$.

Given this formalism, we denote the total transition weight of the $h$ walks, from points $\bx=(x^{(1)},...,x^{(h)})\in (\Z^2)^h$,
subject to constraints $\bx \sim I$ at time $m$, to points $\by=(y^{(1)},...,y^{(h)})\in(\Z^2)^h$, subject to constraints $\by \sim J$ at time $n$, by
\begin{equation} \label{free_evol}
  Q^{I,J}_{n-m}(\bx,\by):=\ind_{\{\bx\sim I\}} \, \prod_{i=1}^h q_{n-m}(y^{(i)}-x^{(i)}) \, \ind_{\{\by \sim J\}}\, .
\end{equation}
We will call this operator the {\bf constrained evolution}.
Furthermore, for a partition $I \vdash \{1,\dots,h\}$ and $\boldsymbol{\beta}=\{\beta_{i,j}\}_{1\leq i<j \leq h}$
we define the {\bf mixed collision weight} subject to $I$ as
\begin{equation}\label{def:multi-sigma}
  \sigma_N(I):=\sigma_N(I,\{\beta_{i,j}\}_{1\leq i<j \leq h})=\prodtwo{1\leq i<j \leq h,}{i \stackrel{I}{\sim}j} \sigma_N^{i,j} ,
\end{equation}
with $\sigma_N^{i,j} $ as defined in \eqref{def:sigma}.
We can then rewrite \eqref{def:expansion} in the form
\begin{equation} \label{laplace_expansion}
  1+ \sum_{r=1}^{\infty} \sumtwo{\circ:=I_0}{I_1,\dots,I_r \neq *} \prod_{i=1}^r \sigma_N(I_i) \sumtwo{1 \leq n_1<\dots<n_r\leq N}{0:=\bx_0 ,\bx_1, \dots, \bx_r \in (\Z^2)^h} \prod_{i=1}^r Q^{I_{i-1};I_i}_{n_i-n_{i-1}}(\bx_{i-1},\bx_i) \, .
\end{equation}

\begin{figure}
  \begin{tikzpicture}[scale=0.45]
    \foreach \i in {6,7,...,10}{
        \draw[-, thick] (2*\i,-6) -- (2*\i, 6);
      }
    \draw[-, thick] (0,-6) -- (0, 6); \draw[-, thick] (6,-6) -- (6, 6); \draw[-, thick] (8,-6) -- (8, 6);
    \fill[black]  (0,0) circle [radius=0.175];
    \fill[black]  (6,0) circle [radius=0.175]; \fill[black]  (8,2.5) circle [radius=0.175];
    \fill[black]  (12,2.5) circle [radius=0.175];
    \fill[black]  (14, -3) circle [radius=0.175]; \fill[black]  (16, -1) circle [radius=0.175];
    \fill[black]  (18, 2) circle [radius=0.175]; \fill[black]  (20, 2) circle [radius=0.175];
    \draw [-,thick, decorate, decoration={snake,amplitude=.4mm,segment length=2mm}] (0,0) -- (6,0);
    \draw[thick] (6,0) to [out=75,in=200] (8,2.5) ; \draw[thick] (0,0) to [out=60,in=160] (8,2.5) ;
    \draw [-,thick, decorate, decoration={snake,amplitude=.4mm,segment length=2mm}] (8,2.5) -- (12,2.5);
    \draw[thick] (0,0) to [out=-60,in=200] (14,-3) ; \draw[thick] (12,2.5) to [out=-60,in=120] (14,-3) ;
    \draw[thick] (6,0) to [out=-30,in=180] (13, -1) ; \draw[thick] (13.5,-1) to [out=0,in=180] (16, -1) ;
    \draw[thick] (14,-3) to [out=30,in=-120] (16, -1) ;
    \draw[thick] (16,-1) to [out=30,in=-100] (18, 2) ; \draw[thick] (12, 2.5) to [out=45,in=135] (18, 2) ;
    \draw[thick] (18, 2) to [out=45,in=135] (20, 2) ;  \draw[thick] (18, 2) to [out=-45,in=-135] (20, 2) ;
    \draw[thick] (16, -1) to [out=-45,in=180] (21, -3) ;  \draw[thick] (14, -3) to [out=60,in=200] (16, -1) ;
    \draw[thick] (16, -1) to [out=0,in=180] (21, -1) ;
    \draw[thick] (20, 2) to [out=-45,in=145] (21, 1) ;   \draw[thick] (20, 2) to [out=45,in=-145] (21, 3) ;
    \fill[white]  (6, -4) circle [radius=0.175]; \draw[thick]  (6,-4) circle [radius=0.175];
    \fill[white]  (6, 2.9) circle [radius=0.175]; \draw[thick]  (6,2.9) circle [radius=0.175];
    \fill[white]  (8, -4.2) circle [radius=0.175]; \draw[thick]  (8,-4.2) circle [radius=0.175];
    \fill[white]  (8, -0.8) circle [radius=0.175]; \draw[thick]  (8,-0.8) circle [radius=0.175];
    \fill[white]  (12, -1) circle [radius=0.175]; \draw[thick]  (12,-1) circle [radius=0.175];
    \fill[white]  (12, -3.6) circle [radius=0.175]; \draw[thick]  (12,-3.6) circle [radius=0.175];
    \fill[white]  (14, -1) circle [radius=0.175]; \draw[thick]  (14,-1) circle [radius=0.175];
    \fill[white]  (14, 3.5) circle [radius=0.175]; \draw[thick]  (14,3.5) circle [radius=0.175];
    \fill[white]  (16, 3.3) circle [radius=0.175]; \draw[thick]  (16,3.3) circle [radius=0.175];
    \fill[white]  (18, -1) circle [radius=0.175]; \draw[thick]  (18,-1) circle [radius=0.175];
    \fill[white]  (18, -2.4) circle [radius=0.175]; \draw[thick]  (18,-2.4) circle [radius=0.175];
    \fill[white]  (20, -3) circle [radius=0.175]; \draw[thick]  (20,-3) circle [radius=0.175];
    \fill[white]  (20, -1) circle [radius=0.175]; \draw[thick]  (20,-1) circle [radius=0.175];
    \node at (22,0) {{$\cdots\cdots$}};  \node at (24.5,-6) {{$N$}};
    \draw[-, thick] (24,-6) -- (24, 6);
  \end{tikzpicture}
  \caption{ This is the simplified version of Figure's \ref{fig:expansion}  graphical representation of the expansion
  \eqref{laplace_expansion}, where we have grouped together the blocks of consecutive collisions between the same
  pair of random walks. These are now represented by the wiggle lines ({\bf replicas}) and we call the evolution
  in strips that contain only one replica as {\bf replica evolution} (although strip seven is the beginning of another wiggle line,
  we have not represented it as such since we have not completed the picture beyond that point). 
  The wiggle lines (replicas) between points $(n,x), (m,y)$,
  corresponding to collisions of a single pair of walks $S^{(k)}, S^{(\ell)}$,
  are assigned weight $U_N^{\beta_{k,\ell}}(m-n, y-x)$. A solid line between points $(m,x), (n,y)$ is assigned the
  weight of the simple random
  walk transition kernel $q_{m-n}(y-x)$.
  }\label{fig:expansion2}
\end{figure}
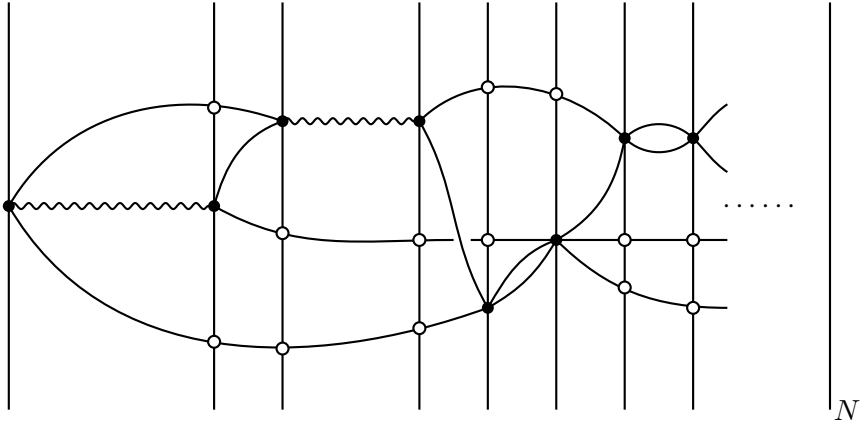

We want to make one more simplification in this representation, which, however, contains an important structural feature. This is to group together
consecutive constrained evolution operators $\sigma_N(I_i) Q^{I_{i-1};I_i}_{n_i-n_{i-1}}(\bx_{i-1},\bx_i) $ for which $I_{i-1}=I_i=h-1$.
An example in Figure \ref{fig:expansion} is the sequence of evolutions in the first three strips and another one is the group of evolutions in strips five and six.
Such groupings can be captured by the following definition:
For a partition $I\vdash\{1,\dots,h\}$ of the form $I=\{k,\ell\} \sqcup \bigsqcup_{j\neq k,\ell}\{j\}$ and $\bx=(x^{(1)},\dots,x^{(h)})$, $\by=(y^{(1)},\dots,y^{(h)}) \in (\Z^2)^h$, we define the {\bf replica evolution} as
  \begin{equation}\label{replica_op}
    \sfU^{I}_n(\bx,\by) := \ind_{\{\bx,\by\sim I\}}  \cdot U^{\beta_{k,\ell}}_{N}(n,y^{(k)}-x^{(k)}) \cdot \prod_{i\neq k,\ell} q_n(y^{(i)}-x^{(i)}) \, ,
  \end{equation}
with $U^{\beta}_{N}(n,y^{(k)}-x^{(k)})$ defined in \eqref{Un}. We name this {\it replica evolution} since in the time interval $[0,n]$
we see a stream of collisions
between only two of the random walks. The simplified version of expansion \eqref{laplace_expansion} (and Figure \ref{fig:expansion})
is presented in Figure \ref{fig:expansion2}.

In order to re-express \eqref{laplace_expansion}
with the reduction of the replica evolution \eqref{replica_op}, we need to introduce one more formalism, which is
\begin{align}\label{def:Pop}
  P^{I;J}_n(\bx,\by):=
  \begin{dcases}
    \sumtwo{m_1\geq 1,\,m_2\geq 0\colon \, m_1+m_2=n\, ,}{  \bz \in (\Z^2)^h}   Q^{I;J}_{m_1}(\bx,\bz)\cdot \sfU^{J}_{m_2}(\bz,\by)\, , & \text{ if } |J|=h-1\, ,  \\
    \qquad \qquad \qquad Q_n^{I;J}(\bx,\by)    \, ,                                                                                     & \text{ if } |J|<h-1 \, ,
  \end{dcases}
\end{align}
where we recall that $|J|$ is the number of parts of $J$ and so $|J|=h-1$ means that $J$ has the form 
$\{k,\ell\}\sqcup \bigsqcup_{i\neq k,\ell} \{i\}$,
corresponding to a pairwise collision, while $|J|<h-1$ means that there are multiple collisions (the latter would correspond to the end of the
eighth strip in Figure \ref{fig:expansion}).
In other words, the operator $P_n^{I;J}$ groups together each replica evolution with its preceding constrained evolution.

We, finally, arrive at the desired expression for the Laplace transform of the many-body collisions:
\begin{equation} \label{def:Pexpansion}
  M^{\boldsymbol{\beta}}_{N,h}
  =  1+\sum_{r = 1}^{\infty}  \sum_{\substack{\circ:=I_0,I_1,\dots,I_r \\ I_j\neq I_{j+1}\,,\, \text{if $|I_j|=|I_{j+1}|=h-1$} \\ \text{for}\,\, j=1,...,r-1}}
   \prod_{i=1}^r  \sigma_N(I_i)\sumtwo{1 \leq n_1<\dots<n_r\leq N \, ,}{0:=\bx_0 ,\bx_1, \dots, \bx_r \in (\Z^2)^h} \prod_{i=1}^r
  P^{I_{i-1};I_i}_{n_i-n_{i-1}}(\bx_{i-1},\bx_i)  \, .
\end{equation}

\subsection{Functional analytic framework and some auxiliary estimates} \label{Operators}

Let us start with some, fairly easy, bounds on operators $Q$ and $\sfU$ (with the estimate on the latter
being an upgrade of estimate \eqref{Uest}). 
\begin{lemma}
  Let the operators $Q^{I;J}_n, \sfU^J_n$ be defined in \eqref{free_evol} and \eqref{replica_op}, respectively.
  For all partitions $I\neq J$ with $|J|=h-1$,  $\bar \beta<1$ defined in \eqref{barb} and $\sigma_N(I)$ defined in \eqref{def:multi-sigma},
  we have the bounds
  \begin{equation} \label{crude_bounds}
    \sum_{0 \leq n \leq N, \, \by \in (\Z^2)^h} \sfU^J_n(\bx,\by) \leq \frac{1}{1-{\bar\beta}\,'} \,
    \quad \text{and} \quad  \sigma_N(J) \cdot\Bigg( \sum_{1\leq n \leq N, \by \in (\Z^2)^h} Q^{I;J}_n(\bx,\by) \Bigg)\leq \bar{\beta}\,' \, ,
  \end{equation}
for any $\bar{\beta}\,' \in (\bar\beta, 1)$ and all large enough $N$.

\end{lemma}
\begin{proof}
  We start by proving the first bound in \eqref{crude_bounds}.
  By definition \eqref{replica_op} we have that
  \begin{equation*}
    \begin{split}
      \sum_{ n\geq 0, \by \in (\Z^2)^h} \sfU^{J}_n(\bx,\by) &:= 
      \sum_{n\geq 0, \, \by \in (\Z^2)^h}\ind_{\{\bx,\by\sim J\}}  \cdot U^{\beta_{k,\ell}}_{N}(n,y^{(k)}-x^{(k)}) \cdot \prod_{j\neq k,\ell} q_n(y^{(j)}-x^{(j)})\\
      & = \sum_{n \geq 0} U^{\beta_{k,\ell}}_N(n) \, ,
    \end{split}
  \end{equation*}
  by using that $\sum_{z \in \Z^2} q_n(z)=1$ to sum all the kernels 
  $q_n(y^{(j)}-x^{(j)})$ for $j\neq k,\ell$ and $\sum_{z \in \Z^2}U^{\beta}_N(n,z)=U^{\beta}_N(n)$. Moreover, by 
  definition \eqref{Un} and \eqref{def:sigma},  since $\beta_{k,\ell} \leq \bar \beta$, we have
  \begin{equation*}
    \sum_{n \geq 0} U^{\beta_{k,\ell}}_N(n) \leq \sum_{n \geq 0} U^{\bar{\beta}}_N(n)
    =\sum_{k \geq 0} (\sigma_N(\bar \beta) R_N)^k \,\P(\tau^{\ms (N)}_k=n) \, ,
  \end{equation*}
  and by \eqref{sigmaR} we have that for any $\bar\beta\,'\in (\bar\beta, 1)$ and all $N$ large enough
  \begin{equation*}
    \sum_{n \geq 0} U^{\beta_{k,\ell}}_N(n)
    \leq \sum_{k \geq 0} (\bar\beta\,')^k \,\P(\tau^{\ms (N)}_k=n)
    \leq \sum_{k \geq 0} (\bar\beta\,')^k
    =\frac{1}{1-\bar\beta\,'} \, .
  \end{equation*}
  Therefore,
  \begin{equation*}
    \sum_{ n\geq 0, \by \in (\Z^2)^h} \sfU^{J}_n(\bx,\by) \leq (1-\bar\beta\,')^{-1} \, .
  \end{equation*}
  For the second bound in \eqref{crude_bounds} we recall from \eqref{free_evol} that when $J=\{k,\ell\} \sqcup \bigsqcup_{j \neq k,\ell} \{j\}$, then
  \begin{equation*}
    Q^{I;J}_n(\bx,\by):=\bigg(\ind_{\{\bx\sim I\}} \, \prod_{j \neq k,\ell} q_n(y^{(j)}-x^{(j)}) \bigg)\cdot q_n(y^{(k)}-x^{(k)})\cdot q_n(y^{(k)}-x^{(\ell)})\, ,
  \end{equation*}
  since $\by\sim J$ means that $y_k=y_\ell$.
  Therefore, $\sigma_N(J) =\sigma_N(\beta_{i,j})\leq \sigma_N(\bar\beta)$.
  We, now, use that $\sum_{z \in \Z^2}q_n(z)=1$ in order to sum the kernels $q_n(y^{(j)}-x^{(j)}), \, j \neq k,\ell$,
  while we also have by Cauchy-Schwarz that
  \begin{equation*} 
    \sigma_N(J)\cdot \bigg( \sum_{1\leq n \leq N,\, y_k \in \Z^2} q_n(y^{(k)}-x^{(k)})\cdot q_n(y^{(k)}-x^{(\ell)})\bigg)
    \leq \sigma_N(\bar\beta)\cdot \Big(\sum_{n=1}^N q_{2n}(0)\Big) \leq  \bar{\beta}\,' \, ,
  \end{equation*}
  by \eqref{sigmaR}, for all $N$ large enough,
  thus establishing the second bound in \eqref{crude_bounds}.
\end{proof}
Next, in Proposition \ref{operator_norm_bound},
we are going to recall some norm estimates from \cite{LZ21} on the Laplace transform of operators $P^{I;J}_{n}$, defined \eqref{def:Pop}. For this, we need
to set up the functional analytic framework.
We start by defining $(\Z^2)^h_I:=\{\by \in (\Z^2)^h: \by \sim I\}$ and, for $q \in (1,\infty)$,
the $\ell^q((\Z^2)^h_I)$ space of functions $f:(\Z^2)^h_I \to \R$ which have finite norm
\begin{equation*}
\norm{f}_{\ell^q_I}:=\norm{f}_{\ell^q((\Z^2)^h_I)}:=\Bigg(\sum_{\by \in ((\Z^2)^h_I) }\big|f(\by)\big|^q \Bigg)^{\frac{1}{q}}\, .
\end{equation*}
For $q \in (1,\infty)$ and for an {operator $\sfT:\ell^q\big((\Z^2)^h_J\big)\rightarrow \ell^q\big((\Z^2)^h_I\big)$
 with kernel $\sfT(\bx,\by)$,} one can define the pairing
\begin{equation} \label{pairing}
  \langle f,\sfT g\rangle:=\sum_{\bx \,\in (\Z^2)^h_I,\by \,\in (\Z^2)^h_J} f(\bx)\sfT(\bx,\by)g(\by) \, .
\end{equation}
The operator norm will be given by
\begin{equation}\label{op_norm}
  \norm{\sfT}_{\ell^q \to \ell^q}:= \sup_{\norm{g}_{\ell^q_J} \leq 1} \norm{\sfT g}_{\ell^q_I}
  =\sup_{\norm{f}_{\ell^p_I}\leq 1, \, \norm{g}_{\ell^q_J}\leq 1} \langle f,\sfT g \rangle \, ,
\end{equation}
for $p,q \in (1,\infty)$ conjugate exponents, i.e. $\frac{1}{p}+\frac{1}{q}=1$.

We introduce the weighted Laplace transforms of operators $Q^{I,J}$ and $\sfU^J$.
In particular, let $w(x)$ be any continuous function in $L^\infty(\R^2)\cap L^1(\R^2)$ such that $\log w(x)$ is Lipschitz (one can think of
$w(x)=e^{-|x|}$) and define $w_N(x):=w(x/\sqrt{N})$. Also, for a function $g\colon \R^2\to \R$ we define the tensor product
$g^{\otimes h}(x_1,...,x_h)=g(x_1)\cdots g(x_h)$,  The weighted Laplace transforms are now defined as
\begin{equation} \label{operators}
  \begin{split}
    & \widehat{\sfQ}^{I;J}_{N,\lambda}(\bx,\by):=\bigg(\sum_{n \geq 1}^{N} e^{-\lambda \frac{n}{N}} Q^{I;J}_{n}(\bx,\by)\bigg)\cdot \frac{w^{\otimes h}_N(\bx)}{w^{\otimes h}_N(\by)} \, ,\\
    & \widehat{\sfU}^{J}_{N,\lambda}(\bx,\by):=\bigg(\sum_{n \geq 0}^{N}  e^{-\lambda \frac{n}{N}}\, \sfU^{J}_{n}(\bx,\by)\bigg)\cdot \frac{w^{\otimes h}_N(\bx)}{w^{\otimes h}_N(\by)}\,  .
  \end{split}
\end{equation}
The passage to a Laplace transform will help to estimate convolutions involving operators $Q^{I;J}_{n}(\bx,\by)$ and $\sfU^{J}_{n}(\bx,\by)$ and the introduction
of the weight comes handy in improving integrability when these operators are applied to functions which are not in $\ell^1((\Z^2)^h)$.
We will see this in Lemma \ref{rterm-estimate} below.

We also define the Laplace transform operator of the combined evolution \eqref{def:Pop}:
\begin{equation}\label{LaplaceP}
  \widehat\sfP^{I;J}_{N,\lambda}= \begin{dcases}
    \widehat{\sfQ}^{I;J}_{N,\lambda} \, ,                                  & \text{ if   } |J| \neq h-1   \\
    \widehat{\sfQ}^{I;J}_{ N,\lambda} \, \widehat{\sfU}^J_{N,\lambda} \, , & \text{ if } |J|=h-1 \, .
  \end{dcases}
\end{equation}
For our purposes, it will be sufficient to take $\lambda=0$
and consider operators $\widehat{\sfQ}^{I;J}_{N,0}, \widehat{\sfU}^J_{N,0}$ and  $ \widehat\sfP^{I;J}_{N,0}$.

Using the above formalism we summarise in the next proposition some key estimates of \cite{LZ21} (see Propositions 3.2-3.4 and the proof of Theorem 1.3 therein), which are refinements of estimates in \cite{CSZ21} (Section 6) and \cite{DFT94} (Section 3).

\begin{proposition} \label{operator_norm_bound}
  Consider the operators $\widehat{\sfQ}^{I;J}_{N,0}$ and  $ \widehat\sfP^{I;J}_{N,0}$ defined in \eqref{operators} and \eqref{LaplaceP} with $\lambda=0$ and a
  weight function $w\in L^\infty(\R^2)\cap L^1(\R^2)$ such that $\log w(x)$ is Lipschitz.
  Then there exists a constant $C=C(h,\bar{\beta},w) \in (0, \infty)$ (recall $\bar\beta$ from \eqref{barb})
  such that for all $p,q\in (1,\infty)$ with $\frac{1}{p}+\frac{1}{q}=1$ and all partitions $I,J \vdash \{1,\dots,h\}$, 
  such that $|I|,|J|\leq h-1$ and $I \neq J$ when $|I|=|J|=h-1$, we have that
  \begin{equation} \label{P_norm_bound}
    \norm{\widehat\sfP_{N,0}^{I;J}}_{\ell^q \to \ell^q} \leq C \, p \, q \, .
  \end{equation}
  Moreover, if $g \in \ell^q(\Z^2)$,
  \begin{equation} \label{singleton_bound}
    \norm{  \, g^{\otimes h} \, \widehat{\sfQ}^{*;I}_{N,0} }_{\ell^p} \leq C \,q\, N^{\frac{1}{q}} \norm{g}^{h}_{\ell^p} \, ,
  \end{equation}
  for $g^{\otimes h}(x_1,...,x_h):=g(x_1)\cdots g(x_h)$.
\end{proposition}
Let us now present the following lemma, which demonstrates how the above functional analytic framework will be used.
This lemma will be useful in the first approximation, that we will perform in the next Section, in showing that contributions from
multiple, i.e. three or more, collisions are negligible.
\begin{lemma}\label{rterm-estimate}
  Let $H_{r,N}$ be the $r^{th}$ term in the expansion \eqref{def:Pexpansion}, that is,
  \begin{equation} \label{def:H}
    H_{r,N}:= \sum_{\circ:=I_0,I_1,\dots,I_r} \prod_{i=1}^r  \sigma_N(I_i)\sumtwo{1 \leq n_1<\dots<n_r\leq N,}{0:=\bx_0 ,\bx_1, \dots, \bx_r \in (\Z^2)^h} \prod_{i=1}^r P^{I_{i-1};I_i}_{n_i-n_{i-1}}(\bx_{i-1},\bx_i)  \, ,
  \end{equation}
  and $H^{\mathsf{(multi)}}_{r,N}$ be
  the corresponding term with the additional constraint that there is at least one multiple collision (i.e. at some point,
  three or more walks meet), that is,
  \begin{equation*} 
    H^{\mathsf{(multi)}}_{r,N}
    := \sum_{\circ:=I_0,I_1,\dots,I_r}\Bigg(\prod_{i=1}^r  \sigma_N(I_i)\Bigg) \ind_{\{\exists\, 1\leq j\leq r \, : \, |I_j|<h-1 \}}\sumtwo{1 \leq n_1<\dots<n_r\leq N}{0:=\bx_0 ,\bx_1, \dots, \bx_r \in (\Z^2)^h} \prod_{i=1}^r P^{I_{i-1};I_i}_{n_i-n_{i-1}}(\bx_{i-1},\bx_i) \, .
  \end{equation*}
  Then the following bounds hold:
  \begin{align}
    H_{r,N} \leq \Big(\frac{C\,p\,q}{\log N} \Big)^r N^{\frac{h+1}{q}}  \qquad
    \text{and} \qquad
    H^{\mathsf{(multi)}}_{r,N}\leq \frac{r}{\log N} \Big(\frac{C\,p\,q}{\log N} \Big)^r N^{\frac{h+1}{q}}  \label{Hbound2}.
  \end{align}
  for any $p,q\in(1,\infty)$ with $\tfrac{1}{p}+\tfrac{1}{q}=1$ and a constant $C$ that depends on $h$ and $ \bar\beta$ but
  is independent of $N,r, p,q$.
\end{lemma}
\begin{proof}
  We start by considering $w(x)=e^{-|x|}, w_N(x):=w(\tfrac{x}{\sqrt{N}})$ and $w_N^{\otimes h}(x_1,...,x_h)=\prod_{i=1}^h w_N(x_i)$
  and by including in the expression \eqref{def:H} the term
  $$\frac{1}{w_N^{\otimes h}(\bx_0)} \Big(\prod_{i=1}^r \frac{w_N^{\otimes h}(\bx_{i-1})}{w_N^{\otimes h}(\bx_i)} \Big)w_N^{\otimes h}(\bx_r)
    =1,$$
  thus rewriting $H_{r,N}$ as
  \begin{equation*} 
    \begin{split}
      H_{r,N}
      &=   \sum_{\circ:=I_0,I_1,\dots,I_r} \prod_{i=1}^r  \sigma_N(I_i)  \\
      &\times  \sumtwo{1 \leq n_1<\dots<n_r \leq N}{0:=\bx_0 ,\bx_1, \dots, \,\bx_r  \in (\Z^2)^h}
      \frac{1}{w_N^{\otimes h}(\bx_0)} \prod_{i=1}^r P^{I_{i-1};I_i}_{n_i-n_{i-1}}(\bx_{i-1},\bx_i)
      \frac{w_N^{\otimes h}(\bx_{i-1})}{w_N^{\otimes h}(\bx_i)}  \cdot
      w_N^{\otimes h}(\bx_r)\, .
    \end{split}
  \end{equation*}
  We can extend the summation on $\bx_0$ from $\bx_0=0$ to $\bx_0\in\Z^2$ by introducing a delta function $\delta_0^{\otimes h}$
  at zero. Then
  \begin{equation*}
    \begin{split}
      H_{r,N}
      &=   \sum_{*:=I_0,I_1,\dots,I_r} \prod_{i=1}^r  \sigma_N(I_i)  \\
      &\times  \sumtwo{1 \leq n_1<\dots<n_r \leq N}{\bx_0 ,\bx_1, \dots, \,\bx_r \in (\Z^2)^h}
      \frac{\delta_0^{\otimes h}(\bx_0)}{w_N^{\otimes h}(\bx_0)} \prod_{i=1}^r P^{I_{i-1};I_i}_{n_i-n_{i-1}}(\bx_{i-1},\bx_i)
      \frac{w_N^{\otimes h}(\bx_{i-1})}{w_N^{\otimes h}(\bx_i)}  \cdot
      w_N^{\otimes h}(\bx_r)\, .
    \end{split}
  \end{equation*}
  We can, now, bound the last expression by extending the temporal range of summations
  from $1\leq n_1<\dots<n_r\leq N$ to $n_i-n_{i-1} \in \{1,\dots, N\}$ for all $i=1,...,r$. Recalling the definition of the Laplace transforms of the
  operators \eqref{operators}, \eqref{LaplaceP}, we, thus, obtain the upper bound
  \begin{equation*}
    H_{r,N}
    \leq   \sum_{*:=I_0,I_1,\dots,I_r}  \prod_{i=1}^r  \sigma_N(I_i)
    \sum_{\bx_0 ,\bx_1, \dots, \,\bx_r \in (\Z^2)^h}
    \frac{\delta_0^{\otimes h}(\bx_0)}{w_N^{\otimes h}(\bx_0)} \prod_{i=1}^r \widehat \sfP^{I_{i-1};I_i}_{N,0}(\bx_{i-1},\bx_i)
    \cdot  w_N^{\otimes h}(\bx_r)\, ,
  \end{equation*}
  which we can write in the more compact and useful notation, using the brackets \eqref{pairing}, as
  \begin{equation*}
    H_{r,N}  \leq \sum_{I_1,\dots,I_r} \Big\langle \, \frac{\delta_0^{\otimes h}}{w_N^{\otimes h}} , \widehat{\sfQ}^{*;I_1}_{N,0} \widehat{\sfP}^{I_1;I_2}_{N,0} \cdots \widehat{\sfP}^{I_{r-1};I_r}_{N,0} \, w^{\otimes h}_N\Big\rangle  \prod_{i=1}^r \sigma_N(I_i) .
  \end{equation*}
  We note, here, that in the right-hand side we set the $I_0$ partition to be equal to $I_0=\{1\}\sqcup\cdots \sqcup \{h\}$. 
  In this case $\widehat{\sfP}^{I_{0};I_1}=\widehat{\sfP}^{*;I_1} =\widehat{\sfQ}^{*;I_1} $ by definition \eqref{LaplaceP}. 
  The
  delta function $\delta_0^{\otimes h}(\bx_0)$ will force all points of $\bx_0$ to coincide at zero, thus, forcing $I_0$ to be equal to
  the partition $\circ=\{1,...,h\}$ but, at the stage of operators, we do not yet need to enforce this constraint.
  At this stage we can proceed with the estimate using the operator norms \eqref{op_norm} as
  \begin{equation} \label{bound_by_pairing}
    H_{r,N}
    \leq \sum_{I_1,\dots,I_r} \norm{ { \frac{\delta^{\otimes h}_0}{w^{\otimes h }_N} \, \widehat{\sfQ}^{*,I_1}_{N,0}} }_{\ell^p} \, \prod_{i=2}^r\,  \norm{\widehat{\sfP}_{N,0}^{I_{i-1};I_i}}_{\ell^q \to \ell^q} \, \norm{ w_N^{\otimes h}}_{\ell^q } \,
    \, \cdot \,  \prod_{i=1}^r \sigma_N(I_i),
  \end{equation}
  By \eqref{singleton_bound} of Proposition \ref{operator_norm_bound} we have that
  \begin{equation*}
    \norm{{  \frac{\delta^{\otimes h}_0}{w^{\otimes h }_N} \, \widehat{\sfQ}^{*,I_1}_{N,0}} }_{\ell^p}
    \leq C\,q\, N^{\frac{1}{q}} \norm{\frac{\delta_0}{w_N}}^{h}_{\ell^p}=C\, q\, N^{\frac{1}{q}},
  \end{equation*}
  and by \eqref{P_norm_bound} we have that for all $1\leq i \leq r-1$,
  \begin{equation*}
    \norm{\widehat\sfP_{N,0}^{I_{i-1};I_{i}}}_{\ell^q \to \ell^q} \leq C \, p \, q \, .
  \end{equation*}
  Inserting these estimates in  \eqref{bound_by_pairing} we deduce that
  \begin{align} \label{last_bound}
    H_{r,N}
     & \leq (C\,p\,q)^r N^{\frac{1}{q}}
    \norm{\frac{\delta_0}{w_N}}^h_{\ell^p} \norm{w_N}_{\ell^q}^{h}  \sum_{I_1,\dots,I_r} \prod_{i=1}^r\sigma_N(I_i) \notag \\
     & = (C\,p\,q)^r N^{\frac{1}{q}} \norm{w_N}_{\ell^q}^{h}  \sum_{I_1,\dots,I_r} \prod_{i=1}^r\sigma_N(I_i) \, ,
  \end{align}
  for a constant $C=C(h,\bar{\beta})\in (0,\infty)$, not depending on $p,q, r, N$.
  We now notice that for any partition $I\vdash\{1,...,h\}$, it holds that $\sigma_N(I)\leq C/ \log N$
  (recall definitions \eqref{def:multi-sigma} and  \eqref{def:sigma}), so
  \begin{align*}
   \sum_{I_1,\dots,I_r} \prod_{i=1}^r\sigma_N(I_i) \leq \Big(\frac{2^h  C}{\log N}\Big)^r.
  \end{align*}
  
  Moreover,  by Riemann summation,
  $N^{-h/q}\norm{w_N}_{\ell^q}^{h}$ is bounded uniformly in $N$. Therefore, applying these on \eqref{last_bound} we arrive at the bound
  \begin{align*} 
    H_{r,N} \leq \Big(\frac{C\,p\,q}{\log N} \Big)^r N^{\frac{h+1}{q}},
  \end{align*}
  for a new constant $C=C(h,\bar{\beta})\in (0,\infty)$, which is the first claimed estimate in \eqref{Hbound2}.
  For the second estimate in \eqref{Hbound2} we follow
  the same steps until we arrive at the bound
  \begin{equation*} 
    H^{\mathsf{(multi)}}_{r,N} \leq (C\,p\,q)^r N^{\frac{1}{q}}
    \norm{w_N}_{\ell^q}^{h} \sum_{I_1,\dots,I_r}\prod_{i=1}^r  \sigma_N(I_i) \ind_{\{\exists\, 1\leq j\leq r:|I_j|<h-1\}}.
  \end{equation*}
  Then we notice that for a partition $I\vdash\{1,...,h\}$ with $|I|<h-1$ it will hold that $\sigma_N(I)\leq C (\log N)^{-2}$
  (recall definitions \eqref{def:multi-sigma} and \eqref{def:sigma}).
  This fact, together with the fact that there are $r$ possible choices among the partitions $I_1,...,I_r$  that can be chosen
  so that $|I_j|<h-1$, leads to the second bound in \eqref{Hbound2}.
\end{proof}

\section{Approximation steps and proof of the theorem} \label{approximation_steps}
In this section we prove Theorem \ref{main_result} through a series of approximations on the chaos expansion
\eqref{def:expansion}, \eqref{def:Pexpansion}. The first step, in Section \ref{step_1}, is to establish that the
series in the chaos expansion \eqref{def:Pexpansion} can be truncated up to a finite order and that the main contribution comes from
diagrams where, at any fixed time, we only have at most two walks colliding. The second step, Section \ref{step_2},
is to show that the main contribution to the expansion and to diagrams like in Figure \ref{fig:scales}, comes
when all jumps between marked dots (see  Figure \ref{fig:scales}) happen within diffusive scale.
The third step, in Section \ref{step_3}, captures the important feature of {\it scale separation}.
This is intrinsic to the two-dimensionality and can be seen as the main feature that leads to the
asymptotic independence of the collision times.
With reference to Figure \ref{fig:scales}, this says that the time between two consecutive replicas, say $a_4-b_3$ in
Figure \ref{fig:scales} must be much larger than the time between the previous replicas, say $b_3-b_2$.
This would then lead to the next step in Section \ref{step_5}, see also Figure \ref{fig:rewiring}, which is that
we can {\it rewire} the links so that the solid lines connect only replicas between the {\it same} pairs of walks.
The final step, which is performed in Section \ref{sec:finalstep} is to reverse all the above approximations within the
rewired diagrams, to which we arrived in the previous step. The summation, then, of all rewired diagrams leads,
in the limit, to the right hand of \eqref{thmeq:laplace2}, thus completing the proof of the theorem.

\subsection{Reduction to 2-body collisions and finite order chaoses.} \label{step_1}

In this step, we use the functional analytic framework and estimates of the previous section
to show that for each $r \geq 1$, $H_{r,N}$ decays exponentially in $r$, uniformly in $N \in \N$ and that it
is concentrated on configurations which contain only two-body collisions between the $h$ random walks.

\begin{proposition} \label{fixed_deg_bounds}
  There exist constants $\sfa \in (0,1)$ and $\bar{C}=C(h,\bar{\beta},\sfa) \in (0,\infty)$ such that
  for all $r \geq 1$,
  \begin{equation} \label{app_1}
    \sup_{N \in \N} H_{r,N}  \leq \bar{C} \, \sfa^r \, ,\qquad \text{and} \qquad
    H^{\mathsf{(multi)}}_{r,N}
    \leq \frac{\bar{C}}{\log N}   \, r\, \sfa^r   \, .
  \end{equation}
\end{proposition}
\begin{proof}
  We use the estimates in \eqref{Hbound2} and make the choice
  $q=q_N:=\frac{\sfa}{C_1} \log N $ with $\sfa \in (0,1)$ and a constant $C_1$
  such that $\frac{Cpq}{\log N} <\sfa$ (recall that $\tfrac{1}{p}+\tfrac{1}{q}=1$).
  Moreover, this choice of $q$ implies that
  \begin{equation*}
    N^{\frac{h+1}{q}}=e^{\frac{h+1}{q} \log N}=e^{\frac{C_1\, (h+1)}{\sfa}} .
  \end{equation*}
  Therefore, choosing $\bar{C}=\, e^{\frac{C_1\, (h+1)}{\sfa}}$ implies the first estimate in \eqref{app_1}.

  The second estimate follows from the same procedure and the same choice of $q=q_N:=\frac{\sfa}{C_1} \log N $ in the second
  bound of \eqref{Hbound2}.
\end{proof}

\begin{proposition} \label{finite_replicas}
  If $M^{\boldsymbol{\beta}}_{N,h}$ is the joint Laplace transform 
  of the collision local times $ \Big\{ \tfrac{\pi}{\log N}  \,\sfL_N^{(i,j)} \Big\}_{1\leq i<j\leq h} $, \eqref{M_N,h}
  and $H_{r,N}$ is the $r^{th}$ term
  in its chaos expansion \eqref{def:H}, then for any $\epsilon>0$ there exists $K=K_\epsilon$ such that
  \begin{equation*}
    \Big|M^{\boldsymbol{\beta}}_{N,h}-\sum_{r=0}^K H_{r,N}\Big|\leq \epsilon \, ,
  \end{equation*}
  uniformly for all $N \in \N$.
\end{proposition}
\begin{proof}
  By Proposition, \ref{fixed_deg_bounds}, $H_{r,N}$ decay exponentially in $r$, uniformly in $N \in \N$ and therefore
  \begin{equation*}
    \limsup_{K \to \infty}\, \Big(\sup_{N\geq 1} \,\sum_{r>K} H_{r,N}\Big)=0 \, ,
  \end{equation*}
  which means that we can truncate the expansion of $M^{\boldsymbol{\beta}}_{N,h}$ to a finite number of terms $K$ depending
  only on $\epsilon$.
\end{proof}

By Proposition \ref{fixed_deg_bounds} we can focus on only two-body collisions, since higher order collisions bear a negligible contribution as $N\to \infty$.
Let us introduce some notation to conveniently describe the expansion of $H_{r,N}$, after the reduction to only two-body collisions,
which we will use in the sequel. Given $r\geq 1$ we will denote by $a_i,b_i \in \N \cup \{0\}$, $a_i\leq b_i$, $i=1,\dots,r$ the times where replicas start and end respectively, see \eqref{replica_op} and Figure \ref{fig:expansion2}, where replicas are represented by wiggle lines.
Thus, $a_i$ will be the time marking the beginning of the $i^{th}$ wiggle line and $b_i$ the time marking its end. Note that, $a_1=0$.
Moreover, we use the notation $\vec{\bx}=(\bx_1,\bx_2,\dots,\bx_r) \in (\Z^{2})^{hr}$ to denote the starting points of the $r$ replicas and $\vec{\by}=(\by_1,\dots,\by_r) \in (\Z^{2})^{hr}$ the corresponding ending points. Again, notice that $\bx_1=\boldsymbol{0}$. We then define the set
\begin{equation} \label{C_set}
  \sfC_{r,N}:=\Big\{ (\vec{a},\vec{b},\vec{\bx},\vec{\by})\Big|\,  0:=a_1\leq b_1<a_2\leq \dots < a_r\leq b_r \leq N, \, \vec{\bx},\vec{\by} \in (\Z^2)^{hr}  \, ,\bx_1=\boldsymbol{0}   \Big\} \, .
\end{equation}
We also define a set of finite sequences of partitions
\begin{equation*} 
  \mathcal{I}^{(2)}=\bigcup_{r=0}^{\infty} \bigg\{(I_1,\dots,I_r) : I_j \neq I_{j+1} \text{ and } |I_j|=h-1\, ,\forall j \in \{1,\dots,r\} \bigg\} \, .
\end{equation*}
Using the notational conventions outlined above we can write
$ H_{r,N}= H^{(2)}_{r,N}+H^{\mathsf{(multi)}}_{r,N}$ with
\begin{align}\label{approx_1}
  H^{(2)}_{r,N} & :=\hspace{-0.1cm}
  \sum_{(I_1,...,I_r) \,\in\, \mathcal{I}^{(2)}}\sum_{(\vec{a},\vec{b},\vec{\bx},\vec{\by})\, \in \,\sfC_{r,N}} \sfU^{I_1}_{b_1}(0,\by_1) \prod_{i=2}^r Q_{a_{i}-b_{i-1}}^{I_{i-1};I_i}(\by_{i-1},\bx_i)  \sfU^{I_i}_{b_i-a_i}(\bx_i,\by_i) \sigma_N(I_i).
\end{align}
In the next sections will focus on  $H^{(2)}_{r,N}$, which by Proposition  \ref{fixed_deg_bounds}
contains the main contributions.

\subsection{Diffusive spatial truncation.}\label{step_2}

In this step we show that we can introduce diffusive spatial truncations in all the kernels appearing in \eqref{approx_1} which originate from the diffusive behaviour of the simple random walk in $\Z^2$. For a vector $\bx=(x^{(1)},\dots,x^{(h)}) \in (\Z^2)^h$, we shall use the notation
\begin{equation*} 
  \norm{\bx}_{\infty}=\max_{1\leq j \leq h} |x^{(j)}| \, ,
\end{equation*}
where $|\cdot|$ denotes the usual Euclidean norm on $\R^2$.
For each $r \in \N$, define $H^{\mathsf{(diff)}}_{r,R,N}$ to be the sum in \eqref{approx_1}
where $\sfC_{r,N}$ is replaced by
\begin{equation} \label{C_dif_set}
  \begin{split}
    \sfC^{\mathsf{(diff)}}_{r,N,R}:=\sfC_{r,N} \cap \Big\{(\vec{a},\vec{b},\vec{\bx},\vec{\by}): \,& \norm{\by_i-\bx_i}_{\infty}\leq R \sqrt{b_i-a_i} \\
    &\text{  and   } \norm{\bx_i-\by_{i-1}}_{\infty}\leq R \sqrt{a_i-b_{i-1}} \text{ for all } 1\leq i\leq r \Big\} \,
  \end{split}
\end{equation}
and similarly we define
\begin{equation} \label{H_sdiff}
  H^{\mathsf{(superdiff)}}_{r,N,R}
  = \hspace{-0.25cm}\sum_{(I_1,...,I_r) \,\in\, \mathcal{I}^{(2)}} \sum_{(\vec{a},\vec{b},\vec{\bx},\vec{\by})\, \in \,  \sfC^{\mathsf{(superdiff)}}_{r,N,R}
  } \sfU^{I_1}_{b_1}(0,\by_1) \prod_{i=2}^r Q_{a_{i}-b_{i-1}}^{I_{i-1};I_i}(\by_{i-1},\bx_i) \,  \sfU^{I_i}_{b_i-a_i}(\bx_i,\by_i) \,\sigma_N(I_i)  ,
\end{equation}
where
\begin{equation*} 
  \begin{split}
    \sfC^{\mathsf{(superdiff)}}_{r,N,R}
    :=\sfC_{r,N} \cap \Big\{(\vec{a},\vec{b},\vec{\bx},\vec{\by}): \, \exists \,1\leq i\leq r: \,& \norm{\by_i-\bx_i}_{\infty}> R \sqrt{b_i-a_i} \\
    &\text{  or   } \norm{\bx_i-\by_{i-1}}_{\infty}> R \sqrt{a_i-b_{i-1}}  \,\Big\} \, .
  \end{split}
\end{equation*}
Note that then we have that
\begin{equation*}
  H^{(2)}_{r,N}=H^{\mathsf{(diff)}}_{r,N,R}+H^{\mathsf{(superdiff)}}_{r,N,R} \, .
  \end{equation*}
We have the following Proposition.
\begin{proposition} \label{dif_proposition}
  For all $r\geq 1$ we have that
  \begin{equation} \label{dif_trunc}
    \lim_{R \to \infty} \sup_{N \in \N} H^{\mathsf{(superdiff)}}_{r,N,R}
    =0 \, .
  \end{equation}
\end{proposition}
\begin{proof}
  We use the bounds established in Lemma \ref{sdif_lemma}, below, and \eqref{crude_bounds} to show \eqref{dif_trunc}.
  We can use a union bound for \eqref{H_sdiff} to obtain that
  \begin{equation} \label{sdif_1}
    \begin{split}
      & H^{\mathsf{(superdiff)}}_{r,N,R}
      =\hspace{-0.2cm}  \sum_{(I_1,...,I_r) \,\in\, \mathcal{I}^{(2)}} \,\sum_{(\vec{a},\vec{b},\vec{\bx},\vec{\by})\, \in \,  \sfC^{\mathsf{(superdiff)}}_{r,N,R}
      } \sfU^{I_1}_{b_1}(0,\by_1) \prod_{i=2}^r Q_{a_{i}-b_{i-1}}^{I_{i-1};I_i}(\by_{i-1},\bx_i) \,  \sfU^{I_i}_{b_i-a_i}(\bx_i,\by_i) \,\sigma_N(I_i) \\
      & \leq  \sum_{j=1}^r \sum_{(I_1,...,I_r) \,\in\, \mathcal{I}^{(2)}}  \sumtwo{0:=a_1\leq b_1<a_2\leq \dots < a_r\leq b_r \leq N,}{0:=\bx_1,\by_1,\dots,\bx_r,\by_r \in (\Z^2)^h}  \bigg(\ind_{\big\{\norm{\by_j-\bx_j}_{\infty}>R \sqrt{b_j-a_j}  \big\}}+ \ind_{\big\{\norm{\bx_j-\by_{j-1}}_{\infty}>R\sqrt{a_j-b_{j-1}} \big\}}\bigg) \\
      & \qquad \times \sfU^{I_1}_{b_1}(0,\by_1) \prod_{i=2}^r Q_{a_{i}-b_{i-1}}^{I_{i-1};I_i}(\by_{i-1},\bx_i) \,  \sfU^{I_i}_{b_i-a_i}(\bx_i,\by_i) \,\sigma_N(I_i) \, .
    \end{split}
  \end{equation}
  We split the sum on the last two lines of \eqref{sdif_1} according to the two indicator functions that appear therein.
  By repeated successive application of the bounds from \eqref{crude_bounds} for $j<i\leq r$ and then by using \eqref{sd2}, which reads as
  \begin{equation*}
    \sum_{\by_j \in (\Z^2)^h, \, a_j\leq b_j\leq N} \sfU^{I_j}_{b_j-a_j}(\bx_j,\by_j)\, \ind_{\big\{\norm{\by_j-\bx_j}_{\infty}>R \sqrt{b_j-a_j}  \big\}}\leq e^{-\kappa R} \, ,
  \end{equation*}
  we deduce that
  \begin{equation} \label{first_ind_1}
    \begin{split}
      & \sumtwo{0:=a_1\leq b_1<a_2\leq \dots < a_r\leq b_r \leq N,}{0:=\bx_1,\by_1,\dots,\bx_r,\by_r \in (\Z^2)^h}
      \ind_{\big\{\norm{\by_j-\bx_j}_{\infty}>R \sqrt{b_j-a_j}  \big\}}
      \sfU^{I_1}_{b_1}(0,\by_1) \prod_{i=2}^r Q_{a_{i}-b_{i-1}}^{I_{i-1};I_i}(\by_{i-1},\bx_i) \,  \sfU^{I_i}_{b_i-a_i}(\bx_i,\by_i) \,\sigma_N(I_i) \\
      & \leq e^{-\kappa R}\bigg(\frac{\bar{\beta}'}{1-\bar{\beta}'}\bigg)^{r-j}\\
      &\qquad  \times  \sumtwo{0:=a_1\leq b_1<a_2\leq \dots<b_{j-1} < a_j \leq N,}{0:=\bx_1,\by_1,\dots,\bx_j \in (\Z^2)^h} \sfU^{I_1}_{b_1}(0,\by_1) \prod_{i=2}^{j-1} Q_{a_{i}-b_{i-1}}^{I_{i-1};I_i}(\by_{i-1},\bx_i) \,  \sfU^{I_i}_{b_i-a_i}(\bx_i,\by_i) \,\sigma_N(I_i) \\
      & \qquad \qquad \times Q^{I_{j-1};I_j}_{a_j-b_{j-1}}(\by_{j-1},\bx_j)\,  \sigma_N(I_j) \, .
    \end{split}
  \end{equation}
  We then continue the summation using the bounds from \eqref{crude_bounds}, to obtain that the right-hand side of the inequality in \eqref{first_ind_1} is bounded by
  \begin{equation*}
    e^{-\kappa R} \bigg(\frac{\bar{\beta}'}{1-\bar{\beta}'}\bigg)^{r-j} \, \bigg(\frac{\bar{\beta}'}{1-\bar{\beta}'}\bigg)^{j-1} =  e^{-\kappa R} \bigg(\frac{\bar{\beta}'}{1-\bar{\beta}'}\bigg)^{r-1}\, .
  \end{equation*}
  Similarly, for the sum involving the second indicator function in \eqref{sdif_1} we obtain by using \eqref{crude_bounds} and \eqref{sd1} of Lemma \ref{sdif_lemma} that
  \begin{alignat*}{2}
     & \sumtwo{0:=a_1\leq b_1<a_2\leq \dots < a_r\leq b_r \leq N\, ,}{0:=\bx_1,\by_1,\dots,\bx_r,\by_r \in (\Z^2)^h}
     &                                                                                                               & \ind_{\big\{\norm{\bx_j-\by_{j-1}}_{\infty}>R\sqrt{a_j-b_{j-1}} \big\}}
    \sfU^{I_1}_{b_1}(0,\by_1)                                                                                                                                                                                                                                                                                                \\
     & \qquad                                                                                                        &                                                                         & \times \prod_{i=2}^r Q_{a_{i}-b_{i-1}}^{I_{i-1};I_i}(\by_{i-1},\bx_i) \,  \sfU^{I_i}_{b_i-a_i}(\bx_i,\by_i) \,\sigma_N(I_i) \\
     & \leq e^{-\kappa R^2}\,  \bigg(\frac{1}{1-\bar{\beta}'}\bigg)^{r} \, (\bar{\beta}')^{r-1} \, .
  \end{alignat*}
  Therefore, the right-hand side of the inequality in \eqref{sdif_1} is bounded by
  \begin{equation*}
    \sum_{j=1}^r \sum_{(I_1,...,I_r) \, \in \cI^{(2)}}\, e^{-\kappa R} \Bigg( \bigg(\frac{\bar{\beta}'}{1-\bar{\beta}'}\bigg)^{r-1} + \bigg(\frac{1}{1-\bar{\beta}'}\bigg)^{r} \, (\bar{\beta}')^{r-1} \Bigg)\leq e^{-\kappa R} \Bigg(2r\cdot  \frac{(\bar{\beta}')^{r-1}}{(1-\bar{\beta}')^{r}}\cdot  \binom{h}{2}^r \,\Bigg) \, ,
  \end{equation*}
  where the $\binom{h}{2}^r$ factor comes from the fact that there are at most $\binom{h}{2}^r$ choices for the sequence $(I_1,\dots,I_r) \in \cI^{(2)}$.
  Thus, recalling \eqref{sdif_1} we get that
  \begin{equation*}
    \sup_{N \in \N} H^{\mathsf{(superdiff)}}_{r,N,R} \leq e^{-\kappa R} \Bigg(2r\cdot  \frac{(\bar{\beta}')^{r-1}}{(1-\bar{\beta}')^{r}}\cdot  \binom{h}{2}^r \,\Bigg) \xrightarrow{R \to \infty} 0.
  \end{equation*}

\end{proof}

\begin{lemma} \label{sdif_lemma}
  Let $I,J \vdash \{1,...,h\}$ such that $|I|=|J|=h-1$ and $I \neq J$. For large enough $R \in (0,\infty)$ and uniformly in $\bx \in (\Z^2)_I^h$ we have that for a constant $\kappa=\kappa(h,\bar{\beta}) \in (0,\infty)$,
  \begin{equation} \label{sd1}
    \sigma_N(J) \cdot \Bigg(\sum_{1\leq n \leq N, \, \by \, \in (\Z^2)^h} Q^{I;J}_n(\bx,\by)\cdot \ind_{\big \{ \norm{\bx-\by}_{\infty}>R \sqrt{n}\big \}}\Bigg) \leq e^{-\kappa R^2}
  \end{equation}
  and
  \begin{equation} \label{sd2}
    \sum_{1 \leq n \leq N, \, \by \, \in (\Z^2)^h} \sfU^J_n(\bx,\by)\cdot \ind_{\big \{\norm{\bx-\by}_{\infty}>R \sqrt{n}\big \}}\leq e^{-\kappa R} \, .
  \end{equation}
\end{lemma}

\begin{proof}
  We start with the proof of \eqref{sd1}.
  Since $|J|=h-1$, let us assume without loss of generality that $J=\{k,\ell\} \sqcup \bigsqcup_{j \neq k,\ell}\{j\}$. In this case, $Q^{I;J}_n(\bx,\by)$ contains
   $h-2$ random walk jumps with free endpoints $y^{(j)}$, $j \neq k,\ell$, that is
  \begin{equation*}
    \prod_{j \neq k,\ell} q_n(y^{(j)}-x^{(j)}) \, .
  \end{equation*}
  Moreover, $J$ imposes the constraint that $y^{(k)}=y^{(\ell)}$, which appears in $Q^{I;J}_n(\bx,\by)$ through the product of transition kernels
  \begin{equation*}
    q_n(y^{(k)}-x^{(k)})\cdot q_n(y^{(k)}-x^{(\ell)}) \, ,
  \end{equation*}
  recall \eqref{free_evol}. The constraint $\norm{\bx-\by}_{\infty}>R \sqrt{n}$ implies that there exists $1\leq j \leq h$ such that $|x^{(j)}-y^{(j)}| > R\, \sqrt{n}$. We distinguish two cases:
  \begin{enumerate}
    \item There exists $j \neq k,\ell$ such that $|x^{(j)}-y^{(j)}| > R\, \sqrt{n}$, or

    \item $|x^{(j)}-y^{(j)}| > R\, \sqrt{n}$ for $j=k$ or $j=\ell$.
  \end{enumerate}

  In both cases, we can use $\sum_{z \in \Z^2} q_n(z)=1$ 
  to sum the kernels $q_n(y^{(i)}-x^{(i)})$ with $i \notin\{ j, k,\ell \}$ 
  to which we do not impose any super-diffusive constraints.
  By symmetry and translation invariance we can upper bound the left-hand side of \eqref{sd1} by
  \begin{equation} \label{Q_sdif_bound}
    \begin{split}
      \sigma_N(J) \cdot \Bigg((h-2) \sum_{1\leq n \leq N, \, z \in \Z^2}q_n(z) \cdot \ind_{\{|z|>R\sqrt{n}\}} &\cdot \Big\{ \sup_{u \, \in \Z^2}\sum_{ \, z \in \Z^2} q_n(z)q_n(z+u)\Big\} \\
      + 2\, \sup_{u \, \in \Z^2}\sum_{1 \leq n \leq N, \, z \in \Z^2}& q_n(z)q_n(z+u) \cdot \ind_{\{|z|>R \sqrt{n}\}}\Bigg) \, .
    \end{split}
  \end{equation}
  Looking at the first summand in \eqref{Q_sdif_bound} we have by Cauchy-Schwarz that
  \begin{equation} \label{CS}
    \begin{split}
      \sup_{u \in \Z^2} \sum_{z \in \Z^2} q_n(z)\,q_n(z+u)  \leq \Big(\sum_{z \in \Z^2} q_n^2(z) \Big )^{1/2}\cdot  \sup_{u \in \Z^2} \Big(\sum_{z \in \Z^2}q^2_n(z+u)\Big)^{1/2} \leq q_{2n}(0) \, ,
    \end{split}
  \end{equation}
  since $\sum_{z \in \Z^2}q^2_{n}(z)=q_{2n}(0)$. Let us recall the
   deviation estimate for the simple random walk, which can be found in \cite{LL10}, Proposition 2.1.2, that is
  \begin{equation} \label{moderate_devation_estimate}
    \P\Big(\max_{0\leq k \leq n}|S_k|>R \sqrt{n}\Big)\leq e^{-cR^2}\, ,
  \end{equation}
  for a constant $c\in (0,\infty)$ and all $R \in (0,\infty)$, large enough.
  By using bound \eqref{CS} and subsequently \eqref{moderate_devation_estimate} on the first summand of \eqref{Q_sdif_bound} we get that
  \begin{equation*}
    \begin{split}
      & \sum_{1\leq n \leq N, \, z \in \Z^2}q_n(z) \cdot \ind_{\{|z|>R\sqrt{n}\}} \cdot \Big\{ \sup_{u \, \in \Z^2}\sum_{ \, z \in \Z^2} q_n(z)q_n(z+u)\Big\}  \\
      \leq  &\sum_{1\leq n \leq N,\, z \in \Z^2} q_{2n}(0)\cdot q_n(z) \ind_{\{|z|> R \sqrt{n}\}}\\
      \leq &\, e^{-c R^2} R_N \, .
    \end{split}
  \end{equation*}
  We recall from \eqref{R_N} and \eqref{R_N_asymp} that $R_N=\sum_{n=1}^N q_{2n}(0) \stackrel{N \to \infty}{\approx} \frac{\log N}{\pi}$, therefore,
  \begin{equation} \label{Q_sdif_bound_1}
    \begin{split}
      & \sigma_N(J) \cdot \Bigg((h-2) \sum_{1\leq n \leq N, \, z \in \Z^2}q_n(z) \cdot \ind_{\{|z|>R\sqrt{n}\}}  \cdot \Big\{ \sup_{u \, \in \Z^2}\sum_{ \, z \in \Z^2} q_n(z)q_n(z+u)\Big\} \Bigg) \\
      \leq &\,(h-2) \,\sigma_N(J) \,R_N\, e^{-cR^2}\\
      \leq &\, (h-2)\, \bar{\beta}' \, \,e^{-cR^2} \, ,
    \end{split}
  \end{equation}
  for some $\bar\beta' \in (\bar\beta,1)$.
  The second summand in the parenthesis in \eqref{Q_sdif_bound} can be bounded via Cauchy-Schwarz by
  \begin{equation} \label{Q_sdif_bound_2}
    \Bigg( \sum_{1 \leq n \leq N, \, z \in \Z^2 } q^2_n(z) \cdot \ind_{\{|z|>R\sqrt{n}\}}\Bigg)^{\frac{1}{2}} \cdot \Bigg( \sum_{1 \leq n \leq N, \, z \in \Z^2 } q^2_n(z+u) \Bigg)^{\frac{1}{2}} \, .
  \end{equation}
  For the first term in \eqref{Q_sdif_bound_2}, using that $\sup_{z \in \Z^2 }q_n(z)\leq \frac{C}{n}$ we get
  \begin{equation} \label{Q_sdif_bound_3}
    \begin{split}
      \sum_{1\leq n \leq N, \, z \in \Z^2} q^2_n(z) \cdot \ind_{\{|z|>R \sqrt{n}\}}\leq  \, C \sum_{1\leq n \leq N, \, z \in \Z^2} \frac{q_n(z)}{n}\cdot \ind_{\{|z|>R \sqrt{n}\}}
      \leq  \, C \, e^{-c R^2} \, \log N
    \end{split}
  \end{equation}
  For the second term in \eqref{Q_sdif_bound}, we have that for all $u \in \Z^2$
  \begin{equation*}
    \sum_{1\leq n \leq N,\,  z \in \Z^2 } q^2_n(z+u)= \sum^N_{n=1} q_{2n}(0)\stackrel{N \to \infty}{\approx} \frac{\log N}{\pi} \, .
  \end{equation*}
  Thus, by \eqref{Q_sdif_bound_2} together with \eqref{Q_sdif_bound_3} we conclude that
  for the second summand in \eqref{Q_sdif_bound} we have
  \begin{equation*}
    \sigma_N(J) \cdot \Bigg( 2\, \sup_{u \, \in \Z^2}\sum_{1 \leq n \leq N, \, z \in \Z^2} q_n(z)q_n(z+u) \cdot \ind_{\{|z|>R \sqrt{n}\}}\Bigg) \leq C \, e^{-\frac{cR^2}{2}} \, .
  \end{equation*}
  Therefore, recalling \eqref{Q_sdif_bound_1} we deduce that there exists a constant $\kappa(h,\bar{\beta}) \in (0,\infty)$ such that
  \begin{equation*}
    \sigma_N(J) \cdot \Bigg(\sum_{1\leq n \leq N, \, \by \, \in (\Z^2)^h} Q^{I;J}_n(\bx,\by)\cdot \ind_{\big \{ \norm{\bx-\by}_{\infty}>R \sqrt{n}\big \}}\Bigg) \leq e^{-\kappa R^2} \, .
  \end{equation*}
  We move to the proof of \eqref{sd2}. Similar to the proof of \eqref{sd1}, we can bound the left-hand side of \eqref{sd2} by
  \begin{equation} \label{U_sdif_bounds}
    (h-1) \sum_{1 \leq n \leq N, \,w, \, z \in \Z^2}U_N^{\bar{\beta}}(n,w)\cdot q_n(z)  \ind_{\{|z|>R \sqrt{n}\}} + \sum_{1\leq n \leq N, z \in \Z^2} U^{\bar{\beta}}_N(n,z) \cdot \ind_{\{ |z|>R \sqrt{n}\}} \, .
  \end{equation}
  For the first summand in \eqref{U_sdif_bounds}, by \eqref{moderate_devation_estimate} we have that
  \begin{equation*}
    \sum_{z \in \Z^2} q_n(z)\ind_{\{|z|>R \sqrt{n}\}}\leq e^{-cR^2} \, ,
  \end{equation*}
  and $\sum_{1\leq n \leq N,\, w \in \Z^2} U_N^{\bar{\beta}}(n,w)\leq \frac{1}{1-\bar{\beta}'}$, therefore
  \begin{equation} \label{U_sdif_1}
    (h-1) \sum_{1 \leq n \leq N, \,w, \, z \in \Z^2}U_N^{\bar{\beta}}(n,w)\cdot q_n(z)  \ind_{\{|z|>R \sqrt{n}\}}\leq \frac{h-1}{1-\bar{\beta}'} \, e^{-c R^2} \, .
  \end{equation}
  For the second summand, we use the renewal representation of $U^{\bar{\beta}}_N(\cdot,\cdot)$ introduced in \eqref{ren_rep}.
  In particular, we have that
  \begin{equation} \label{sdif_renewal}
    \sum_{1\leq n \leq N, z \in \Z^2} U^{\bar{\beta}}_N(n,z) \cdot \ind_{\{ |z|>R \sqrt{n}\}}=\sum_{k\geq 0} 
    (\sigma_N(\bar\beta)R_N)^k \sum_{n=0}^N \P\Big(\big|S^{\ms (N)}_k \big|>R\sqrt{n},\tau^{\ms (N)}_k=n\Big) \, .
  \end{equation}
  Then, by conditioning on the times $(T^{\ms (N)}_i)_{1 \leq i \leq k}$ for which $\tau_k^{\ms(N)}=T^{\ms(N)}_1+\dots+T^{\ms (N)}_k$ we have that
  \begin{equation} \label{sdif_cond}
    \begin{split}
      \P\Big(\big|S^{\ms (N)}_k \big|>& R\sqrt{n},\, \tau^{\ms (N)}_k=n\Big)\\
      &=\sum_{n_1+\dots+n_k=n} \P\Big(\big|S^{\ms (N)}_k \big|>R\sqrt{n}\, \Big | \cap_{i=1}^k\big \{T^{\ms(N)}_i=n_i \big\}\Big)\, \prod_{i=1}^k \P\big(T^{\ms (N)}_i=n_i\big) \, .
    \end{split}
  \end{equation}
  Note that when we condition on $\cap_{i=1}^k\big \{T^{(\ms N)}_i=n_i\big \}$, $S_k^{\ms (N)}$ is a sum of $k$ independent random variables
   $(\xi_i)_{1 \leq i \leq k}$ taking values in $\Z^2$, with law 
  \begin{equation*}
    \P(\xi_i=x)=\frac{q^2_{n_i}(x)}{q_{2n_i}(0)} \, .
  \end{equation*}
  The proof of Proposition 3.5 in \cite{LZ21} showed that there exists a constant $C \in (0,\infty)$ such that for all $\lambda \geq 0$
  \begin{equation} \label{mom_ineq}
    \E\Big[e^{\lambda |\sum_{i=1}^k \xi_i|}\Big] \leq 2e^{4C\lambda^2 n} \, ,
  \end{equation}
  uniformly over the values $n_1,...,n_k$.
  Therefore, by \eqref{mom_ineq} with $\lambda=\frac{1}{\sqrt{n}}$  and Markov's inequality we obtain that
  \begin{equation*}
    \P\Big(\big|S^{\ms (N)}_k \big|>R\sqrt{n}\, \Big | \cap_{i=1}^k\{T^{\ms(N)}_i=n_i\}\Big) \leq 2e^{4C-R} \, .
  \end{equation*}
  Thus, looking back at \eqref{sdif_cond} we have that for all $k \geq 0$,
  \begin{equation*}
    \P\Big(\big|S^{\ms (N)}_k \big|> R\sqrt{n},\, \tau^{\ms (N)}_k=n\Big) \leq 2e^{4C-R}\,\, \P(\tau^{\ms (N)}_k=n) \, ,
  \end{equation*}
  therefore, plugging the last inequality into \eqref{sdif_renewal}, we get that
  \begin{equation} \label{U_sdif_2}
    \sum_{1\leq n \leq N, z \in \Z^2} U^{\bar{\beta}}_N(n,z) \cdot \ind_{\{ |z|>R \sqrt{n}\}} \leq 2e^{4C-R} \sum_{n=1}^N U^{\bar{\beta}}_N(n) \leq \frac{2e^{4C-R}}{1-\bar{\beta}'} \, ,
  \end{equation}
  therefore by \eqref{U_sdif_1} and \eqref{U_sdif_2} we have that there exists a constant $\kappa(h,\bar \beta)\in (0,\infty)$ such that
  \begin{equation*}
    \sum_{1 \leq n \leq N, \, \by \, \in (\Z^2)^h} \sfU^J_n(\bx,\by)\cdot \ind_{\big \{\norm{\bx-\by}_{\infty}>R \sqrt{n}\big \}}\leq e^{-\kappa R} \, ,
  \end{equation*}
  for large enough $R \in (0,\infty)$, thus concluding the proof of \eqref{sd2}.
\end{proof}

\subsection{Scale separation.}\label{step_3}

In this step we show that given $r \in \N $, $r \geq 2$, the main contribution to $H^{\mathsf{(diff)}}_{r,N}$ comes from
configurations where $a_{i+1}-b_i>M(b_i-b_{i-1})$ for all $1\leq i \leq r$ and large $M$, as $N\to \infty$.
Recall from \eqref{C_dif_set} that
\begin{equation*}
  H^{\mathsf{(diff)}}_{r,N,R}=\sum_{(I_1,...,I_r) \,\in\, \mathcal{I}^{(2)}} \,\sum_{(\vec{a},\vec{b},\vec{\bx},\vec{\by})\, \in \,\sfC^{\mathsf{(diff)}}_{r,N,R}
  } \sfU^{I_1}_{b_1}(0,\by_1) \prod_{i=2}^r Q_{a_i-b_{i-1}}^{I_{i-1};I_i}(\by_{i-1},\bx_i) \,  \sfU^{I_i}_{b_i-a_i}(\bx_i,\by_i) \cdot\sigma_N(I_i) \, .
\end{equation*}
Define the set
\begin{equation} \label{C_main_set}
  \sfC^{\mathsf{(main)}}_{r,N,R,M}:=\sfC^{\mathsf{(diff)}}_{r,N,R}
  \cap \Big\{ (\vec{a},\vec{b},\vec{\bx},\vec{\by}):\, a_{i+1}-b_i>M(b_i-b_{i-1})\text{ for all } 1\leq i\leq r-1\,  \Big\} \, ,
\end{equation}
with the convention $b_0:=0$ and
accordingly define
\begin{equation} \label{H_main}
  H^{\mathsf{(main)}}_{r,N,R,M}:=\hspace{-0.2cm}\sum_{(I_1,...,I_r) \,\in\, \mathcal{I}^{(2)}} \sum_{(\vec{a},\vec{b},\vec{\bx},\vec{\by})\, \in \,\sfC^{\mathsf{(main)}}_{r,N,R,M}} \sfU^{I_1}_{b_1}(0,\by_1) \prod_{i=2}^r Q_{a_i-b_{i-1}}^{I_{i-1};I_i}(\by_{i-1},\bx_i)  \sfU^{I_i}_{b_i-a_i}(\bx_i,\by_i) \,\sigma_N(I_i) \, .
\end{equation}
We then have the following approximation proposition:
\begin{proposition} \label{M_proposition}
  For all fixed $r\in \N$, $r \geq 2$ and $M \in (0,\infty)$,
  \begin{equation} \label{main_approx}
    \lim_{N \to \infty}  \sup_{R \, \in (0,\infty)}\Big|H^{\mathsf{(diff)}}_{r,N,R}-H^{\mathsf{(main)}}_{r,N,R,M}\Big|=0\, .
  \end{equation}
\end{proposition}

\begin{proof}
  Fix $M>0$. Let us begin by showing \eqref{main_approx} for the simplest case which is $r=2$. We have
  \begin{equation*}
    \begin{split}
      H^{\mathsf{(diff)}}_{2,N,R}&- H^{\mathsf{(main)}}_{2,N,R,M}\\
      \leq &\sum_{(I_1,I_2) \in \,\mathcal{I}^{(2)}}\,\, \sumtwo{0 \leq b_1 <a_2\leq N, \,a_2-b_1\leq M b_1,}{\by_1,\bx_2,\by_2 \in (\Z^2)^h} \sfU^{I_1}_{b_1}(0,\by_1)
      Q_{a_2-b_{1}}^{I_{1};I_2}(\by_{1},\bx_2)  \,\sigma_N(I_2)  \sfU^{I_2}_{b_2-a_2}(\bx_2,\by_2).
    \end{split}
  \end{equation*}
  We can bound $\sigma_N(I_2)$ by
  $\frac{\pi\bar\beta\,'}{\log N}$, for some $\bar\beta\,'\in (\bar\beta,1)$
  and use \eqref{crude_bounds} to bound the last replica, i.e. the sum over $(b_2,\by_2)$, thus getting
  \begin{equation} \label{reduction_1}
    H^{\mathsf{(diff)}}_{2,N,R}- H^{\mathsf{(main)}}_{2,N,R,M}
    \leq
    \frac{\pi\bar\beta\,'(1-\bar\beta\,')^{-1}}{\log N}  \sum_{(I_1,I_2) \in\, \mathcal{I}^{(2)}} \, 
    \sumtwo{0 \leq b_1 <a_2\leq N, \,a_2-b_1\leq M b_1 \, ,}{\by_1,\bx_2\in (\Z^2)^h}
    \sfU^{I_1}_{b_1}(0,\by_1) Q_{a_2-b_{1}}^{I_1;I_2}(\by_{1},\bx_2)  \, .
  \end{equation}
  Notice that at this stage we can sum out the spatial endpoints of the free kernels in \eqref{reduction_1} and bound the coupling strength $\beta_{k,\ell}$ of any replica $\sfU^{I_1}_{b_1}(0,\by_1)$ with $I_1=\{k,\ell\}\sqcup \bigsqcup_{j\neq k,\ell}\{j\}$ by $\bar\beta$ to obtain
  \begin{align} \label{2_main}
    H^{\mathsf{(diff)}}_{2,N,R} & - H^{\mathsf{(main)}}_{2,N,R,M} \notag                                                                             \\
                                & \leq \frac{\pi\bar\beta\,'(1-\bar\beta\,')^{-1}}{\log N} \sum_{(I_1,I_2) \in \, \mathcal{I}^{(2)}} \,
    \sumtwo{0 \leq b_1 <a_2\leq N, \,a_2-b_1\leq M b_1 \, ,}{y_1,x_2 \in \Z^2} U^{\bar{\beta}}_N(b_1,y_1) q_{a_2-b_1}(x_2-y_1)\, q_{a_2}(x_2) \notag \\
                                & \leq \frac{\pi\bar\beta\,'(1-\bar\beta\,')^{-1}}{\log N} \binom{h}{2}^2
    \sumtwo{0 \leq b_1 <a_2\leq N, \,a_2-b_1\leq M b_1 \, ,}{y_1,x_2 \in \Z^2} U^{\bar{\beta}}_N(b_1,y_1) q_{a_2-b_1}(x_2-y_1)\, q_{a_2}(x_2) \, .
  \end{align}
  For the last inequality we also have used that the number of possible partitions $(I_1,I_2)\in \cI^{(2)}$ is bounded by $\binom{h}{2}^2$. For every fixed value of $b_1$ in \eqref{2_main}, we use Cauchy-Schwarz for the sum over $(a_2,x_2) \in (0,(1+M)b_1]\times \Z^2$ in \eqref{2_main} to obtain that
  \begin{equation} \label{ss_CS}
    \begin{split}
      \sum_{b_1<a_2\leq (1+M) b_1, \, x_2 \in \Z^2}& q_{a_2-b_1}(x_2-y_1)\,  q_{a_2}(x_2)\\
      \leq &\Bigg(\sum_{0<a_2\leq (1+M) b_1,\, x_2 \in \Z^2} q^2_{a_2-b_1}(x_2-y_1) \Bigg)^{\frac{1}{2}} \Bigg( \sum_{b_1<a_2\leq (1+M) b_1,\,  x_2 \in \Z^2} q^2_{a_2}(x_2)\Bigg)^{\frac{1}{2}} \\
      = &\Bigg(\sum_{0<a_2\leq (1+M) b_1} q_{2(a_2-b_1)}(0) \Bigg)^{\frac{1}{2}} \Bigg( \sum_{b_1<a_2\leq (1+M) b_1} q_{2a_2}(0)\Bigg)^{\frac{1}{2}} \, .
    \end{split}
  \end{equation}
  We can bound the leftmost parenthesis in the last line of \eqref{ss_CS} by 
  $R^{1/2}_N=\Big(\sum_{n=1}^N q_{2n}(0)\Big)^{1/2}=O(\sqrt{\log N})$. For the other term we have
  \begin{equation} \label{ss_logM}
    \sum_{b_1<a_2\leq(1+M)b_1} q_{2a_2}(0) \leq c\, \sum_{b_1<a_2\leq(1+M)b_1} \tfrac{1}{a_2}\leq c \log(1+M) \, .
  \end{equation}
  Therefore, using \eqref{ss_CS} and \eqref{ss_logM} along with $\sum_{0\leq b_1\leq N,\,y_1 \in \Z^2} U^{\bar\beta}_N(b_1,y_1) \leq (1-\bar\beta\,')^{-1}$ in \eqref{2_main} we obtain that
  \begin{equation*} 
    H^{\mathsf{(diff)}}_{2,N,R}- H^{\mathsf{(main)}}_{2,N,R,M} \leq C \pi\bar\beta\,'(1-\bar\beta\,')^{-2}  \sqrt{\frac{\log(1+M)}{\log N}} \xrightarrow{N \to \infty} 0 \, .
  \end{equation*}

  Let us show how this argument can be extended to work for general $r \in \N$. The key observation is that for every fresh collision between two random walks, that is $I_{i+1}=\{k,\ell\} \sqcup \bigsqcup_{j \neq k,\ell} \{j\} $, happening at time $0<a_{i+1}\leq N$, we have $I_i \neq I_{i+1}$, therefore one of the two colliding walks with labels $k,\ell$ has to have travelled freely, for time at least $a_{i+1}-b_{i-1}$ from its previous collision. More precisely, every term in the expansion of $ H^{\mathsf{(diff)}}_{r,N,R} - H^{\mathsf{(main)}}_{r,N,R,M}$ contains for every $1\leq i\leq r-1$ a product of the form
  \begin{equation*}
    q_{a_{i+1}-b_i}(x_{i+1}-y_i)\cdot  q_{a_{i+1}-b_{i-1}}(x_{i+1}-y_{i-1})\, ,  \end{equation*}
  see Figure \ref{fig:scales}.
  Recall from \eqref{C_set} and \eqref{H_main} that we have the expansion
  \begin{equation} \label{dif_minus_main}
    \begin{split}
      &H^{\mathsf{(diff)}}_{r,N,R} - H^{\mathsf{(main)}}_{r,N,R,M}\\
      &=\sum_{(I_1,...,I_r) \,\in\, \mathcal{I}^{(2)}} \,\sum_{(\vec{a},\vec{b},\vec{\bx},\vec{\by})\, \in \,\sfC^{\mathsf{(diff)}}_{r,N,R}\smallsetminus \sfC^{\mathsf{(main)}}_{r,N,R,M}} \sfU^{I_1}_{b_1}(0,\by_1) \prod_{i=2}^r Q_{a_i-b_{i-1}}^{I_{i-1};I_i}(\by_{i-1},\bx_i) \,  \sfU^{I_i}_{b_i-a_i}(\bx_i,\by_i) \,\sigma_N(I_i) \, ,
    \end{split}
  \end{equation}
  where by definition \eqref{C_main_set} we have that
  \begin{equation} \label{Cdiff_minus}
    \sfC^{\mathsf{(diff)}}_{r,N,R}\smallsetminus \sfC^{\mathsf{(main)}}_{r,N,R,M}=  \sfC^{\mathsf{(diff)}}_{r,N,R}\,
    \cap \, \bigcup_{i=1}^{r-1}\Big\{ (\vec{a},\vec{b},\vec{\bx},\vec{\by}):\, a_{i+1}-b_i\leq M(b_i-b_{i-1})\,  \Big\} \, .
  \end{equation}
 We will start the summation of \eqref{dif_minus_main} from the end until we find the index $1\leq i \leq r-1$ for which the sum over $a_{i+1}$ is restricted to $\big(b_i,b_i+M(b_i-b_{i-1})\big ]$, in agreement with \eqref{Cdiff_minus}, using \eqref{crude_bounds} to bound the contribution of the sums over $b_{j},a_{j+1}$ and the corresponding spatial points for $i<j\leq r-1$.
  Next, notice that we can bound the contribution of the sum over $a_{i+1} \in \big(b_i,b_i+M(b_i-b_{i-1})\big ]$ and $x_{i+1} \in \Z^2$,
  using a change of variables, by a factor of
  \begin{equation*}
    \frac{C}{\log N}\Bigg(\sup_{1\leq t\leq N,\, u \in \Z^2} \sum_{1\leq n \leq Mt ,\,  z \in \Z^2} q_{n}(z)q_{n+t}(z+u)\Bigg)\leq C \sqrt{\frac{\log (1+M)}{\log N}}\, ,
  \end{equation*}
  using Cauchy-Schwarz as in \eqref{ss_CS} and \eqref{ss_logM}. The remaining sums over $b_j,a_{j-1}$, $1\leq j\leq i$ can be bounded again via \eqref{crude_bounds}.
  Therefore, taking into account that by \eqref{Cdiff_minus} there are $r-1$ choices for the index $i$ such that the sum over $a_{i+1}$ is restricted to $\big(b_i,b_i+M(b_i-b_{i-1})\big ]$, we can give an upper bound to $H^{\mathsf{(diff)}}_{r,N,R}- H^{\mathsf{(main)}}_{r,N,R,M}$ as follows:
  \begin{equation*}
    H^{\mathsf{(diff)}}_{r,N,R}- H^{\mathsf{(main)}}_{r,N,R,M}\leq C (r-1) \, \binom{h}{2}^r \,
    \Big(\frac{\bar{\beta}\,'}{1-\bar{\beta}\,'}\Big)^r  \sqrt{\frac{\log(1+M)}{\log N}} \xrightarrow{N \to \infty} 0 \, ,
  \end{equation*}
  where we also used that the number of distinct sequences $(I_1,\dots,I_r) \in \cI^{(2)}$ is bounded by $\binom{h}{2}^r$.
\end{proof}

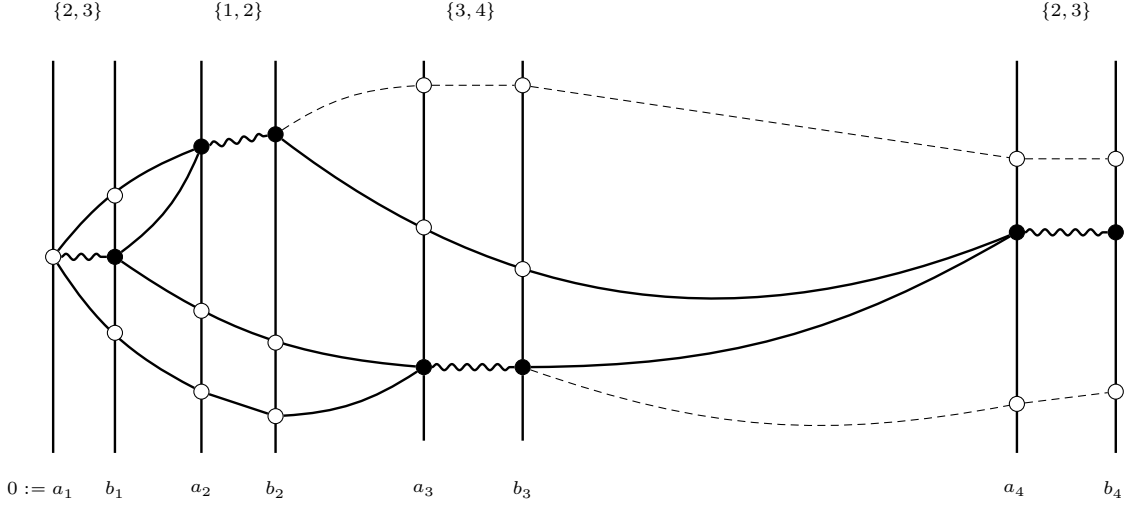
\begin{figure}
  \tikzstyle{filled_vertex}=[fill={black}, draw={black}, shape=circle,inner sep=2]
  \tikzstyle{new style 0}=[fill=white, draw={black}, shape=circle, inner sep=2pt]
  \tikzstyle{straight_line}=[-, fill=none, draw={black},line width=0.9pt]
  \tikzstyle{wiggly_line}=[-, fill=none, draw={black}, decorate=true, decoration={snake,amplitude=.4mm,segment length=2mm},line width=0.9pt]
  \tikzstyle{densely_dashed}= [dash pattern=on 3pt off 2pt]

  \tikzstyle{none}=[inner sep=0pt]
  \pgfdeclarelayer{edgelayer}
  \pgfdeclarelayer{nodelayer}
  \pgfsetlayers{edgelayer,nodelayer,main}
  \usetikzlibrary{decorations.pathmorphing}

  { \tiny

    \begin{tikzpicture}[scale=0.65]
      \begin{pgfonlayer}{nodelayer}
        \node [style=none] (1) at (0, 4) {};
        \node [style=none] (2) at (0, -4) {};
        \node [style=new style 0] (4) at (0, 0) {};
        \node [style={filled_vertex}] (5) at (19.5, 0.5) {};
        \node [style=none] (6) at (1.25, 4) {};
        \node [style=none] (7) at (1.25, -4) {};
        \node [style={filled_vertex}] (8) at (1.25, 0) {};
        \node [style={filled_vertex}] (9) at (3, 2.25) {};
        \node [style={filled_vertex}] (11) at (4.5, 2.5) {};
        \node [style=none] (14) at (3, 4) {};
        \node [style=none] (15) at (3, -4) {};
        \node [style=none] (16) at (4.5, 4) {};
        \node [style=none] (17) at (4.5, -4) {};
        \node [style={filled_vertex}] (18) at (21.5, 0.5) {};
        \node [style=none] (19) at (19.5, 4) {};
        \node [style=none] (20) at (21.5, 4) {};
        \node [style=none] (21) at (19.5, -4) {};
        \node [style=none] (22) at (21.5, -4) {};
        \node [style=new style 0] (27) at (3, -2.75) {};
        \node [style=new style 0] (28) at (4.5, -3.25) {};
        \node [style=none] (30) at (-0.25, -4.75) {$0:=a_1$};
        \node [style=none] (31) at (1.25, -4.75) {$b_1$};
        \node [style=none] (32) at (3, -4.75) {$a_2$};
        \node [style=none] (33) at (4.5, -4.75) {$b_2$};
        \node [style=none] (34) at (7.5, -4.75) {$a_3$};
        \node [style=none] (35) at (9.5, -4.75) {$b_3$};
        \node [style=new style 0] (36) at (1.25, 1.25) {};
        \node [style=new style 0] (37) at (1.25, -1.55) {};
        \node [style=new style 0] (54) at (19.5, -3) {};
        \node [style=none] (59) at (0.5, 5) {$\{2,3\}$};
        \node [style=none] (60) at (3.75, 5) {$\{1,2\}$};
        \node [style=none] (61) at (20.5, 5) {$\{2,3\}$};
        \node [style={filled_vertex}] (62) at (7.5, -2.25) {};
        \node [style=new style 0] (63) at (3, -1.10) {};
        \node [style=new style 0] (64) at (4.5, -1.75) {};
        \node [style={filled_vertex}] (65) at (9.5, -2.25) {};
        \node [style=none] (68) at (7.5, 4) {};
        \node [style=none] (69) at (9.5, 4) {};
        \node [style=none] (70) at (7.5, -3.75) {};
        \node [style=none] (71) at (9.5, -3.75) {};
        \node [style=none] (72) at (19.45, -4.75) {$a_4$};
        \node [style=none] (73) at (21.45, -4.75) {$b_4$};
        \node [style=none] (74) at (8.45, 5) {$\{3,4\}$};
        \node [style=new style 0] (75) at (7.5, 0.6) {};
        \node [style=new style 0] (76) at (9.5, -0.25) {};
        \node [style=new style 0] (77) at (7.5, 3.5) {};
        \node [style=new style 0] (78) at (9.5, 3.5) {};
        \node [style=new style 0] (79) at (19.5, 2) {};
        \node [style=new style 0] (80) at (21.5, 2) {};
        \node [style=new style 0] (81) at (21.5, -2.75) {};
      \end{pgfonlayer}
      \begin{pgfonlayer}{edgelayer}
        \draw [style={straight_line}] (1.center) to (2.center);
        \draw [style={straight_line}] (6.center) to (7.center);
        \draw [style={wiggly_line}] (4) to (8);
        \draw [style={straight_line}, bend left=15] (4) to (9);
        \draw [style={wiggly_line}] (9) to (11);
        \draw [style={straight_line}] (14.center) to (15.center);
        \draw [style={straight_line}] (16.center) to (17.center);
        \draw [style={wiggly_line}] (5) to (18);
        \draw [style={straight_line}] (19.center) to (21.center);
        \draw [style={straight_line}] (20.center) to (22.center);
        \draw [style={straight_line}, bend right=15] (4) to (27);
        \draw [style={straight_line}] (27) to (28);
        \draw [style={straight_line}, bend right=15] (8) to (9);
        \draw [style={straight_line}, bend right=15] (28) to (62);
        \draw [style={straight_line}, bend right=15] (8) to (62);
        \draw [style={wiggly_line}] (62) to (65);
        \draw [style={straight_line}, bend right] (11) to (5);
        \draw [style={straight_line}, bend right=15] (65) to (5);
        \draw [style={straight_line}] (68.center) to (62);
        \draw [style={straight_line}] (69.center) to (65);
        \draw [style={straight_line}] (62) to (70.center);
        \draw [style={straight_line}] (65) to (71.center);
        \draw [style={densely_dashed}, bend left=15] (11) to (77);
        \draw [style={densely_dashed}] (77) to (78);
        \draw [style={densely_dashed}] (78) to (79);
        \draw [style={densely_dashed}, bend right=15] (65) to (54);
        \draw [style={densely_dashed}] (54) to (81);
        \draw [style={densely_dashed}] (79) to (80);
      \end{pgfonlayer}
    \end{tikzpicture}

  }
  \caption{A diagramatic representation of a configuration of collisions between $4$ random walks in $H^{(2)}_{4,N}$ with $I_1=\{2,3\}$, $I_2=\{1,2\}$, $I_3=\{3,4\}$ and $I_4=\{2,3\}$. Wiggly lines represent replica evolution, see \eqref{replica_op}. }
  \label{fig:scales}
\end{figure}

\subsection{Rewiring.} \label{step_5}
Recall the expansion of $H^{\mathsf{(main)}}_{r,N,R,M}
$,
\begin{equation} \label{main_sum}
  H^{\mathsf{(main)}}_{r,N,R,M}
  :=\sum_{(I_1,...,I_r) \,\in\, \mathcal{I}^{(2)}} \,\sum_{(\vec{a},\vec{b},\vec{\bx},\vec{\by})\, \in \,\sfC^{\mathsf{(main)}}_{r,N,R,M}} \sfU^{I_1}_{b_1}(0,\by_1) \prod_{i=2}^r Q_{a_i-b_{i-1}}^{I_{i-1};I_i}(\by_{i-1},\bx_i) \,  \sfU^{I_i}_{b_i-a_i}(\bx_i,\by_i) \,\sigma_N(I_i)  \, .
\end{equation}

We also remind the reader that we may identify a partition $I=\{k,\ell\}\sqcup\bigsqcup_{j\neq k,\ell}\{j\}$ with its non-trivial part $\{k,\ell\}$. Moreover, if $\big(\{i_1,j_1\},\cdots,\{i_r,j_r\}\big) \in \cI^{(2)}$ we will use the notation
\begin{equation*} 
  \sfp(m):=\max \big\{ k<m: \,\{i_m,j_m\}=\{i_k,j_k\} \big\} \, ,
\end{equation*}
with the convention that $\sfp(m)=0$ if $\{i_k,j_k\} \neq \{i_m,j_m\}$ for all $1\leq k <m$.  Given this definition, the time $b_{\sfp(m)}$ represents the last time walks ${i_m,j_m}$ collided before their new collision at time $a_m$. Note that since we always have $\{i_k,j_k\} \neq \{i_{k+1},j_{k+1}\}$ by construction, $\sfp(m)<m-1$.

Consider a sequence of partitions $\big(\{i_1,j_1\},\dots,\{i_m,j_m\}\big) \in \cI^{(2)}$ and let $m \in \{2,\dots,r\}$. The goal of this step will be to show that we can make the replacement of weight
\begin{equation} \label{ker_repl}
  q_{a_{m}-b_{m-1}}\Big(x^{(i_m)}_m-y^{(i_m)}_{m-1}\Big)\cdot q_{a_{m}-b_{m-1}}\Big(x^{(j_m)}_m-y^{(j_m)}_{m-1}\Big) \quad\longleftrightarrow \quad q^2_{a_m-b_{\sfp(m)}}\Big(x^{(i_m)}_m-y^{(i_m)}_{\sfp(m)}\Big) \, .
\end{equation}
by inducing an error which is negligible when $M \to \infty$.
We iterate this procedure for all partitions $I_1,\dots,I_r$. We call the procedure described above {\bf rewiring}, see Figures \ref{fig:scales} and \ref{fig:rewiring}. The first step towards the full rewiring is to show the following lemma which quantifies the error of a single replacement \eqref{ker_repl}.

\begin{lemma} \label{lemma_ker_repl}
  Let $r\geq 2$ fixed and $m \in \{2,\dots,r\}$ with $I_m=\{i_m,j_m\}$. Then, for every fixed $R \in (0,\infty)$ and uniformly in $(\vec{a},\vec{b},\vec{\bx},\vec{\by})\, \in \,\sfC^{\mathsf{(main)}}_{r,N,R,M}$ and all sequences of partitions $(I_1,\dots,I_r) \in \cI^{(2)}$,
  \begin{equation} \label{loc_rew_error}
    q_{a_{m}-b_{m-1}}\Big(x^{(i_m)}_m-y^{(i_m)}_{m-1}\Big)\cdot q_{a_{m}-b_{m-1}}\Big(x^{(j_m)}_m-y^{(j_m)}_{m-1}\Big) = q^2_{a_m-b_{\sfp(m)}}\Big(x^{(i_m)}_m-y^{(i_m)}_{\sfp(m)}\Big)\cdot e^{o_{M}(1)} \, ,
  \end{equation}
  where $o_{M}(1)$ denotes a quantity such that $\lim_{M \to \infty}  o_{M}(1)=0$.
\end{lemma}
\begin{proof}
  We will show that
  \begin{equation} \label{one_each_time}
    \begin{split}
      q_{a_{m}-b_{m-1}}\Big(x^{(i_m)}_m-y^{(i_m)}_{m-1}\Big)&=q_{a_m-b_{\sfp(m)}}\Big(x^{(i_m)}_m-y^{(i_m)}_{\sfp(m)}\Big)\cdot e^{o_{M}(1)}
    \end{split}
  \end{equation}
  and by symmetry we will get \eqref{loc_rew_error}.
  To this end, we invoke the local limit theorem for simple random walks, which we recall from \cite{LL10}. In particular, by Theorem 2.3.11 \cite{LL10}, we have that there exists $\rho>0$ such that for all $n\geq 0$ and $x \in \Z^2$ with $|x|<\rho\, n$,
  \begin{equation} \label{llt}
    q_n(x)=g_{\frac{n}{2}}(x)\cdot e^{O\Big(\frac{1}{n}+\frac{|x|^4}{n^3}\Big)}\cdot 2\cdot\ind_{\big\{(n,x) \in \Z^3_{\mathsf{even}}\big\}}\, ,
  \end{equation}
  where $g_t(x)=\frac{e^{-\frac{|x|^2}{2t}}}{2 \pi t}$ denotes the $2$-dimensional heat kernel and
  \begin{equation*}
    \Z^3_{\mathsf{even}}:=\{(n,x) \in \Z \times \Z^2:\, n+x_1+x_2=0 \,\, (\text{mod}\,\, 2)\}\, .
  \end{equation*}
  The last constraint in $\eqref{llt}$ is a consequence of the periodicity of the simple random walk. Let us proceed with the proof of Lemma \ref{lemma_ker_repl}. First, we derive two  inequalities which are going to be useful for the approximations using the local limit theorem.
\vskip 1mm
{\bf Auxiliary inequality 1.}
  We claim that
  \begin{equation} \label{long_sc_sep}
    a_m-b_{m-1}>(M-1)\,\big( b_{m-1}-b_{\sfp(m)}\big) \, .
  \end{equation}
    To this end, we start with the fact that
  \begin{align}\label{recursion}
   a_{k+1}-b_k>M(b_k-b_{k-1})\,\,\, \text{holds for all $1\leq k \leq r-1$},
   \end{align}
   which we will apply repeatedly. Starting with $k:=m-1$, we have that
  \begin{equation} \label{onestep_recursion}
    a_m-b_{m-1}>M(b_{m-1}-b_{m-2})=M(b_{m-1}-a_{m-1})+M(a_{m-1}-b_{m-2})\, .
  \end{equation}
  We can then estimate the second term in the right-hand-side as
  \begin{equation*}
    \begin{split}
      M(a_{m-1}-b_{m-2})& =(M-1)(a_{m-1}-b_{m-2})+(a_{m-1}-b_{m-2})\\
      &>(M-1)(a_{m-1}-b_{m-2})+M(b_{m-2}-b_{m-3})
    \end{split}
  \end{equation*}
  where in the last step we used \eqref{recursion} with $k:=m-2$. Inserting this into \eqref{onestep_recursion}
   we have that
  \begin{align*}
  a_m-b_{m-1} > M(b_{m-1}-a_{m-1}) + (M-1)(a_{m-1}-b_{m-2})+M(b_{m-2}-b_{m-3})
  \end{align*}
  We next decompose $M(b_{m-2}-b_{m-3})$ as we did in \eqref{onestep_recursion} for $M(b_{m-1}-b_{m-2})$ and iterating this procedure up to $\mathfrak{p}(m)$ we obtain that
  \begin{align*}
  a_m-b_{m-1} 
  &> \sum_{j=\mathfrak{p}(m)+1}^{m-1}   \Big( M(b_{j}-a_{j}) + (M-1)(a_{j}-b_{j-1}) \Big) \\
  & > (M-1) \big( b_{m-1} - b_{\mathfrak{p}(m)}\big),
 \end{align*}
 where in the last step we reduced $M(b_{j}-a_{j})$ to $(M-1)(b_{j}-a_{j})$ and then telescoped.
  \vskip 1mm
  {\bf  Auxiliary inequality  2.} As a second step, we will prove that
  \begin{equation} \label{spatial_final}
    \Big |\,\big|x^{(i_m)}_m-y^{(i_m)}_{\sfp(m)}\big|-\big|x^{(i_m)}_m-y^{(i_m)}_{m-1}\big|\, \Big|\leq R\cdot \sqrt{2r-1} \cdot \sqrt{\frac{a_m-b_{m-1}}{M-1}} \, .
  \end{equation}
  To this end, by the reverse triangle inequality we have that
  \begin{equation} \label{tr_ineq}
    \Big |\,\big|x^{(i_m)}_m-y^{(i_m)}_{\sfp(m)}\big|-\big|x^{(i_m)}_m-y^{(i_m)}_{m-1}\big|\, \Big|\leq \big|y^{(i_m)}_{m-1}-x^{(i_m)}_{m-1}\big|+\big|x^{(i_m)}_{m-1}-y^{(i_m)}_{m-2}\big|+\dots+\big|x^{(i_m)}_{\sfp(m)+1}-y^{(i_m)}_{\sfp(m)}\big| \, .
  \end{equation}
  Note that by the diffusivity constraints of $\sfC^{\mathsf{(main)}}_{r,N,R,M}$ we have that
  \begin{equation} \label{spatial_dif}
    \begin{split}
      \big|y^{(i_m)}_{m-1}-x^{(i_m)}_{m-1}\big|+\big|x^{(i_m)}_{m-1}-y^{(i_m)}_{m-2}\big|+\dots+&\big|x^{(i_m)}_{\sfp(m)+1}-y^{(i_m)}_{\sfp(m)}\big| \\
      \leq &R\cdot \bigg(\sum_{k=\sfp(m)}^{m-2}\sqrt{a_{k+1}-b_k}+\sum_{k=\sfp(m)+1}^{m-1}\sqrt{b_k-a_k}\bigg)\, .
    \end{split}
  \end{equation}
  By Cauchy-Schwarz on the right hand side of \eqref{spatial_dif}, \eqref{long_sc_sep} and the fact that $m\leq r$, we furthermore have that
  \begin{equation}\label{eq:CS}
    \begin{split}
      \Big(\sum_{k=\sfp(m)}^{m-2}\sqrt{a_{k+1}-b_k}+\sum_{k=\sfp(m)+1}^{m-1}\sqrt{b_k-a_k}\Big)
      & \leq \sqrt{2r-1} \,\bigg(\sum_{k=\sfp(m)}^{m-2}(a_{k+1}-b_k)+\sum_{k=\sfp(m)+1}^{m-1}(b_k-a_k) \bigg)^{1/2}\\
      &= \sqrt{2r-1}\cdot \sqrt{b_{m-1}-b_{\sfp(m)}} \\
      & \leq \sqrt{2r-1} \cdot \sqrt{\frac{a_m-b_{m-1}}{M-1}} \, .
    \end{split}
  \end{equation}
   Combining \eqref{eq:CS} and \eqref{tr_ineq} we arrive at \eqref{spatial_final}.
  \vskip 2mm
  Now, we are ready to show approximation \eqref{one_each_time}. By \eqref{llt} we have
  \begin{equation} \label{err_handling}
    \begin{split}
      &\frac{q_{a_{m}-b_{m-1}}\big(x^{(i_m)}_m-y^{(i_m)}_{m-1}\big)}{q_{a_m-b_{\sfp(m)}}\big(x^{(i_m)}_m-y^{(i_m)}_{\sfp(m)}\big)}\\
      &=e^{-\frac{|x^{(i_m)}_m-y^{(i_m)}_{m-1}|^2}{a_{m}-b_{m-1}}+\frac{|x^{(i_m)}_m-y^{(i_m)}_{\sfp(m)}|^2}{a_m-b_{\sfp(m)}}} \cdot \bigg( \frac{a_m-b_{\sfp(m)}}{a_{m}-b_{m-1}}\bigg) \cdot
      e^{O\Big(\frac{1}{a_{m}-b_{m-1}}+\frac{|x^{(i_m)}_m-y_{m-1}^{(i_m)}|^4+|x^{(i_m)}_m-y^{(i_m)}_{\sfp(m)}|^4}{(a_m-b_{m-1})^3}\Big)}  \, .
    \end{split}
  \end{equation}
  Let us look at each term on the right hand side of \eqref{err_handling}, separately.
  First, we have
  \begin{equation*}
    \begin{split}
      e^{-\frac{\big|x^{(i_m)}_m-y^{(i_m)}_{m-1}\big|^2}{a_{m}-b_{m-1}}+\frac{\big|x^{(i_m)}_m-y^{(i_m)}_{\sfp(m)}\big|^2}{a_m-b_{\sfp(m)}}}  &\leq e^{\frac{1}{a_{m}-b_{m-1}}\Big( \, \big|x^{(i_m)}_m-y^{(i_m)}_{\sfp(m)}\big|^2 -\big|x^{(i_m)}_m-y^{(i_m)}_{m-1}\big|^2 \,\Big) } \\
      &= e^{\frac{1}{a_{m}-b_{m-1}}\, \Big(\big|x^{(i_m)}_m-y^{(i_m)}_{\sfp(m)}\big| +\big|x^{(i_m)}_m-y^{(i_m)}_{m-1}\big|\Big)\cdot \Big(\big|x^{(i_m)}_m-y^{(i_m)}_{\sfp(m)}\big| -\big|x^{(i_m)}_m-y^{(i_m)}_{m-1}\big|\Big) }\, .
    \end{split}
  \end{equation*}
  and by using \eqref{spatial_final} we have
  \begin{equation*}
    \begin{split}
      \big|x^{(i_m)}_m-y^{(i_m)}_{\sfp(m)}\big| +\big|x^{(i_m)}_m-y^{(i_m)}_{m-1}\big| &\leq 2 \big|x^{(i_m)}_{m}-y^{(i_m)}_{m-1}\big|+R\sqrt{2r-1}\sqrt{\frac{a_m-b_{m-1}}{M-1}} \\
      &\leq 2R\sqrt{a_m-b_{m-1}} + R \sqrt{2r-1} \sqrt{\frac{a_m-b_{m-1}}{M-1}} \\
      & = R \sqrt{a_m-b_{m-1}} \Big(2+\sqrt{\frac{2r-1}{M-1}}\Big) \, .
    \end{split}
  \end{equation*}
  Therefore, by \eqref{spatial_final}, again, we get
  \begin{equation*}
    e^{\frac{1}{a_{m}-b_{m-1}}\, \Big(\big|x^{(i_m)}_m-y^{(i_m)}_{\sfp(m)}\big| +\big|x^{(i_m)}_m-y^{(i_m)}_{m-1}\big|\Big)\cdot \Big(\big|x^{(i_m)}_m-y^{(i_m)}_{\sfp(m)}\big| -\big|x^{(i_m)}_m-y^{(i_m)}_{m-1}\big|\Big) }\leq e^{R^2 \Big(2+\sqrt{\frac{2r-1}{M-1}} \Big)\sqrt{\frac{2r-1}{M-1}}} \, .
  \end{equation*}
  Similarly, we can get a lower bound of
  \begin{equation*}
    e^{-\frac{\big|x^{(i_m)}_m-y^{(i_m)}_{m-1}\big|^2}{a_{m}-b_{m-1}}+\frac{\big|x^{(i_m)}_m-y^{(i_m)}_{\sfp(m)}\big|^2}{a_m-b_{\sfp(m)}}} \geq e^{-R^2\big(1-\frac{1}{M}\big) \Big(2+\sqrt{\frac{2r-1}{M-1}} \Big)\sqrt{\frac{2r-1}{M-1}}} \, ,
  \end{equation*}
  since $ a_m-b_{\sfp(m)}<\big(1+\frac{1}{M-1}\big)(a_m-b_{m-1})$ by \eqref{long_sc_sep}. The second term in \eqref{err_handling} can be handled by \eqref{long_sc_sep} as
  \begin{equation*}
    1\leq   \Big( \frac{a_m-b_{\sfp(m)}}{a_{m}-b_{m-1}}\Big) =\Big( 1+\frac{b_{m-1}-b_{\sfp(m)}}{a_{m}-b_{m-1}}\Big) < 1+\frac{1}{M-1} \xrightarrow{M \to \infty} 1\, .
  \end{equation*}
  For the last term in \eqref{err_handling} we have that
  \begin{equation*}
    \begin{split}
      & \frac{\big|x^{(i_m)}_m-y_{m-1}^{(i_m)}\big|^4+ \big|x^{(i_m)}_m-y^{(i_m)}_{\sfp(m)}\big|^4}{(a_{m}-b_{m-1})^3}\\
      &\qquad \qquad\qquad \qquad \leq \frac{\big|x^{(i_m)}_m-y_{m-1}^{(i_m)}\big|^4}{(a_{m}-b_{m-1})^3}+\frac{\Big(|x^{(i_m)}_m-y_{m-1}^{(i_m)}|+R \cdot\sqrt{2r-1}\cdot\sqrt{\frac{(a_{m}-b_{m-1})}{M-1}}\, \Big)^4}{(a_{m}-b_{m-1})^3} \\
      & \qquad \qquad \qquad \qquad \leq \frac{9\,R^4}{(a_{m}-b_{m-1})}+\frac{8R^4(2r-1)^2}{(a_{m}-b_{m-1})\cdot (M-1)^2}\, ,
    \end{split}
  \end{equation*}
  where we used \eqref{spatial_final} along with the inequality $(x+y)^4\leq 8(x^4+y^4)$ for $x,y \in \R$. Therefore,
  \begin{equation*}
    e^{ \frac{\big|x^{(i_m)}_m-y_{m-1}^{(i_m)}\big|^4+\big|x^{(i_m)}_m-y^{(i_m)}_{\sfp(m)}\big|^4}{(a_{m}-b_{m-1})^3}} \leq e^{\frac{9\,R^4}{(a_{m}-b_{m-1})}+\frac{8R^4(2r-1)^2}{(a_{m}-b_{m-1})\cdot {(M-1)^2}}} \leq  e^{\frac{9\,R^4}{M}+\frac{8R^4(2r-1)^2}{M \cdot {(M-1)^2}}}  \xrightarrow{M \to \infty} 1\, ,
  \end{equation*}
  where we used in the last inequality that $a_m-b_{m-1}>M(b_{m-1}-b_{m-2})\geq M$ by \eqref{C_main_set}.
\end{proof}

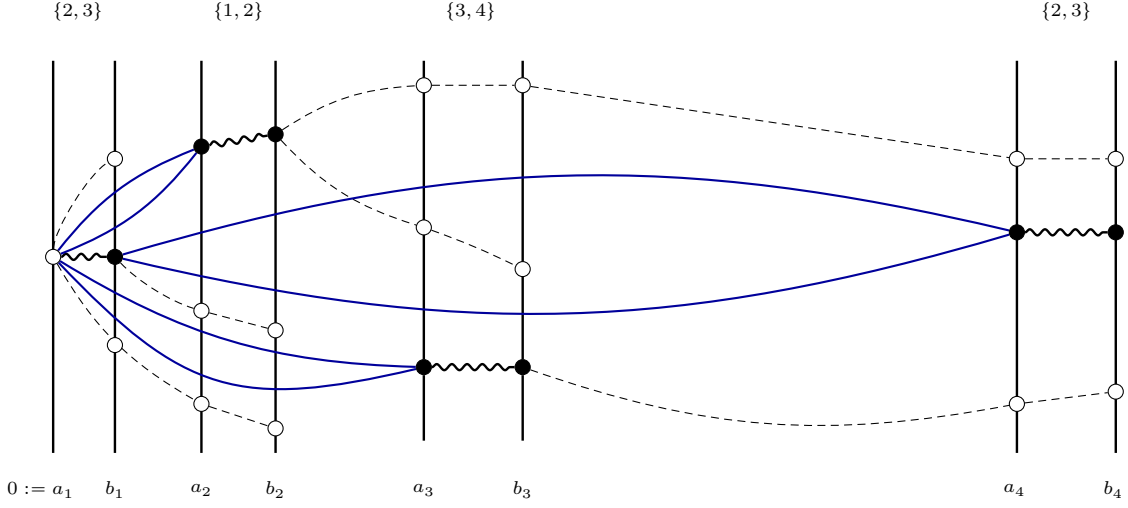
\begin{figure}
  \tikzstyle{filled_vertex}=[fill={black}, draw={black}, shape=circle,inner sep=2]
  \tikzstyle{new style 0}=[fill=white, draw={black}, shape=circle, inner sep=2pt]
  \tikzstyle{straight_line}=[-, fill=none, draw={black},line width=0.9pt]
  \tikzstyle{wiggly_line}=[-, fill=none, draw={black}, decorate=true, decoration={snake,amplitude=.4mm,segment length=2mm},line width=0.9pt]
  \tikzstyle{densely_dashed}= [dash pattern=on 3pt off 2pt, draw={black}]
  \tikzstyle{rewired}=[-, draw={rgb,255: red,128; green,0; blue,128},line width=0.8pt]
  \tikzstyle{none}=[inner sep=0pt]
  \tikzstyle{straight_line_colored}=[-, fill=none, draw={dukeblue}, line width=0.8pt]
  \pgfdeclarelayer{edgelayer}
  \pgfdeclarelayer{nodelayer}
  \pgfsetlayers{edgelayer,nodelayer,main}
  \usetikzlibrary{decorations.pathmorphing}

  { \tiny

    \begin{tikzpicture}[scale=0.65]
      \begin{pgfonlayer}{nodelayer}
        \node [style=none] (1) at (0, 4) {};
        \node [style=none] (2) at (0, -4) {};
        \node [style=new style 0] (4) at (0, 0) {};
        \node [style={filled_vertex}] (5) at (19.5, 0.5) {};
        \node [style=none] (6) at (1.25, 4) {};
        \node [style=none] (7) at (1.25, -4) {};
        \node [style={filled_vertex}] (8) at (1.25, 0) {};
        \node [style={filled_vertex}] (9) at (3, 2.25) {};
        \node [style={filled_vertex}] (11) at (4.5, 2.5) {};
        \node [style=none] (14) at (3, 4) {};
        \node [style=none] (15) at (3, -4) {};
        \node [style=none] (16) at (4.5, 4) {};
        \node [style=none] (17) at (4.5, -4) {};
        \node [style={filled_vertex}] (18) at (21.5, 0.5) {};
        \node [style=none] (19) at (19.5, 4) {};
        \node [style=none] (20) at (21.5, 4) {};
        \node [style=none] (21) at (19.5, -4) {};
        \node [style=none] (22) at (21.5, -4) {};
        \node [style=new style 0] (27) at (3, -3) {};
        \node [style=new style 0] (28) at (4.5, -3.5) {};
        \node [style=none] (30) at (-0.25, -4.75) {$0:=a_1$};
        \node [style=none] (31) at (1.25, -4.75) {$b_1$};
        \node [style=none] (32) at (3, -4.75) {$a_2$};
        \node [style=none] (33) at (4.5, -4.75) {$b_2$};
        \node [style=none] (34) at (7.5, -4.75) {$a_3$};
        \node [style=none] (35) at (9.5, -4.75) {$b_3$};
        \node [style=new style 0] (36) at (1.25, 2) {};
        \node [style=new style 0] (37) at (1.25, -1.8) {};
        \node [style=none] (59) at (0.5, 5) {$\{2,3\}$};
        \node [style=none] (60) at (3.75, 5) {$\{1,2\}$};
        \node [style=none] (61) at (20.5, 5) {$\{2,3\}$};
        \node [style={filled_vertex}] (62) at (7.5, -2.25) {};
        \node [style=new style 0] (63) at (3, -1.1) {};
        \node [style=new style 0] (64) at (4.5, -1.5) {};
        \node [style={filled_vertex}] (65) at (9.5, -2.25) {};
        \node [style=none] (68) at (7.5, 4) {};
        \node [style=none] (69) at (9.5, 4) {};
        \node [style=none] (70) at (7.5, -3.75) {};
        \node [style=none] (71) at (9.5, -3.75) {};
        \node [style=none] (72) at (19.45, -4.75) {$a_4$};
        \node [style=none] (73) at (21.45, -4.75) {$b_4$};
        \node [style=none] (74) at (8.45, 5) {$\{3,4\}$};
        \node [style=new style 0] (75) at (7.5, 0.6) {};
        \node [style=new style 0] (76) at (9.5, -0.25) {};
        \node [style=new style 0] (77) at (7.5, 3.5) {};
        \node [style=new style 0] (78) at (9.5, 3.5) {};
        \node [style=new style 0] (79) at (19.5, 2) {};
        \node [style=new style 0] (80) at (21.5, 2) {};
        \node [style=new style 0] (81) at (21.5, -2.75) {};
        \node [style=new style 0] (54) at (19.5, -3) {};
      \end{pgfonlayer}
      \begin{pgfonlayer}{edgelayer}
        \draw [style={straight_line}] (1.center) to (2.center);
        \draw [style={straight_line}] (6.center) to (7.center);
        \draw [style={straight_line}] (14.center) to (15.center);
        \draw [style={straight_line}] (16.center) to (17.center);
        \draw [style={straight_line}] (19.center) to (21.center);
        \draw [style={straight_line}] (20.center) to (22.center);
        \draw [style={straight_line}] (68.center) to (62);
        \draw [style={straight_line}] (69.center) to (65);
        \draw [style={straight_line}] (62) to (70.center);
        \draw [style={straight_line}] (65) to (71.center);

        \draw [style={wiggly_line}] (4) to (8);
        \draw [style={wiggly_line}] (9) to (11);
        \draw [style={wiggly_line}] (5) to (18);
        \draw [style={wiggly_line}] (62) to (65);

        \draw [style={densely_dashed}, bend right=15] (4) to (27);
        \draw [style={densely_dashed}] (27) to (28);
        \draw [style={densely_dashed}, bend right=15] (11) to (75);
        \draw [style={densely_dashed}, bend left=3] (75) to (76);
        \draw [style={densely_dashed}, bend right=15, looseness=0.75] (8) to (63);
        \draw [style={densely_dashed}] (63) to (64);
        \draw [style={densely_dashed}, bend left, looseness=0.50] (4) to (36);
        \draw [style={densely_dashed}, bend left=15] (11) to (77);
        \draw [style={densely_dashed}] (77) to (78);
        \draw [style={densely_dashed}] (78) to (79);
        \draw [style={densely_dashed}, bend right=15] (65) to (54);
        \draw [style={densely_dashed}] (54) to (81);
        \draw [style={densely_dashed}] (79) to (80);

        \draw [style={straight_line_colored}, bend right=15] (5) to (8);
        \draw [style={straight_line_colored}, bend left=15] (5) to (8);
        \draw [style={straight_line_colored}, bend left, looseness=1.25] (62) to (4);
        \draw [style={straight_line_colored}, bend right=345] (62) to (4);
        \draw [style={straight_line_colored}, bend right=15] (9) to (4);
        \draw [style={straight_line_colored}, bend right=15] (4) to (9);
      \end{pgfonlayer}
    \end{tikzpicture}

  }
  \caption{Figure \ref{fig:scales} after {\em rewiring}. We use blue lines to represent the new kernels produced by rewiring. The dashed lines represent remaining free kernels from the rewiring procedure as well as kernels coming from using the Chapman-Kolmogorov formula for the simple random walk. }
  \label{fig:rewiring}
\end{figure}
\subsection{Final step} \label{sec:finalstep}
Now that we have Lemma \ref{lemma_ker_repl} at our disposal, we can prove the main approximation result of this step. Recall from \eqref{main_sum} that
\begin{equation*}
  \begin{split}
    & H^{\mathsf{(main)}}_{r,N,R,M}\\
    &\quad = \sum_{(I_1,...,I_r) \,\in\, \mathcal{I}^{(2)}} \,\sum_{(\vec{a},\vec{b},\vec{\bx},\vec{\by})\, \in \,\sfC^{\mathsf{(main)}}_{r,N,R,M}} \sfU^{I_1}_{b_1}(0,\by_1) \prod_{i=2}^r Q_{a_i-b_{i-1}}^{I_{i-1};I_i}(\by_{i-1},\bx_i) \,  \sfU^{I_i}_{b_i-a_i}(\bx_i,\by_i) \,\sigma_N(I_i)  \, .
  \end{split}
\end{equation*}

Define $H^{\mathsf{(rew)}}_{r,N,R,M}$ to be the resulting sum after rewiring has been applied to every term of $H^{\mathsf{(main)}}_{r,N,R,M}$, that is,
given a sequence of partitions $(I_1,\dots,I_r)\in \mathcal{I}^{(2)}$ and $(\vec{a},\vec{b},\vec{\bx},\vec{\by})\, \in \,\sfC^{\mathsf{(main)}}_{r,N,R,M}$, we apply the kernel replacement \eqref{ker_repl}
to all partitions $I_1,\dots,I_r$ starting from $I_r$ and moving backward. We remind the reader that 
we may denote a partition $I=\{i,j\} \sqcup \bigsqcup_{k \neq i,j}\{k\} \in \cI^{(2)}$ by its non-trivial part $\{i,j\}$, see subsection \ref{chaos_expansion_for_many_body_collisions}.
\begin{proposition} \label{introduction_of_rewired}
  Fix $0\leq r\leq K$. We have that
  \begin{equation} \label{main_rewired_comparison}
    H^{\mathsf{(rew)}}_{r,N,R,M}=e^{ K\cdot o_M(1)}\, H^{\mathsf{(main)}}_{r,N,R,M}
  \end{equation}
  and
  \begin{equation} \label{rewired_expansion}
    \begin{split}
      &H^{\mathsf{(rew)}}_{r,N,R,M} \\
      &=\hspace{-0.3cm} \sum_{\big(\{i_1,j_1\},\dots,\{i_r,j_r\}\big) \in \, \cI^{(2)}} \,\,  \sum_{(\vec{a},\vec{b},\vec{\bx},\vec{\by})\, \in \,\sfC^{\mathsf{(main)}}_{r,N,R,M}} \prod_{k=1}^r U^{\beta_{i_k,j_k}}_N\big(b_k-a_k,y^{(i_k)}_k-x^{(i_k)}_k\big) \cdot \prodtwo{1\leq \ell \leq h, }{\ell \neq i_k,j_k} q_{b_k-a_k}(y^{(\ell)}_k-x^{(\ell)}_k) \\
      & \qquad \qquad \times \prod_{m=2}^r \sigma^{i_m,j_m}_N \cdot q^{2}_{a_{m}-b_{\sfp(m)}}\big(x^{(i_m)}_{m}- y^{(i_{m})}_{\sfp(m)} \big) \cdot 
      \prodtwo{1\leq \ell \leq h, }{\ell \neq i_m,j_m} q_{a_m-b_{m-1}}\big(x^{(\ell)}_m-y^{(\ell)}_{m-1}\big) \, .
    \end{split}
  \end{equation}
\end{proposition}

\begin{proof}
  Equation \eqref{main_rewired_comparison} is a consequence of Lemma \ref{lemma_ker_repl} and the fact that $r\leq K$, while expansion \eqref{rewired_expansion} is a direct consequence of the rewiring procedure we described in the previous step, see also 
  Figures \ref{fig:scales} and \ref{fig:rewiring}.
\end{proof}
Next, we derive upper and lower bounds for $H^{\mathsf{(rew)}}_{r,N,R,M}$. We begin with the upper bound.

\begin{proposition} \label{rewired_upper_bound}
  We have that
  \begin{equation*}
    \begin{split}
      H^{\mathsf{(rew)}}_{r,N,R,M} \leq \sum_{\big(\{i_1,j_1\},\dots,\{i_r,j_r\}\big) \in  \, \mathcal{I}^{(2)}} \,\,\, \sumtwo{0:=a_1\leq b_1<a_2\leq \dots < a_r\leq b_r \leq N,}{0:=x_1,y_1,\dots,x_r,y_r \in \Z^2} \prod_{k=1}^r U^{\beta_{i_k,j_k}}_N\big(b_k-a_k,y_k-x_k\big) \, \\
      \times \, \prod_{m=2}^r \sigma^{i_m,j_m}_N \cdot q^{2}_{a_{m}-b_{\sfp(m)}}\big(x_{m}-y_{\sfp(m)}\big) \, .
    \end{split}
  \end{equation*}
\end{proposition}

\begin{proof}
  Fix $r\geq 1$ and from \eqref{rewired_expansion} recall that
  \begin{equation} \label{rewired_expansion_repeat}
    \begin{split}
      &H^{\mathsf{(rew)}}_{r,N,R,M}\\
      &\quad=\hspace{-0.3cm} \sum_{\big(\{i_1,j_1\},\dots,\{i_r,j_r\}\big) \in  \, \mathcal{I}^{(2)}} \,\,\, \sum_{(\vec{a},\vec{b},\vec{\bx},\vec{\by})\, \in \,\sfC^{\mathsf{(main)}}_{r,N,R,M}} \prod_{k=1}^r U^{\beta_{i_k,j_k}}_N\big(b_k-a_k,y^{(i_k)}_k-x^{(i_k)}_k\big) \cdot 
      \prodtwo{1\leq \ell \leq h, }{\ell \neq i_k,j_k} q_{b_k-a_k}(y^{(\ell)}_k-x^{(\ell)}_k) \\
      & \qquad \qquad \qquad\times \prod_{m=2}^r \sigma^{i_m,j_m}_N \cdot q^{2}_{a_{m}-b_{\sfp(m)}}\big(x^{(i_m)}_{m}-y^{(i_{m})}_{\sfp(m)} \big) \cdot 
      \prodtwo{1\leq \ell \leq h, }{\ell \neq i_m,j_m} q_{a_m-b_{m-1}}\big(x^{(\ell)}_m-y^{(\ell)}_{m-1}\big) \, .
    \end{split}
  \end{equation}
  For the sake of obtaining an upper bound on $H^{\mathsf{(rew)}}_{r,N,R,M}$ we can sum $(\vec{a},\vec{b},\vec{\bx},\vec{\by})$ in \eqref{rewired_expansion_repeat} over $\sfC_{r,N}$, see definition in \eqref{C_set}, instead of
  $\sfC^{\mathsf{(main)}}_{r,N,R,M}$.
  We start the summation of the right hand side of \eqref{rewired_expansion_repeat} from the end. Using that for $n \in \N$, $\sum_{z \in \Z^2} q_n(z)=1$ we deduce that
  \begin{equation*}
    \sumtwo{y_r^{(\ell)}\in \Z^2:}{1\leq \ell \leq h, \,\ell \neq i_r,j_r}\, \prod_{\ell \neq i_r,j_r} q_{b_r-a_r}(y^{(\ell)}_r-x^{(\ell)}_r)= 1\, .
  \end{equation*}
  We leave the sum $\sum_{b_r \in [a_r,N], \, y^{(i_r)}_r \in \Z^2} U_N\big(b_r-a_r,y_r^{(i_r)}-x_r^{(i_r)}\big)$ intact and move on to the time interval 
  $[b_{r-1},a_r]$. We use again that for $n \in \N$, $\sum_{z \in \Z^2} q_n(z)=1$, to deduce that
  \begin{equation*}
    \sumtwo{x_r^{(\ell)}\in \Z^2:}{1\leq \ell \leq h, \,\ell \neq i_r,j_r}\,\prod_{\ell \neq i_r,j_r}\,q_{a_r-b_{r-1}}(x^{(\ell)}_r-y^{(\ell)}_{r-1}) =1 \, .
  \end{equation*}
  Again, we leave the sum $\sum_{a_r \in \, (b_{r-1},b_r], \,x_r^{(i_r)} \in \Z^2}\, q_{a_r-b_{\sfp(r)}}^2\big(x^{(i_r)}_r- y^{(i_{r})}_{\sfp(r)} \big)$ intact. We can iterate this procedure inductively since due to rewiring all the spatial variables $y^{(\ell)}_{r-1}$, $\ell \neq i_{r-1},j_{r-1}$ are free, that is, there are no outgoing laces starting off $y^{(\ell)}_{r-1}$, $\ell \neq i_{r-1},j_{r-1}$ at time $b_{r-1}$. The summations we have performed correspond to getting rid of the dashed lines in Figure \ref{fig:rewiring}.
  Iterating this procedure inductively then implies the following upper bound for $H^{\mathsf{(rew)}}_{r,N,R,M}$.
  \begin{equation*}
    \begin{split}
      H^{\mathsf{(rew)}}_{r,N,R,M}
      \leq \sum_{\big(\{i_1,j_1\},\dots,\{i_r,j_r\}\big) \in  \, \mathcal{I}^{(2)}} \,\, \sumtwo{0:=a_1\leq b_1<a_2\leq \dots < a_r\leq b_r \leq N\, ,}{0:=x_1,y_1,\dots,x_r,y_r \in \Z^2} \prod_{k=1}^r U^{\beta_{i_k,j_k}}_N\big(b_k-a_k,y_k-x_k\big) \, \\
      \times \, \prod_{m=2}^r \sigma^{i_m,j_m}_N \cdot q^{2}_{a_{m}-b_{\sfp(m)}}\big(x_{m}-y_{\sfp(m)}\big) \, .
    \end{split}
  \end{equation*}

\end{proof}

In the next proposition we derive complementary lower bounds for $H^{\mathsf{(rew)}}_{r,N,R,M}$. Given $0 \leq r \leq K$ and a sequence of partitions $\vec{I}=(\{i_1,j_1\},\dots,\{i_r,j_r\}) \in \mathcal{I}^{(2)}$ we define the set $\sfC^{\mathsf{(rew)}}_{r,N,R,M}(\vec{I}\,)$ to be $\sfC^{\mathsf{(main)}}_{r,N,R,M}$ where for every $2\leq m \leq r$ we replace the diffusivity constraint
$\norm{\bx_m-\by_{m-1}}_{\infty} \leq R \sqrt{a_m-b_{m-1}}$
by the constraints
\begin{equation*}
  \begin{split}
    &|x^{(\ell)}_m-y^{(\ell)}_{m-1}|\leq R\sqrt{a_m-b_{m-1}},\, \ell \in \{1,\dots,h\}\smallsetminus \{i_m,j_m\} \quad \text{and} \\
    &\quad  |x^{(\ell')}_m-y^{(\ell')}_{\sfp(m)}|\leq R \sqrt{1-\frac{1}{M}}\bigg(1-\sqrt{\frac{2K-1}{M-1}}\bigg) \sqrt{a_m-b_{\sfp(m)}}, \, \, \ell' \in \{i_m,j_m\} \, .
  \end{split}
\end{equation*}

This replacement transforms the diffusivity constraints imposed on the jumps of two walks $\{i_m,j_m\}$ from their respective positions at time $b_{m-1}$ to time $a_m$, which is the time they (re)start colliding, to a diffusivity constraint connecting their common position at time $b_{\sfp(m)}$, which is the last time they collided before time $a_m$, to their common position at time $a_m$ when they start colliding again.

We have the following Lemma.
\begin{lemma} \label{rewired_reduction}
  Let $0 \leq r \leq K$ and $M>2K$. For all $\vec{I}=\big(\{i_1,j_1\},\dots,\{i_r,j_r\}\big) \in  \, \mathcal{I}^{(2)}$ we have that
  $$ \sfC^{\mathsf{(rew)}}_{r,N,R,M}(\vec{I}\,) \subset \sfC^{\mathsf{(main)}}_{r,N,R,M} \, .$$
\end{lemma}

\begin{proof}
  Fix $0\leq r \leq K$, a sequence $\vec{I}=\big(\{i_1,j_1\}, \dots, \{i_r,j_r\}\big) \in \cI^{(2)}$ 
  and $(\vec{a},\vec{b},\vec{\bx},\vec{\by}) \in \sfC^{\mathsf{(rew)}}_{r,N,M,R}(\vec{I}\,)$.  Moreover, let $2\leq m \leq r$.
  By symmetry it suffices to prove that
  \begin{equation*}
    \begin{split}
      \ind_{\big\{|x^{(i_m)}_m- y^{(i_{m})}_{\sfp(m)}|\leq  R \sqrt{1-\frac{1}{M}}\Big(1-\sqrt{\frac{2r-1}{M-1}}\Big)\sqrt{a_m-b_{\sfp(m)}}\big \}} \leq  \ind_{\big\{|x^{(i_m)}_m-y^{(i_{m})}_{m-1}|\leq R \sqrt{a_m-b_{m-1}}\big\}}
      \,  .
    \end{split}
  \end{equation*}
  Indeed, by the definition of $\sfC^{\mathsf{(rew)}}_{r,N,M,R}(\vec{I}\,)$ and \eqref{spatial_final} we have that for $(\vec{a},\vec{b},\vec{\bx},\vec{\by}) \in \sfC^{\mathsf{(rew)}}_{r,N,M,R}(\vec{I}\, )$,
  \begin{equation} \label{rew_subset_main_1}
    \Big |\,|x^{(i_m)}_m- y^{(i_{m})}_{\sfp(m)} |-|x^{(i_m)}_m-y^{(i_{m})}_{m-1}|\, \Big|\leq R\cdot \sqrt{2r-1} \cdot \sqrt{\frac{a_m-b_{m-1}}{M-1}} \, .
  \end{equation}
  Moreover by \eqref{long_sc_sep} we have that $$a_m-b_{m-1}>(M-1)(b_{m-1}-b_{\sfp(m)}) \Rightarrow a_m-b_{\sfp(m)}>M(b_{m-1}-b_{\sfp(m)}) \, .$$
  Therefore,
  \begin{equation} \label{rew_subset_main_2}
    a_m-b_{m-1}=a_m-b_{\sfp(m)}-(b_{m-1}-b_{\sfp(m)})>a_m-b_{\sfp(m)}-\frac{1}{M}(a_m-b_{\sfp(m)})=\Big(1-\frac{1}{M}\Big)(a_m-b_{\sfp(m)}) \, .
  \end{equation}
  Combining inequalities \eqref{rew_subset_main_1} and \eqref{rew_subset_main_2} we get the result.

\end{proof}

\begin{proposition} \label{rewired_lower_bound}
  Let $0 \leq r \leq K$. For $M>2K$ we have that
  \begin{align*}
      &H^{\mathsf{(rew)}}_{r,N,R,M} \\
      &\quad \geq (1-e^{-cR^2})^{2Kh} \, \hspace{-0.1cm} \sum_{\vec{I}=\big(\{i_1,j_1\},\dots,\{i_r,j_r\}\big) \in  \, \mathcal{I}^{(2)}} \, \sumthree{0:=a_1\leq b_1<a_2\leq \dots < a_r\leq b_r \leq N ,}{a_{i+1}-b_i>M(b_i-b_{i-1}), \, 1\leq i \leq r-1,}{0:=x_1,y_1,\dots,x_r,y_r \in \Z^2} \prod_{k=1}^r U^{\beta_{i_k,j_k}}_N\big(b_k-a_k,y_k-x_k\big) \, \\
    &\quad   \times \, \prod_{m=2}^r \sigma^{i_m,j_m}_N \cdot q^{2}_{a_{m}-b_{\sfp(m)}}\big(x_{m}-y_{\sfp(m)}\big) \cdot \ind_{\big\{|y_k-x_k|\leq R\sqrt{b_k-a_k}, \, |x_m-y_{\sfp(m)}|\leq R\, C_{K,M}\sqrt{a_m-b_{\sfp(m)}}\big\}},
  \end{align*}
  with $C_{K,M}:=\sqrt{1-\frac{1}{M}}\bigg(1-\sqrt{\frac{2K-1}{M-1}}\bigg)$.
\end{proposition}

\begin{proof}
  Recall from \eqref{rewired_expansion} that
  \begin{equation*}
    \begin{split}
      H^{\mathsf{(rew)}}_{r,N,R,M}&=\hspace{-0.2cm} \sum_{\vec{I}=\big(\{i_1,j_1\},\dots,\{i_r,j_r\}\big) \in  \, \mathcal{I}^{(2)}} \,\,\,  \sum_{(\vec{a},\vec{b},\vec{\bx},\vec{\by})\, \in \,\sfC^{\mathsf{(main)}}_{r,N,R,M}} \\
      &\qquad \qquad \times  \prod_{k=1}^r U^{\beta_{i_k,j_k}}_N\big(b_k-a_k,y^{(i_k)}_k-x^{(i_k)}_k\big)\cdot \prod_{\ell \neq i_k,j_k}q_{b_k-a_k}(y^{(\ell)}_k-x^{(\ell)}_k) \\
      & \qquad \qquad \quad \times \prod_{m=2}^r \sigma^{i_m,j_m}_N \cdot q^{2}_{a_{m}-b_{\sfp(m)}}\big(x^{(i_m)}_{m}-
      y^{(i_{m})}_{\sfp(m)} \big) \cdot \prod_{\ell \neq i_m,j_m}q_{a_m-b_{m-1}}\big(x^{(\ell)}_m-y^{(\ell)}_{m-1}\big) \, .
    \end{split}
  \end{equation*}
  By Lemma \ref{rewired_reduction} we have that
  \begin{equation} \label{temp_lower_bound}
    \begin{split}
      H^{\mathsf{(rew)}}_{r,N,R,M}&\geq \hspace{-0.2cm} \sum_{\vec{I}=\big(\{i_1,j_1\},\dots,\{i_r,j_r\}\big) \in  \, \mathcal{I}^{(2)}} \,\,\,  \sum_{(\vec{a},\vec{b},\vec{\bx},\vec{\by})\, \in \,\sfC^{\mathsf{(rew)}}_{r,N,R,M}(\vec{I}\,)} \\
      &\qquad \qquad \times  \prod_{k=1}^r U^{\beta_{i_k,j_k}}_N\big(b_k-a_k,y^{(i_k)}_k-x^{(i_k)}_k\big) \cdot\prod_{\ell \neq i_k,j_k}q_{b_k-a_k}(y^{(\ell)}_k-x^{(\ell)}_k) \\
      & \qquad \qquad \quad \times \prod_{m=2}^r \sigma^{i_m,j_m}_N \cdot q^{2}_{a_{m}-b_{\sfp(m)}}\big(x^{(i_m)}_{m}-
      y^{(i_{m})}_{\sfp(m)} \big) \cdot \prod_{\ell \neq i_m,j_m}q_{a_m-b_{m-1}}\big(x^{(\ell)}_m-y^{(\ell)}_{m-1}\big) \, .
    \end{split}
  \end{equation}
  The first step in getting a lower bound for $H^{\mathsf{(rew)}}_{r,N,R,M}$ is to get rid of the dashed lines, see Figure \ref{fig:rewiring}.
  We follow the steps we took in the proof of Proposition \ref{rewired_upper_bound} for the upper bound. In particular, we start the summation of \eqref{temp_lower_bound}  beginning from the end. Using that for $n \in \N$ and $R\in (0,\infty)$
  \begin{equation} \label{lower_deviation_bound}
    \sum_{z \in \Z^2: \, |z|\leq R \sqrt{n}} q_n(z)=1 - \sum_{z \in \Z^2: \, |z|> R \sqrt{n}} q_n(z) \geq 1-e^{-cR^2} \, ,
  \end{equation}
  by \eqref{moderate_devation_estimate}, we get that
  \begin{equation*}
    \sumtwo{y_r^{(\ell)}\in \Z^2: |y_r^{(\ell)}-x^{(\ell)}_r|\leq R\sqrt{b_r-a_r}, \,}{1\leq \ell \leq h,\,\ell \neq i_r,j_r}\prod_{\ell \neq i_r,j_r} q_{b_r-a_r}(y^{(\ell)}_r-x^{(\ell)}_r) \geq  (1-e^{-cR^2})^h\, .
  \end{equation*}
  We leave the sum
  $$\sumtwo{b_r \in [a_r,N],}{y^{(i_r)}_r \in \Z^2:\, |y^{(i_r)}_r-x_r^{(i_r)}|\leq R \sqrt{b_r-a_r} } U_N\big (b_r-a_r,y_r^{(i_r)}-x_r^{(i_r)}\big )$$
  as is and move on to the time interval $[b_{r-1},a_r]$. We use \eqref{lower_deviation_bound} to deduce that
  \begin{equation*}
    \sumtwo{x_r^{(\ell)}\in \Z^2: |x_r^{(\ell)}-y^{(\ell)}_{r-1}|\leq R\sqrt{a_r-b_{r-1}}, \,}{ 1\leq \ell \leq h,\,\ell \neq i_r,j_r} \,\prod_{\ell \neq i_r,j_r}\,q_{a_r-b_{r-1}}(x^{(\ell)}_r-y^{(\ell)}_{r-1}) \geq (1-e^{-cR^2})^{h} \, .
  \end{equation*}
  Again, we leave the sum $\sum_{a_r \in \, (b_{r-1},b_r], \,x_r^{(i_r)} \in \Z^2}\, q_{a_r-b_{\sfp(r)}}^2\big(x^{(i_r)}_r-y^{(i_{r})}_{\sfp(r)}\big)$ intact. We can continue this procedure since due to rewiring all the spatial variables $y^{(\ell)}_{r-1}$, $\ell \neq i_{r-1},j_{r-1}$ are free, 
  i.e. there are no outgoing laces starting off $y^{(\ell)}_{r-1}$, $\ell \neq i_{r-1},j_{r-1}$ at time $b_{r-1}$, and there are no diffusivity constraints linking
  $x^{(i_r)}_r=x_r^{(j_r)}$ with $y^{(i_r)}_{r-1},y^{(j_r)}_{r-1}$ by definition of $\sfC^{\mathsf{(rew)}}_{r,N,R,M}(\vec{I})$. Iterating this procedure we obtain that
  \begin{equation*}
    \begin{split}
      &H^{\mathsf{(rew)}}_{r,N,R,M}\\
      & \geq (1-e^{-cR^2})^{2Kh}\, \hspace{-0.1cm} \sum_{\vec{I}=\big(\{i_1,j_1\},\dots,\{i_r,j_r\}\big) \in  \, \mathcal{I}^{(2)}} \, \sumthree{0:=a_1\leq b_1<a_2\leq \dots < a_r\leq b_r \leq N ,}{a_{i+1}-b_i>M(b_i-b_{i-1}), \, 1\leq i \leq r-1,}{0:=x_1,y_1,\dots,x_r,y_r \in \Z^2} \prod_{k=1}^r U^{\beta_{i_k,j_k}}_N\big(b_k-a_k,y_k-x_k\big) \, \\
      &\qquad \times \, \prod_{m=2}^r \sigma^{i_m,j_m}_N \cdot q^{2}_{a_{m}-b_{\sfp(m)}}\big(x_{m}-y_{\sfp(m)}\big) \cdot \ind_{\big\{|y_k-x_k|\leq R\sqrt{b_k-a_k}, \, |x_m-y_{\sfp(m)}|\leq R\, C_{K,M}\sqrt{a_m-b_{\sfp(m)}}\big\}},
    \end{split}
  \end{equation*}
  with $C_{K,M}=\sqrt{1-\frac{1}{M}}\bigg(1-\sqrt{\frac{2K-1}{M-1}}\bigg)$.
\end{proof}

\begin{figure}
\begin{tikzpicture}[scale=0.6]
\draw  [thick] (-8, -1)  circle [radius=0.1];
\draw  [fill] (0, 0)  circle [radius=0.1]; \draw  [fill] (2, 0)  circle [radius=0.1];
\draw  [fill] (-5,0)  circle [radius=0.1]; \draw  [fill] (-3,0)  circle [radius=0.1];
\draw  [fill] (6,1)  circle [radius=0.1];  \draw  [fill] (8,1)  circle [radius=0.1];
\node at (0.2,-0.5) {\scalebox{0.7}{$a_3$}};  \node at (2.2,-0.5) {\scalebox{0.7}{$b_3$}};
\node at (-4.8,-0.5) {\scalebox{0.7}{$a_1$}};  \node at (-2.8,-0.5) {\scalebox{0.7}{$b_1$}};
\node at (6.2,0.5) {\scalebox{0.6}{$a_5$}};  \node at (8.2,0.5) {\scalebox{0.7}{$b_5$}};
\draw [-,thick, decorate, decoration={snake,amplitude=.4mm,segment length=2mm}] (0,0) -- (2,0);
\draw [-,thick, decorate, decoration={snake,amplitude=.4mm,segment length=2mm}] (-5,0) -- (-3,0);
\draw [-,thick, decorate, decoration={snake,amplitude=.4mm,segment length=2mm}] (6,1) -- (8,1);
\draw[thick] (-3,0) to [out=50,in=130] (0,0);   \draw[thick] (-3,0) to [out=-50,in=-130] (0,0);
\draw[thick] (-8,-1) to [out=30,in=170] (-5,0); \draw[thick] (-8,-1) to [out=10,in=-150] (-5,0);
\draw[thick] (-8,-1) to [out=60,in=150] (6,1); \draw[thick] (-8,-1) to [out=50,in=160] (6,1);
\draw  [fill] (3, -1.5)  circle [radius=0.1]; \draw  [fill] (5, -1.5)  circle [radius=0.1];
\draw  [fill] (-2.5,-1.5)  circle [radius=0.1]; \draw  [fill] (-0.5,-1.5)  circle [radius=0.1];
\draw  [fill] (9,-1.5)  circle [radius=0.1];  \draw  [fill] (11,-1.5)  circle [radius=0.1];
\node at (3.2,-2) {\scalebox{0.7}{$a_4$}};  \node at (5.2,-2) {\scalebox{0.7}{$b_4$}};
\node at (-2.3,-2) {\scalebox{0.7}{$a_2$}};  \node at (-0.4,-2) {\scalebox{0.7}{$b_2$}};
\node at (9.2,-2) {\scalebox{0.6}{$a_6$}};  \node at (11.2,-2) {\scalebox{0.7}{$b_6$}};
\draw [-,thick, decorate, decoration={snake,amplitude=.4mm,segment length=2mm}] (3,-1.5) -- (5,-1.5);
\draw [-,thick, decorate, decoration={snake,amplitude=.4mm,segment length=2mm}] (-2.5,-1.5) -- (-0.5,-1.5);
\draw [-,thick, decorate, decoration={snake,amplitude=.4mm,segment length=2mm}] (9,-1.5) -- (11,-1.5);
\draw[thick] (-8,-1) to [out=-30,in=-150] (-2.5,-1.5); 
\draw[thick] (-0.5,-1.5) to [out=-50,in=-130] (3,-1.5); \draw[thick] (-0.5,-1.5) to [out=50,in=130] (3,-1.5); 
\draw[thick] (5,-1.5) to [out=-20,in=-160] (9,-1.5); \draw[thick] (5,-1.5) to [out=20,in=160] (9,-1.5); 
\draw[thick] (-8,-1) to [out=0,in=160] (-2.5,-1.5);
\end{tikzpicture}
\caption{Graphical representation of a term of the chaos representation of  
$\prod_{1\leq i<j \leq h}\E\Big[e^{\frac{\pi\beta_{i,j}}{\log N}\, \sfL_N^{(i,j)}}\Big]$ for $h=3$. 
This diagram captures the main contribution, which is when all $a$'s and $b$'s are distinct. Configurations where more than one pair of walks collide 
at a same time $n\leq N$ have lower order contributions.}
\label{fig:graphical-collisions}
\end{figure}
To proceed with the last steps of our proof it is useful to record two chaos expansions of 
$\prod_{1\leq i<j \leq h} \E\Big[e^{\frac{\pi\beta_{i,j}}{\log N}\, \sfL_N^{(i,j)}}\Big]$. These read as
\begin{align}
\prod_{1\leq i<j \leq h}\E\Big[e^{\frac{\pi\beta_{i,j}}{\log N}\, \sfL_N^{(i,j)}}\Big]
&= \sum_{\big(\{i_1,j_1\},\dots,\{i_r,j_r\}\big) \in  \, \mathcal{I}^{(2)}} \, 
 \sumtwo{0:=a_1\leq b_1\leq a_2\leq \dots \leq a_r\leq b_r \leq N\, ,}{0:=x_1,y_1,\dots,x_r,y_r \in \Z^2}  
\prod_{k=1}^r U^{\beta_{i_k,j_k}}_N\big(b_k-a_k,y_k-x_k\big) \notag \\
      & \hskip 4cm \times \prod_{m=2}^r \sigma^{i_m,j_m}_N \cdot q^{2}_{a_{m}-b_{\sfp(m)}}\big(x_{m}-y_{\sfp(m)}\big)
      \label{graphical-collisions1} \\
&= \sum_{\big(\{i_1,j_1\},\dots,\{i_r,j_r\}\big) \in  \, \mathcal{I}^{(2)}}
    \,\,\,  \sum_{0:=a_1\leq b_1\leq a_2\leq \dots \leq a_r\leq b_r \leq N}  \prod_{k=1}^r U^{\beta_{i_k,j_k}}_N\big(b_k-a_k\big) \notag\\
    & \hskip 4cm \times     \prod_{m=2}^r \sigma^{i_m,j_m}_N  q_{2(a_{m}-b_{\sfp(m)})}(0) .
    \label{graphical-collisions2}
\end{align}
This follows from \eqref{eq:2body} and \eqref{Un} and by grouping together intervals $[a,b]$ where one observes collisions of
only a single pair of random walks. Firgure \ref{fig:graphical-collisions} presents a graphical explanation / representation of these
chaos expansions.
\begin{proposition}\label{rewired_limit}
  We have that
  \begin{equation*}
    \lim_{K \to \infty} \lim_{M \to \infty} \lim_{R \to \infty} \lim_{N \to \infty} \sum_{r=0}^{K} H^{\mathsf{(rew)}}_{r,N,R,M}=\prod_{1\leq i<j \leq h} \frac{1}{1-\beta_{i,j}} \, .
  \end{equation*}
\end{proposition}

\begin{proof}We are going to prove this Proposition via means of the lower and upper bounds established in Propositions \ref{rewired_upper_bound} and \ref{rewired_lower_bound}.
  By Proposition \ref{rewired_upper_bound} we have that
  \begin{equation} \label{rewired_upper_bound_2}
    \begin{split}
      H^{\mathsf{(rew)}}_{r,N,R,M} \leq \sum_{\big(\{i_1,j_1\},\dots,\{i_r,j_r\}\big) \in  \, \mathcal{I}^{(2)}} \, \sumtwo{0:=a_1\leq b_1<a_2\leq \dots < a_r\leq b_r \leq N\, ,}{0:=x_1,y_1,\dots,x_r,y_r \in \Z^2} \prod_{k=1}^r U^{\beta_{i_k,j_k}}_N\big(b_k-a_k,y_k-x_k\big) \, \\
      \times \, \prod_{m=2}^r \sigma^{i_m,j_m}_N \cdot q^{2}_{a_{m}-b_{\sfp(m)}}\big(x_{m}-y_{\sfp(m)}\big) \, .
    \end{split}
  \end{equation}
  Summing the spatial points on the right hand side of \eqref{rewired_upper_bound_2} we obtain that
  \begin{align} \label{rewired_upper_bound_3}
     & H^{\mathsf{(rew)}}_{r,N,R,M}                                                                 \\
     & \leq \hspace{-0.3cm}\sum_{\big(\{i_1,j_1\},\dots,\{i_r,j_r\}\big) \in  \, \mathcal{I}^{(2)}}
    \,\,\, \sum_{0:=a_1\leq b_1<a_2\leq \dots < a_r\leq b_r \leq N} \prod_{k=1}^r U^{\beta_{i_k,j_k}}_N\big(b_k-a_k\big)
    \prod_{m=2}^r \sigma^{i_m,j_m}_N  q_{2(a_{m}-b_{\sfp(m)})}(0) \, . \notag
  \end{align}
 Using \eqref{graphical-collisions1} one can deduce that
  \begin{equation} \label{rewired_final_upper_bound}
    \ \sum_{r \geq 0}  H^{\mathsf{(rew)}}_{r,N,R,M} \leq \prod_{1\leq i<j \leq h} \E\Big[e^{\frac{\pi\beta_{i,j}}{\log N}\, \sfL_N^{(i,j)}}\Big]= \big(1+o_N(1)\big)\prod_{1\leq i<j \leq h} \frac{1}{1-\beta_{i,j}} \, .
  \end{equation}
  Next, by Proposition \ref{rewired_lower_bound} we have that
  \begin{equation} \label{rewired_lower_bound_2}
    \begin{split}
      &H^{\mathsf{(rew)}}_{r,N,R,M} \\
      &\geq (1-e^{-cR^2})^{2Kh}\, \hspace{-0.1cm} \sum_{\vec{I}=\big(\{i_1,j_1\},\dots,\{i_r,j_r\}\big) \in  \, \mathcal{I}^{(2)}} \, \sumthree{0:=a_1\leq b_1<a_2\leq \dots < a_r\leq b_r \leq N ,}{a_{i+1}-b_i>M(b_i-b_{i-1}), \, 1\leq i \leq r-1,}{0:=x_1,y_1,\dots,x_r,y_r \in \Z^2} \prod_{k=1}^r U^{\beta_{i_k,j_k}}_N\big(b_k-a_k,y_k-x_k\big) \, \\
      &\quad \times \, \prod_{m=2}^r \sigma^{i_m,j_m}_N \cdot q^{2}_{a_{m}-b_{\sfp(m)}}\big(x_{m}-y_{\sfp(m)}\big) \cdot \ind_{\big\{|y_k-x_k|\leq R\sqrt{b_k-a_k}, \, |x_m-y_{\sfp(m)}|\leq R\, C_{K,M}\sqrt{a_m-b_{\sfp(m)}}\big\}},
    \end{split}
  \end{equation}
  Lifting the diffusivity conditions imposed on the right-hand side of \eqref{rewired_lower_bound_2} can be done using arguments already present in Lemma \ref{sdif_lemma}. More specifically,
  we use that for $0 \leq m \leq N$, $w \in \Z^2$ and $1\leq i<j \leq h$,
  \begin{equation*}
    \begin{split}
      \sumtwo{n \in [m,N] ,\,}{ z \in \Z^2: \, |z-w|\leq R\sqrt{n-m}} \hspace{-0.3cm} U^{\beta_{i,j}}_N\big(n-m,z-w\big)& = \sum_{n \in [m,N]} U^{\beta_{i,j}}_N\big(n-m\big) -\hspace{-0.2cm}\sumtwo{n \in [m,N],\,}{ z \in \Z^2:\, |z-w|> R\sqrt{n-m}}  \hspace{-0.3cm}U^{\beta_{i,j}}_N\big(n-m,z-w\big) \\
      & \geq \sum_{n \in [m,N]} U^{\beta_{i,j}}_N\big(n-m\big) -e^{-\kappa R}\sum_{n \in [m,N]} U^{\beta_{i,j}}_N\big(n-m\big) \\
      & \geq \big(1-e^{-\kappa R}\big) \,\sum_{n \in [m,N]} U^{\beta_{i,j}}_N\big(n-m\big) \,  ,
    \end{split}
  \end{equation*}
  where in the first inequality we used \eqref{U_sdif_2} from Lemma \ref{sdif_lemma} with a suitable constant $\kappa(\bar \beta) \in (0,\infty)$.
  Similarly we have that
  \begin{equation*}
    \begin{split}
      \sumtwo{n \in [m,N], \,}{ z \in \Z^2: |z-w|\leq R\,C_{K,M}\,\sqrt{n-m} }\hspace{-0.3cm} q^2_{n-m}(z-w) =&  \sum_{n \in [m,N]} q_{2(n-m)}(0) - \sumtwo{n \in [m,N], \,}{ z \in \Z^2: |z-w|>R\, C_{K,M}\, \sqrt{n-m} } q^2_{n-m}(z-w)\\
      \geq &\big(1-e^{-\kappa\,R^2\, C^2_{K,M}}\big)\sum_{n \in [m,N]} q_{2(n-m)}(0)
    \end{split}
  \end{equation*}
  by tuning the constant $\kappa$ if needed.
  Therefore, we finally obtain that
  \begin{equation} \label{rewired_only_M}
    \begin{split}
      &H^{\mathsf{(rew)}}_{r,N,R,M} \geq (1-e^{-cR^2})^{2Kh}\,(1-e^{-\kappa\,R})^K \, (1-e^{-\kappa\,R^2\,C^2_{K,M}})^K \times \\
      &\,  \sum_{\big(\{i_1,j_1\},\dots,\{i_r,j_r\}\big) \in  \, \mathcal{I}^{(2)}}\,\,\,\sumtwo{0:=a_1\leq b_1<a_2\leq \dots < a_r\leq b_r \leq N ,}{a_{i+1}-b_i>M(b_i-b_{i-1}), \, 1\leq i \leq r-1\, .} \prod_{k=1}^r U^{\beta_{i_k,j_k}}_N\big(b_k-a_k\big) \prod_{m=2}^r \sigma^{i_m,j_m}_N \cdot q_{2(a_{m}-b_{\sfp(m)})}(0) \, .
    \end{split}
  \end{equation}

  The last restriction we need to lift is the restriction $a_{i+1}-b_i>M(b_i-b_{i-1}), \, 1\leq i \leq r-1$. This can be done via the arguments used in Proposition \ref{M_proposition}, so we do not repeat it here, but only note that there exists a constant $\widetilde{C}_{K}=\widetilde{C}_K(\bar\beta,h) \in (0,\infty)$ such that
  for all $0 \leq r \leq K$, the corresponding sum to the right hand side of \eqref{rewired_only_M}, but with its temporal range of summation be such that there exists $1\leq i\leq r-1 :\, a_{i+1}-b_i\leq M(b_i-b_{i-1})$, satisfies the bound
  \begin{equation*}
    \begin{split}
      &\sum_{\big(\{i_1,j_1\},\dots,\{i_r,j_r\}\big) \in  \, \mathcal{I}^{(2)}}\,\,\, \sumtwo{0:=a_1\leq b_1<a_2\leq \dots < a_r\leq b_r \leq N ,}{\exists\, 1\leq i\leq r-1 :\, a_{i+1}-b_i\leq M(b_i-b_{i-1})} \prod_{k=1}^r U^{\beta_{i_k,j_k}}_N\big(b_k-a_k\big) \prod_{m=2}^r \sigma^{i_m,j_m}_N \cdot q_{2(a_{m}-b_{\sfp(m)})}(0) \\
      & \leq \widetilde{C}_K \cdot \epsilon_{N,M} \, ,
    \end{split}
  \end{equation*}
  where $\epsilon_{N,M}$ is such that $\lim_{N \to \infty} \epsilon_{N,M}=0$ for any fixed $M \in (0,\infty)$. Therefore, 
  the resulting lower bound on $H^{\mathsf{(rew)}}_{r,N,R,M}$ will be 
  \begin{equation} \label{rewired_norestrictions}
    \begin{split}
      H^{\mathsf{(rew)}}_{r,N,R,M} &\geq (1-e^{-cR^2})^{2Kh}\,(1-e^{-\kappa\,R})^K \, (1-e^{-\kappa\,R^2\,C^2_{K,M}})^K  \\
      &\times\Bigg(  \sum_{\big(\{i_1,j_1\},\dots,\{i_r,j_r\}\big) \in  \, \mathcal{I}^{(2)}}\,\,\,\sum_{0:=a_1\leq b_1<a_2\leq \dots < a_r\leq b_r \leq N} \prod_{k=1}^r U^{\beta_{i_k,j_k}}_N\big(b_k-a_k\big)\,  \\
      &\,\qquad \qquad \times  \prod_{m=2}^r \sigma^{i_m,j_m}_N \cdot q_{2(a_{m}-b_{\sfp(m)})}(0)   - \widetilde{C}_K \cdot \epsilon_{N,M} \Bigg) \, .
    \end{split}
  \end{equation}
  Note that
  \begin{equation} \label{laplace_parts}
    \begin{split}
      \prod_{1\leq i<j \leq h}\E\Big[e^{\frac{\pi \beta_{i,j}}{\log N}\,  \sfL^{(i,j)}_N}\Big]
      &= \sum_{r=0}^{K} \sum_{\big(\{i_1,j_1\},\dots,\{i_r,j_r\}\big) \in  \, \mathcal{I}^{(2)}}
      \,\, \sum_{0:=a_1\leq b_1<a_2\leq \dots < a_r\leq b_r \leq N} \prod_{k=1}^r U^{\beta_{i_k,j_k}}_N\big(b_k-a_k\big)
      \\
      & \qquad \times \prod_{m=2}^r \sigma^{i_m,j_m}_N \cdot q_{2(a_{m}-b_{\sfp(m)})}(0) + A^{(1)}_{N}+ A_{N,K}^{(2)} \, ,
    \end{split}
  \end{equation}
  where $A_N^{(1)}$ denotes the part of the chaos expansion of $\prod_{1\leq i<j \leq h}\E\Big[e^{\frac{\pi \beta_{i,j}}{\log N}\,  \sfL^{(i,j)}_N}\Big]$
  where there exists a time $n\leq N$ at which multiple pairs collide.
    Moreover,  $A^{(2)}_{N,K}$ 
  denotes the corresponding sum on the right hand side of \eqref{laplace_parts} but from $r=K+1$ to $\infty$, that is
  \begin{equation*}
    \begin{split}
      A^{(2)}_{N,K}= \sum_{r>K} \sum_{\big(\{i_1,j_1\},\dots,\{i_r,j_r\}\big) \in  \, \mathcal{I}^{(2)}}\,\,\sum_{0:=a_1\leq b_1<a_2\leq \dots < a_r\leq b_r \leq N} & \prod_{k=1}^r U^{\beta_{i_k,j_k}}_N\big(b_k-a_k\big) \\
      \times & \prod_{m=2}^r \sigma^{i_m,j_m}_N q_{2(a_{m}-b_{\sfp(m)})}(0) \, .
    \end{split}
  \end{equation*}
  Next, we will give bounds for $A^{(1)}_{N}$ and $A_{N,K}^{(2)}$. Beginning with $A_{N,K}^{(2)}$, let $\rho_K:=\floor[\Big]{\frac{K}{2\binom{h}{2}}} \, $.
  Since we are summing over $r>K$, there has to be a pair $1\leq i<j \leq h$ which has recorded more than $\rho_K$ collisions.
  We recall from \eqref{ren_rep} that $U^{\beta}_N(\cdot)$ admits the renewal representation
  $$
    U_N^{\beta}(n)=\sum_{k \geq 0}(\sigma_N(\beta)\, R_N)^k\, \P\big(\tau^{\ms (N)}_{k}=n\big) \, .
  $$
  There are $\binom{h}{2}$ choices for the pair with more than $\rho_K$ collisions. We can also use the bound \eqref{Uest} to bound the contribution of the rest $\binom{h}{2}-1$ pairs in $A^{(2)}_{N,K}$.
  Therefore, we can write
  \begin{equation*}
    A^{(2)}_{N,K}\leq \frac{\binom{h}{2}}{(1-\bar{\beta}')^{\binom{h}{2}}} \sum_{k>\rho_K} (\sigma_N(\bar \beta)R_N)^k\, \P(\tau^{\ms (N)}_k \leq N) \leq \frac{\binom{h}{2}}{(1-\bar{\beta}')^{\binom{h}{2}}} \sum_{k>\rho_K} (\bar{\beta}')^k \xrightarrow{K \to \infty}0\, ,
  \end{equation*}
  uniformly in $N$, where $\bar{\beta}' \in (\bar\beta,1)$.

  Let us now proceed with estimating $A_N^{(1)}$ and showing that it has negligible contribution. 
  $A_N^{(1)}$ consists of configurations where $\mathfrak{m}\geq 2$ pairs 
 $\{i_1,j_1\}, \{i_2,j_2\},...,\{i_{\mathfrak{m}}, j_{\mathfrak{m}} \}$
   collide at a same time $n\leq N$. Referring to Figure \ref{fig:graphical-collisions}, a case of $\mathfrak{m}=2$ could 
   correspond to a situation when, for example, $a_3=a_4$. Let $n$ be the first time a multiple pair-collision takes place.
   We can choose the pairs $\{i_1,j_1\}, \{i_2,j_2\},...,\{i_{\mathfrak{m}}, j_{\mathfrak{m}} \}$, which collide at that time, in 
   $\binom{h}{2}\cdot\big(\binom{h}{2}-1\big) \cdots \big(\binom{h}{2}-\mathfrak{m}+1\big) / \mathfrak{m}! < \binom{h}{2}^{\mathfrak{m}}$ ways. 
   We can use bound \eqref{Uest} to bound the contribution to $A^{(1)}_{N}$ from the rest of the $\binom{h}{2}-\mathfrak{m}$ pairs by
  c$(1-\bar\beta')^{-{h\choose 2}+\mathfrak{m}}$ and the contribution from collisions involving pairs $\{i_1,j_1\}, \{i_2,j_2\},...,\{i_{\mathfrak{m}}, j_{\mathfrak{m}} \}$
   beyond time $n$ by $c(1-\bar\beta')^{-\mathfrak{m}}$.
     More precisely, we obtain that 
  \begin{equation*}
    \begin{split}
      A_N^{(1)}\leq
       \sum_{\mathfrak{m}=2}^{h \choose 2 } \frac{\binom{h}{2}^\mathfrak{m}}{(1-\bar{\beta}' )^{\binom{h}{2}}} \sum_{n \geq 0, \, x_1,...,x_{\mathfrak{m}} \, \in \Z^2}  
      \, \prod_{i=1}^\mathfrak{m} U^{\bar \beta}_N(n, x_i)
       \leq \sum_{\mathfrak{m}=2}^{h \choose 2 }  \frac{\binom{h}{2}^\mathfrak{m}}{(1-\bar{\beta}' )^{\binom{h}{2}}} 
      \sum_{n \geq 0}\big(U^{\bar \beta}_N\big)^{\mathfrak{m}}(n) \, .
    \end{split}
  \end{equation*}
  Schematically (making reference again to Figure \ref{fig:graphical-collisions}), we have summed out everything except the `wiggle' lines
  which span $[0,n]$ and which correspond to the pairs that simultaneously collide at time $n$ and then summed out the spatial dependence of the latters.
  We will next use Proposition 1.5 of \cite{CSZ19a}, which provides the estimate
  $$ \P(\tau_k^{\ms (N)}=n)\leq  \frac{C\, k\, q_{2n}(0)}{R_N} \leq \frac{C' \, k\, }{n \, \log N} \, , $$ where the second inequality follows by the local limit theorem.
   Therefore, from this, \eqref{ren_rep}  and \eqref{sigmaR} we get that
  $$\big(U^{\bar \beta }_N\big)^\mathfrak{m}(n) =\Big(\sum_{k \, \geq 0} (\sigma_N(\bar\beta)R_N)^{k} \, \P(\tau_k^{\ms (N)}=n) \Big)^{\mathfrak{m}}
  \leq \frac{(C')^{\mathfrak{m}}}{n^{\mathfrak{m}} \, (\log N)^{\mathfrak{m}} } \Big(\sum_{k \geq 0} k \cdot (\bar{\beta}')^k \Big)^{\mathfrak{m}} \, .$$
  for some $\bar{\beta}' \in (\bar\beta,1)$.
  Since $\bar{\beta}'<1$ we have that $\sum_{k \geq 0} k \cdot(\bar{\beta}')^k  < \infty$ and we deduce that there exists a constant $C=C(\bar{\beta}')$ such that
  $\big(U^{\bar \beta}_N\big)^{\mathfrak{m}}(n) \leq \frac{C^{\mathfrak{m}}}{n^{\mathfrak{m}}\,(\log N)^{\mathfrak{m}}}$. 
  Since $\sum_{n \geq 1} \frac{1}{n^{\mathfrak{m}}}<\infty$, for ${\mathfrak{m}}\geq 2$, there exists a (different) constant $C=C(\bar{\beta}')\in (0,\infty)$ such that
  $$ \sum_{n \geq 0} \big(U^{\bar \beta}_N\big)^{\mathfrak{m}}(n)\leq \frac{C}{(\log N)^{\mathfrak{m}}} \xrightarrow{N \to \infty} 0\, .$$
  \vskip 2mm
  The two bounds above, in combination with \eqref{rewired_norestrictions} and \eqref{laplace_parts}, allow us to write:
  \begin{equation*}
    \begin{split}
      \sum_{r =0}^K H^{\mathsf{(rew)}}_{r,N,R,M} & \geq  (1-e^{-cR^2})^{2Kh}\,(1-e^{-\kappa\,R})^K \, (1-e^{-\kappa\,R^2\,C^2_{K,M}})^K \\
      & \qquad \qquad \times  \Bigg(\prod_{1\leq i<j \leq h}\E\Big[e^{\frac{\pi \beta_{i,j}}{\log N}\,  \sfL^{(i,j)}_N}\Big]- K \cdot \widetilde{C}_K\cdot \epsilon_{N,M}- o_N(1)-o_K(1) \Bigg) \, ,
    \end{split}
  \end{equation*}
  which together with upper bound \eqref{rewired_final_upper_bound} entail that
  \begin{equation*}
    \lim_{K \to \infty} \lim_{M \to \infty} \lim_{R \to \infty} \lim_{N \to \infty} \sum_{r=0}^{K} H^{\mathsf{(rew)}}_{r,N,R,M}=\prod_{1\leq i<j \leq h} \frac{1}{1-\beta_{i,j}} \, .
  \end{equation*}
\end{proof}

We are now ready to put all pieces together and prove the main result of the paper, Theorem \ref{main_result}.
\begin{proof}[Proof of Theorem \ref{main_result}]
The joint convergence statement \eqref{thmeq:joint} will follow from the convergence of the joint Laplace transform 
\eqref{thmeq:laplace2} for $|\beta_{i,j}|<1$ and general results on the relation between convergence of Laplace transforms and distributional convergence.
Let us sketch this argument: Since $\big(\tfrac{\pi}{\log N} \sfL_N^{(i,j)} \big)_{1\leq i<j\leq N}$ are non-negative random variables, it suffices by the Cramer-Wold device (see \cite{K97}, Corollary 4.5)
 to establish the distributional convergence of any nonnegative linear combinations, i.e. 
$\sfL_N^{\boldsymbol{\beta}}:=\sum_{i<j} \tfrac{\beta_{i,j} \pi}{\log N} \sfL_N^{(i,j)}$ with $\beta_{i,j}\geq 0$. 
The fact that $\sfL_N^{\boldsymbol{\beta}}$ has 
exponential moments both positive (as follows from our estimates) and negative (as is clear from the nonnegativity of  
$\sfL_N^{\boldsymbol{\beta}}$) implies that the sequence $\sfL_N^{\boldsymbol{\beta}}$  is tight. 
Tightness together with exponential moments imply that subsequential limits of the Laplace transforms exist,
with the convergence being uniform in the vicinity of $0$ in the complex plane. 
This leads to the analyticity of  the limiting Laplace transforms in the  neighbourhood of $0$.
We will establish that the Laplace
transform \eqref{thmeq:laplace2} with positive parameters converge to that of the corresponding linear combination of independent exponential distributions with parameter $1$, which is also analytic in the vicinity of zero (see right-hand side of
 \eqref{thmeq:laplace2}). Since the limiting Laplace transforms agree on an interval with non-empty interior
 (in this case determined by the conditions $\beta_{i,j}\in [0,1)$), by analyticity they have to agree in the 
whole domain of analyticity, which includes an open disc containing $0$.

Therefore, it only remains to establish the convergence of the joint Laplace transform 
\eqref{thmeq:laplace2} for $\beta_{i,j}\in [0,1)$. Let us do so relying on the results we have proven so far.

  Let $\epsilon >0$. There exists large $K=K_\epsilon \in \N$ such that uniformly in $N \in \N$
  \begin{equation} \label{replica_truncation}
    \Big |M^{\boldsymbol{\beta}}_{N,h}- \sum_{r=0}^K H_{r,N}\Big|\leq \epsilon \, ,
  \end{equation}
  by Proposition \ref{finite_replicas}. We have
  \begin{equation*} 
    \begin{split}
      \Big| \sum_{r=0}^K H_{r,N}- \sum_{r=0}^K H^{\mathsf{(rew)}}_{r,N,R,M}\Big | \leq  \bigg(\sum_{r=0}^K (H^{\mathsf{(superdiff)}}_{r,N,R}+H^{\mathsf{(multi)}}_{r,N} )\bigg)+& \Big|\sum_{r=0}^K (H^{\mathsf{(diff)}}_{r,N,R}-H^{\mathsf{(main)}}_{r,N,R,M})\Big|\\
      &+ \Big| \sum_{r=0}^K (H^{\mathsf{(main)}}_{r,N,R,M}-H^{\mathsf{(rew)}}_{r,N,R,M})\Big| \, .
    \end{split}
  \end{equation*}
  By Propositions \ref{fixed_deg_bounds}, \ref{dif_proposition} we have that
  \begin{equation*}
    \lim_{R \to \infty}\lim_{N \to \infty} \bigg(\sum_{r=0}^K (H^{\mathsf{(superdiff)}}_{r,N,R}+H^{\mathsf{(multi)}}_{r,N} )\bigg) \leq \lim_{R \to \infty} \sup_{N \in \N}\sum_{r=0}^K H^{\mathsf{(superdiff)}}_{r,N,R}+\lim_{N \to \infty} \sum_{r=0}^K H^{\mathsf{(multi)}}_{r,N}=0 \, .
  \end{equation*}
  Moreover, by Proposition \ref{M_proposition}
  we have that
  \begin{equation*}
    \lim_{R \to \infty}\lim_{M \to \infty}\lim_{N \to \infty} \Big |  \sum_{r=0}^K H^{\mathsf{(diff)}}_{r,N,R}- \sum_{r=0}^K H^{\mathsf{(main)}}_{r,N,R,M}\Big|=0 \, .
  \end{equation*}
  Last, by Proposition \ref{introduction_of_rewired} we have that
  \begin{equation*}
    \lim_{R \to \infty} \lim_{M \to \infty} \lim_{N \to \infty} \Big |\sum_{r=0}^K H^{\mathsf{(main)}}_{r,N,R,M}- \sum_{r=0}^K H^{\mathsf{(rew)}}_{r,N,R,M}\Big|=0 \, ,
  \end{equation*}
  therefore
  \begin{equation} \label{original_rewired}
    \lim_{R \to \infty} \lim_{M \to \infty} \lim_{N \to \infty} \Big| \sum_{r=0}^K H_{r,N}- \sum_{r=0}^K H^{\mathsf{(rew)}}_{r,N,R,M}\Big |=0 \, .
  \end{equation}
  By Proposition \ref{rewired_limit} we have that
  \begin{equation} \label{rew_conv}
    \lim_{K \to \infty} \lim_{R \to \infty}\lim_{M \to \infty}\lim_{N \to \infty} \sum_{r=0}^K H^{\mathsf{(rew)}}_{r,N,R,M}=\prod_{1\leq i <j \leq h} \frac{1}{1-\beta_{i,j}} \, ,
  \end{equation}
  Therefore, by \eqref{replica_truncation}, \eqref{original_rewired} and \eqref{rew_conv} we obtain that
  \begin{equation*}
    \lim_{N \to \infty} M^{\boldsymbol{\beta}}_{N,h}=\prod_{1\leq i <j \leq h} \frac{1}{1-\beta_{i,j}}\, .
  \end{equation*}
\end{proof}

\noindent \textbf{Acknowledgements.}
We thank Francesco Caravenna and Rongfeng Sun for useful comments on the draft.
D. L. acknowledges financial support from EPRSC through grant EP/HO23364/1 as part of the MASDOC DTC at the University of Warwick.
NZ was supported by EPSRC through grant EP/R024456/1.


\begin{thebibliography}{20}

  \bibitem[CSZ17]{CSZ17}
  F. Caravenna, R. Sun, N. Zygouras. Universality in marginally relevant
  disordered systems. {\em Ann. Appl. Prob.} 27, 3050--3112 (2017).
  
  \bibitem[CSZ18]{CSZ18}
	F. Caravenna, R. Sun, N. Zygouras, The two-dimensional KPZ equation in the
	entire subcritical regime, {\em Ann. Prob.} 48(3), 1086-1127 (2020)

  \bibitem[CSZ19a]{CSZ19a}
  F. Caravenna, R. Sun, N. Zygouras.
  The Dickman subordinator, renewal theorems and disordered systems.
    {\em  Electronic J. Prob.}, 24, (2019).




  \bibitem[CSZ21]{CSZ21}
  F. Caravenna, R. Sun, N. Zygouras.
  The critical $2d$ stochastic heat flow, 
{\em Inventiones Math.}, 1-136, (2023)

 \bibitem[C17]{C17}
F. Comets,
Directed Polymers in Random Environments,
{\em Lecture Notes in Mathematics}, 2175. Springer, Cham, (2017)

  \bibitem[CZ21]{CZ21}
  C. Cosco, O. Zeitouni, Moments of partition functions of $2d$ gaussian polymers in the weak disorder regime - I, arXiv:2112.03767, (2021).

  \bibitem[CZ23]{CZ23}
  C. Cosco, O. Zeitouni, 
Moments of partition functions of 2D Gaussian polymers in the weak disorder regime -- II,
 arXiv:2305.05758

  \bibitem[DFT94]{DFT94}
  G. F. Dell'Antonio, R. Figari, A. Teta.
  Hamiltonians for systems of N particles interacting through point interactions.
    {\em Ann. de l'IHP Physique th\'eorique}, Vol. 60, No. 3, pp. 253-290, (1994).

  \bibitem[DR04]{DR04}
  J. Dimock and S. Rajeev. Multi-particle Schr\"odinger operators with point interactions in the plane. {\em J Phys A:
      Math Gen}, 37(39):9157, (2004).

  \bibitem[ET60]{ET60}
  P. Erd\H{o}s, S.J. Taylor. Some problems concerning the structure of random walk paths. {\em Acta Math. Acad. Sci. Hungar.} 11, 137–162, (1960).

 \bibitem[GS09]{GS09}
J. G\"artner, R. Sun, 
A quenched limit theorem for the local time of random walks on $Z^2$,
{\em Stoch. Processes and their Appl.}, 119(4), 1198-1215, (2009),


  \bibitem[GQT21]{GQT21}
  Y. Gu. J. Quastel, L.C. Tsai,
  Moments of the 2D SHE at criticality.
    {\em Prob. and Math. Physics}, 2(1), 179-219, (2021).


\bibitem[K97]{K97}
O. Kallenberg. Foundations of modern probability. 
{\em Springer}, (1997).

  \bibitem[Kn93]{Kn93}
  F. B. Knight. Some remarks on mutual windings. Séminaire de probabilit\'{e}s de Strasbourg, {\em Springer}, 27, p. 36-43, (1993).

  \bibitem[Kn94]{Kn94}
  F. B. Knight. Erratum to: ``Some remarks on mutual windings'',  
  {\em S\'eminaire de probabilit\'es de Strasbourg,  Springer}, 28, p. 334, (1994).

  \bibitem[LL10]{LL10}G.F. Lawler, V. Limic. Random walk: a modern introduction.
  {\em Cambridge Studies in Advanced Mathematics, 123,  Cambridge University Press},
  (2010).




  \bibitem[LZ21]{LZ21}
  D. Lygkonis, N. Zygouras. Moments of the 2d directed polymer in the subcritical regime and a generalisation of the Erd\H{o}s-Taylor theorem, 
  {\em Comm. Math. Phys.}, 1-38 (2023).

  \bibitem[PY86]{PY86}
  J. Pitman, M. Yor. Asymptotic laws of planar Brownian motion. {\em Ann. Probab.} 14(3): 733-779,  (1986).


  \bibitem[Y91]{Y91}
  M. Yor. \'Etude asymptotique des nombres de tours de plusieurs mouvement browniens complexes corr\'el\'es.  
  {\em Random walks, Brownian motion, and interacting particle systems}, 441- 455, Prog. Probab. 28, {\em Birkhauser}, (1991).
\end{thebibliography}
\end{document}